\documentclass[11pt,reqno]{amsart}
\usepackage{amsmath, amsfonts, amsthm, amssymb,amscd, graphicx, amscd}
\usepackage{float,epsf}
\usepackage[english]{babel}
\usepackage{enumerate}
\usepackage{tikz}
\usepackage{mathrsfs}
\usepackage[numbers,sort&compress]{natbib}

\usepackage{srcltx}
\usepackage{geometry}
\usepackage{verbatim}
\usepackage{mathrsfs}
\usepackage{hyperref}
\usepackage{enumitem} 
\usepackage{bbm} 
\usepackage{stmaryrd}

\usepackage{relsize}
\usepackage{exscale}

\usepackage{mathtools}

\newtheorem{thm}{Theorem}[section]
\newtheorem{defn}[thm]{Definition}
\newtheorem{prop}[thm]{Proposition}
\newtheorem{lem}[thm]{Lemma}

\newtheorem{rem}[thm]{Remark}

\numberwithin{equation}{section}

\def\dd{{\rm d}}
\hypersetup{bookmarksdepth=2}

\begin{document}
\title{FROM MAGNETIZED COULOMBIC QUANTUM DYNAMICS TO MAGNETIZED FLUIDS}
\author{Immanuel Ben Porat}
\address{Immanuel Ben-Porat, Mathematical Universit\"at Basel
 Spiegelgasse 1
 CH-4051 Basel, Switzerland.}
\email{immanuel.ben-porath@unibas.ch}
\begin{abstract}
We extend the quantum modulated energy developed in \cite{golse2022mean} in order to derive the magnetized pressureless Euler-Poisson equation as a semiclassical and mean field semiclassical limit from the magnetized Schr\"odinger-Poisson and von-Neumann equations, respectively. Local well-posedness of the underlying monokinetic PDE is also addressed. In both limits, the magnetic field is external and may be spatially non-uniform. Our results fall in the broader scope of semiclassical and quantum mean field limits for magnetized quantum dynamics.        
\end{abstract}
\maketitle

\section{Introduction}

\subsection{Main objective}
The purpose of this  work is twofold: the first, is to study the semiclassical limit of the $3D$ \textbf{magnetized  Schr\"odinger-Poisson} equation 
\begin{align} \tag{SPA}
i\hbar \partial_{t}\psi_{\hbar}(t,x)=\frac{1}{2}(i\hbar\nabla_{x}+A(x))^{2}\psi_{\hbar}(t,x)+V\ast \left\vert \psi_{\hbar}\right\vert^{2}(t,x)\psi_{\hbar}(t,x),\  \psi_{\hbar}(0,\cdot)=\psi_{\hbar}^{\mathrm{in}}.  \label{SP equation intro}  
\end{align}
Here $A:\mathbb{R}^{3}\rightarrow \mathbb{R}^{3}$ is a divergence free vector field  ($\mathrm{div}_{x}A=0$) representing the magnetic vector potential; the unknown $\psi_{\hbar}\in C([0,T];L^{2}(\mathbb{R}^{3};\mathbb{C})
)$ is a wave function; $\hbar\ll 1$ is the Planck constant; $V$ is the $3D$ repulsive Coulomb interaction given by $V(x)=\frac{1}{4\pi \left\vert x\right\vert}$.
The second objective is to study  the combined semiclassical mean-field limit of the \textbf{$N$-body magnetized von-Neumann} equation 
\begin{align} \tag{v-NA}
i\hbar\partial_{t}R_{\hbar,N}(t)=\left[\mathscr{H}_{\hbar,A}^{N},R_{\hbar,N}(t)\right], \ R_{\hbar,N}(0)=R^{\mathrm{in}}_{\hbar,N}. \label{von Neumman magnetic intro}   
\end{align}
The unknown $R_{\hbar,N}(t)$ is a time dependent symmetric density operator on the Hilbert space $\mathfrak{H}^{\otimes  N}$ where $\mathfrak{H}\coloneqq L^{2}(\mathbb{R}^{3};\mathbb{C})$ (the precise definition of a density operator will be recalled in  \S\ref{relevant back sec}); the operator $\mathscr{H}_{\hbar,A}^{N}$ is the magnetized quantum $N$-body Hamiltonian defined by 
\begin{align*}
\mathscr{H}_{\hbar,A}^{N}\coloneqq \frac{1}{2}\sum_{\ell=1}^{N}(i\hbar\nabla_{x^{\ell}}+A(x^{\ell}))^{2}+\frac{1}{2N}\sum_{\ell\neq m}V(x^{\ell}-x^{m}).    
\end{align*}
In both limit regimes, we aim  to derive a monokinetic PDE, namely the following  \textbf{magnetized Euler-Poisson} equation given by 
\begin{align}\tag{EPA}
\begin{cases}
\begin{array}{lc}
\partial_{t}\rho+\mathrm{div}_{x}(\rho u)=0,\ \rho(0,\cdot)=\rho^{\mathrm{in}}\\
\partial_{t}u+uD_{x}u+u\mathbf{J}+\nabla_{x}V\ast \rho=0, \ u(0,\cdot)=u^{\mathrm{in}}. 
\end{array}\end{cases} \label{Magnetized Euler Poisson velocity density intro}    
\end{align}
The unknown is $(\rho,u)$, where   $\rho$ is a time dependent probability density $\rho(t,\cdot)\in \mathcal{P}(\mathbb{R}^{3})$ and  $u:\mathbb{R}_{+}\times \mathbb{R}^{3}\rightarrow \mathbb{R}^{3}$ is the velocity field; $\mathbf{J}$ is given by means of the anti-symmetrization of the Jacobian of $A$, namely  
\begin{align}
\mathbf{J}\coloneqq D_{x}A-D_{x}^{T}A. \label{def of antisymmetrized gradient}    
\end{align}
The derivation of \eqref{Magnetized Euler Poisson velocity density intro} as a semiclassical/semiclassical mean-field limit from \eqref{SP equation intro} and \eqref{von Neumman magnetic intro} is to be understood in terms of the \textbf{quantum density} and \textbf{quantum current}, to be introduced in the sequel. 
Our final objective would be to prove weak convergence of the quantum density/current to their classical analogues, namely   $\rho$ and $\rho u$, uniformly on any time interval before the first blow-up time of the solution $(\rho,u)$.   
\subsection{Relevant background}\label{relevant back sec} Prior to presenting our main results, we review the relevant literature regarding the semiclassical limit and mean-field semiclassical limit, both in the presence and absence of magnetic effects. We start with the semiclassical limit and assume momentarily that $A=0$, so that \eqref{SP equation intro} simplifies to 
\begin{align}\tag{SP}
i\hbar\partial_{t}\psi_{\hbar}=-\frac{\hbar^{2}}{2}\Delta_{x}\psi_{\hbar}+V\ast \left\vert \psi_{\hbar}\right\vert^{2}\psi_{\hbar}, \ \psi_{\hbar}(0,\cdot)=\psi_{\hbar}^{\mathrm{in}}. \label{Schrodinger Poisson no magnetic}     
\end{align}
It is instructive to point out the analogy between the Schr\"odinger-Poisson equation and the Vlasov-Poisson equation which reads 
\begin{align}\tag{VP}
\partial_{t}f+\left\{H,f\right\}=0, \ \rho_{f}(t,x)=\int_{\mathbb{R}^{3}}f(t,x,\xi)\ \dd \xi. \label{Vlasov Poisson}    
\end{align}
Here, we designate by $H(x,\xi)=\frac{1}{2}\left\vert \xi\right\vert^{2} +V\ast \rho_{f}$ the classical Hamiltonian  and by $\left\{\cdot,\cdot\right\}$ the Poisson brackets defined by 
\begin{align*}
\left\{f,g\right\}\coloneqq\nabla_{\xi}f\cdot \nabla_{x}g-\nabla_{x}f\cdot \nabla_{\xi}g.     
\end{align*}
On the other hand, consider the time-dependent operator $R_{\hbar}(t)=\vert \psi_{\hbar}(t)\rangle  \vert\langle \psi_{\hbar}(t)\vert $. By direct calculation $R_{\hbar}(t)$ is governed by \textbf{Hartree's equation} which reads 
\begin{align}\tag{HE}i\hbar \partial_{t}R_{\hbar}=\left[\mathscr{H}_{\hbar},R_{\hbar}(t)\right], \ R_{\hbar}(0)= \vert \psi_{\hbar}^{\mathrm{in}}\rangle  \vert\langle \psi_{\hbar}^{\mathrm{in}}\vert. \label{Hartree}      
\end{align}
Here  $\mathscr{H}_{\hbar}=-\frac{h^{2}}{2}\Delta_{x} +V\ast\left\vert\psi_{\hbar} \right\vert^{2}$ is the quantum Hamiltonian and the commutator $\left[\cdot,\cdot\right]$ is defined by $\left[T,S\right]\coloneqq TS-ST$ for any two operators $T,S$ on the Hilbert space $\mathfrak{H}$. Hartree's equation reveals a dictionary between classical and quantum systems in the following manner: 
if we replace probability densities on phase space by density operators on a Hilbert space, Poisson brackets by commutators, velocities $\xi$ by the momentum operator $-i\hbar\nabla_{x}$  and classical Hamiltonians by quantum Hamiltonian then we recover \eqref{Hartree} from \eqref{Vlasov Poisson}.\vspace{0.5 cm}

The  passage from \eqref{Schrodinger Poisson no magnetic}  (or more generally \eqref{Hartree}) to \eqref{Vlasov Poisson} in the limit as $\hbar\rightarrow 0$ is referred to as the \textbf{semiclassical limit} and goes back to the work of Lions-Paul \cite{lions1993mesures}, who proved the semiclassical limit in terms of the Wigner transform.  Other foundational works on the derivation of \eqref{Vlasov Poisson} from \eqref{Schrodinger Poisson no magnetic} include \cite{gerard1997homogenization,Athanassoulis, mauser2002semi}. The works \cite{lafleche2019propagation, lafleche2021global,lafleche2023strong, iacobelli2024enhanced, chong20232} demonstrate more modern approaches to the semiclassical limit, including quantum optimal transport as developed in \cite{Caglioti2023TowardsOptimalTransport}. We refer to chapter 3 in \cite{maas2024optimal} and references therein for an exhaustive overview of quantum optimal transport and its applications in semiclassical limits and quantum mean-field limits. The works mentioned so far are mainly focused on the derivation of kinetic dynamics (Vlasov-Poisson) from dispersive dynamics (Schr\"odinger-Poisson). As indicated previously, the semiclassical limit can also be studied in the context of monokinetic equations, and specifically  the Euler-Poisson equation which reads 
\begin{align}\tag{EP}
\begin{cases}
\begin{array}{lc}
\partial_{t}\rho+\mathrm{div}_{x}(\rho u)=0, \ \rho(0,\cdot)=\rho^{\mathrm{in}}\\
\partial_{t}u+uD_{x}u+\nabla_{x}V\ast \rho=0, \ u(0,\cdot)=u^{\mathrm{in}}. 
\end{array}\end{cases} \label{Euler Poisson no magnetic}    
\end{align}
In this case, the limit is to be understood in terms of the quantum density and quantum current, denoted by $\rho_{\hbar}$ and $J_{\hbar}$ respectively, and defined by 
\begin{align}
\rho_{\hbar}(t,x)=\left\vert \psi_{\hbar}\right\vert^{2}(t,x), \ J_{\hbar}(t,x)= \,\hbar\mathrm{Im}\big(\overline{\psi_{\hbar}}(t,x)\nabla_{x}\psi_{\hbar}(t,x)\big).  
\end{align}
Ultimately, we seek to prove $\rho_\hbar(t,\cdot)\underset{\hbar \rightarrow 0}{\rightarrow}\rho(t,\cdot)$ and $J_{\hbar}(t,\cdot)\underset{\hbar\rightarrow 0}{\rightarrow}\rho u(t,\cdot)$ in some appropriate topology.
It should be mentioned that the relation between  \eqref{Euler Poisson no magnetic} and \eqref{Vlasov Poisson} is given via the monokinetic ansatz: 
\begin{align*}
f(t,x,\xi)=\rho(t,x)\otimes \delta(\xi-u(t,x)) \mbox{ is a solution to \eqref{Vlasov Poisson}}\iff (\rho,u) \mbox{ is a solution to \eqref{Euler Poisson no magnetic}}.    
\end{align*}
There have been various works which employed WKB approximations in order to justify the semiclassical limit leading from \eqref{Schrodinger Poisson no magnetic} to \eqref{Euler Poisson no magnetic} - for example   \cite{alazard2006limite,alazard2007semi, LiuTadmor2002}, just to mention a few.  See also   \cite{grenier1998semiclassical} for the semiclassical limit for a Schr\"odinger equation with a local (instead of convolutional) nonlinearity. The WKB method has the limitation that it provides short time semiclassical convergence, meaning that the convergence holds on a time interval independent of $\hbar$ but possibly shorter than the first blow-up time of \eqref{Euler Poisson no magnetic}. A different method to prove the semiclassical limit, which will be central for the present work, is the method of the \textbf{quantum modulated energy}. In \cite{golse2022mean} Golse-Paul introduce a quantum version of the modulated energy and exploit it to prove the semiclassical limit leading from \eqref{Hartree} to \eqref{Euler Poisson no magnetic} on the entire time interval on which the solution is defined. The quantum modulated energy has also the advantage that it is well suited to study hydrodynamical semiclassical limit regimes, see \cite{rosenzweig2021quantum, ben2026quantum}, and many body quantum dynamics and their joint mean-field semiclassical limits. 
\par\medskip
We proceed by elaborating on the combined semiclassical mean field limit. First we introduce some additional terminology.  By a \textbf{density operator}, 
we mean a bounded operator $R$ on a Hilbert space such that $R$ is self-adjoint non-negative ($R=R^{*}\geq 0$) 
and $\mathrm{tr}(R)=1$. By a \textbf{symmetric operator}  on $\mathfrak{H}^{\otimes N}$ 
we mean an operator $R_{N}$ such that, for all permutations $\sigma \in \mathfrak{S}_{N}$ 
($\mathfrak{S}_{N}$ is the symmetric group on $N$ elements), 
\begin{align*}
U_{\sigma}R_{N}U_{\sigma}^{\ast}=R_{N},
\end{align*}
where  $U_{\sigma}$ is the operator on $\mathfrak{H}^{\otimes N}$ defined by
\begin{align*}
(U_{\sigma}\Psi_{N})(x_{1},\cdots,x_{N})\coloneqq\Psi_{N}(x_{\sigma^{-1}(1)}, \cdots, x_{\sigma^{-1}(N)}) 
\qquad \text{ for any } \Psi_{N}\in \mathfrak{H}^{\otimes N}.
\end{align*}
We denote the class of density operators on $\mathfrak{H}$ by $\mathcal{D}(\mathfrak{H})$ and the class of symmetric density operators on $\mathfrak{H}^{\otimes N}$ by $\mathcal{D}_{s}(\mathfrak{H}^{\otimes N})$. 
If $R_{N}\in \mathcal{D}_{s}(\mathfrak{H}^{\otimes N})$ with integral kernel $k(X_{N},Y_{N})$, 
then  $\rho_{N}\in L^{1}(\mathbb{R}^{3N})$ is  the function 
defined by $\rho_{N}(X_{N})\coloneqq k(X_{N},X_{N})$. 
It is well known that $\rho_{N}$ is a symmetric probability density on $\mathbb{R}^{3N}$
 and is called the \textbf{density} of $R_{N}$ (see the footnote pp. 61--62 in \cite{golse2017schrodinger} for more details). With these definitions we can extend the notion of quantum densities and quantum currents to $N$-body dynamics. Recall first the definition of a marginal.
\begin{defn}
Let $R_{N}\in \mathcal{D}_{s}(\mathfrak{H}^{\otimes N})$. 
For each $1\leq k\leq N,$ define \textit{the $k$-th marginal of $R_{N}$}, as the unique element $R_{N:k}\in \mathcal{D}_{s}(\mathfrak{H}^{\otimes k})$, such that
\begin{align*}
\mathrm{tr}_{\mathfrak{H}^{\otimes k}}(A_{k}R_{N:k})=\mathrm{tr}_{\mathfrak{H}^{\otimes N}}\big((A_{k}\otimes I^{\otimes(N-k)})R_{N}\big)
\end{align*}
for all bounded operators $A_{k}$ on $\mathfrak{H}^{\otimes k}$. \label{def of marginal}
\label{def of current}
\end{defn}
The quantum current of $R_{N}$, denoted by $J_{\hbar,N:1}$, is the unique signed vector-valued Radon measure on $\mathbb{R}^{3}$ such that, for all $a \in W^{1,\infty}(\mathbb{R}^{3};\mathbb{R}^{3})$ it holds that 
 \begin{align}
\int_{\mathbb{R}^{3}}a(x)\cdot J_{\hbar,N:1}(\dd x)=-\frac{1}{2}\mathrm{tr}\big(( i\hbar\nabla_{x}\vee a)R_{N:1}\big). \label{non mag current def}
 \end{align}
 Here $\vee$ designates the anti-commutator defined by $T\vee S\coloneqq TS+ST$. The quantum density of $R_{N}$, denoted by $\rho_{N:1}$, is the density of the operator $R_{N:1}$. In some precise sense \eqref{Hartree} is a limiting case of von-Neumman's equation (see \cite{ErdosYau2001,rodnianski2009quantum}), and therefore the joint semiclassical  mean-field limit is more complicated than the semiclassical limit. Still working under the simplification $A\equiv 0$ the equation \eqref{von Neumman magnetic intro} becomes 
 \begin{align}\tag{v-N}
 i\hbar\partial_{t}R_{\hbar,N}(t)=\left[-\frac{\hbar^{2}}{2}\sum_{\ell=1}^{N}\Delta_{x^{\ell}}+\frac{1}{2N}\sum_{\ell\neq m}V(x^{\ell}-x^{m}),R_{\hbar,N}(t)\right], \ R_{\hbar,N}(0)=R_{\hbar,N}^{\mathrm{in}}. \label{von neumman non magnetic} \end{align}
 Recall that the unknown $R_{\hbar,N}(t)$ in \eqref{von neumman non magnetic} is a time-dependent symmetric density operator on $\mathfrak{H}^{\otimes N}$. Some of the earliest works handling the passage from \eqref{von neumman non magnetic} to \eqref{Vlasov Poisson} in the limit $\hbar+\frac{1}{N}\rightarrow 0$ are \cite{spohn1981vlasov,narnhofer1981vlasov}—in both works the interaction $V$ is required to enjoy enough regularity. The theory of quantum optimal transport has also been applied in order to justify this passage in \cite{golse2017schrodinger, golse2016mean}, but again subject to regularity assumptions on $V$ which exclude Coulomb interactions. The mean-field semiclassical limit leading from \eqref{von neumman non magnetic} to \eqref{Vlasov Poisson} for interactions with Coulomb singularity remains open. However, the 1D screened Coulomb case has been solved in \cite{chen2024global} and there have been some promising partial results in higher dimensions demonstrated in \cite{chong2024many, chen2022convergence}. The monokinetic version of this problem, i.e. the derivation of \eqref{Euler Poisson no magnetic} from \eqref{von neumman non magnetic}, has been fully resolved in \cite{golse2022mean}, by applying the breakthrough of Serfaty \cite{Serfaty2020MeanField} which offered the first full characterization of the mean-field limit for Coulomb flows.
 
 \par\medskip
To conclude the discussion on the relevant literature, we review the state of the art of the limit regimes discussed above in the presence of magnetic fields. The semiclassical limit for the magnetized Liouville equation, which is a linear version of the magnetized Vlasov equation, has been studied in \cite{arnold1989electromagnetic} using a Wigner transform approach. The semiclassical limit for the magnetized Liouville equation has also been the main object of interest of \cite{ben2024magnetic}, in which the author employs a quantum optimal transport approach, which yields explicit convergence rates. In \cite{moller2025pauli} the semiclassical limit of the the Pauli-Poisson equation is proved by adapting the Lions-Paul approach mentioned earlier: however the limit is proved along a subsequence and is valid subject to decay assumptions on $A$ which exclude constant magnetic fields. All of the latter works consider external magnetic vector fields. The case of self-consistent magnetic fields has also received attention: in \cite{leopold2026derivation} the authors introduce a regularized version of the Schr\"odinger-Maxwell equation and prove semiclassical convergence to the associated Vlasov-Maxwell equation. In \cite{mauser2023semiclassical} the authors apply WKB methods in order to semiclassically derive the Euler-Poisswell equation, which is a version of \eqref{Euler Poisson no magnetic} accounting for self consistent magnetic fields.   We also mention the related semiclassical limit of the magnetic Dirac equation, which gives rise to relativistic dispersive dynamics, and has been investigated in \cite{moller2025pauli}. We are not aware of works about the combined semiclassical mean-field limit in the presence of magnetic effects, although we mention \cite{luhrmann2012mean} which provides a non-uniform in $\hbar$ derivation of the magnetized Hartree equation as a cutoff mean-field limit.      
\subsection{Statement of main results.} In this work we adapt the quantum modulated energy method in order to prove the semiclassical limit and mean-field semiclassical limit leading from \eqref{SP equation intro} and \eqref{von Neumman magnetic intro}, respectively, to \eqref{Magnetized Euler Poisson velocity density intro}. Before stating our main theorems, we introduce the following quantum version of the modulated energy, denoted by $\mathcal{E}_{\hbar}(t)$. Given a solution $(\rho,u)$ and a solution $\psi_{\hbar}$ to \eqref{SP equation intro} to \eqref{Magnetized Euler Poisson velocity density intro} define  
\begin{align}
\mathcal{E}_{\hbar}[\rho,u](t)\coloneqq \int_{\mathbb{R}^{3}}\left\vert (i\hbar\nabla_{x}+A+u)\psi_{\hbar}\right\vert^{2}(t,x)\ \dd x +\int_{\mathbb{R}^{3}}V\ast(\rho_{\hbar}-\rho)(t,x)(\rho_{\hbar}-\rho)(t,x)\ \dd x.   \label{quantum modulated energy}
\end{align}
The quantity \eqref{quantum modulated energy} is a magnetized version of the quantum modulated energy introduced in  \cite{golse2022mean}. Note that taking $\rho\equiv 0$ and $u\equiv 0$ in \eqref{quantum modulated energy} we recover the \textbf{total energy} of the system given by 
\begin{align*}
\mathcal{F}_{\hbar}(t)=\int_{\mathbb{R}^{3}}\left\vert (i\hbar\nabla_{x}+A)\psi_{\hbar}\right\vert^{2}(t,x)\ \dd x +\int_{\mathbb{R}^{3}}V\ast \rho_{\hbar}(t,x)\rho_{\hbar}(t,x)\ \dd x.   \end{align*}  By a simple calculation we have conservation of total energy, that is  
\begin{align}
\frac{\dd}{\dd t}\mathcal{F}_{\hbar}(t)=0. \label{Conservation of energy}    
\end{align}
For the $N$-body problem we need to introduce a \textbf{re-normalized} quantum modulated energy, denoted by $\mathcal{E}_{\hbar,N}(t)$ and defined by 
\begin{align}
\mathcal{E}_{\hbar,N}[\rho,u](t)&= \frac{1}{N}\sum^N_{\ell=1}\mathrm{tr}\big((i\hbar\nabla_{x
^{\ell}}+A(x^{\ell})+u(t,x^{\ell})
)^{2}R_{\hbar,N}(t)\big) \notag\\
&+\int_{\mathbb{R}^{3N}}\left(\int_{\Delta^{c}}V(x-y)(\mu_{X^{N}}-\rho)^{\otimes 2}(\dd x\dd y)\right)\rho_{\hbar,N}(t,\dd X^{N})+\frac{C}{N^{\frac{2}{3}}}. \label{Def of renormalized modulated energy}    
\end{align}
In the above definition we invoked the following notation: 
\begin{itemize}
\item $\Delta^{c}$ is the complement of the diagonal, i.e. $\Delta^{c}\coloneqq\{(x,y)\in \mathbb{R}^{3}\times \mathbb{R}^{3}\vert x\neq y\}$. More generally, define $\Delta_{N}^{c}\coloneqq\{(x^{1},\dots,x^{N})\in \mathbb{R}^{3N}\vert \forall i\neq j:x_{i}\neq x_{j}\}$.
\item $\mu_{X^{N}}$ is the empirical measure centered at a configuration $X^{N}\in \Delta_{N}^{c}$, i.e. $$\mu_{X^{N}}\coloneqq \frac{1}{N}\sum_{i=1}^{N}\delta_{x_{i}}$$. 
\item The constant $C>0$ is given by point i. in Lemma \ref{reminder of positivity}. The reason for the inclusion of this constant is to ensure that the quantity $\mathcal{E}_{\hbar,N}(t)$ is non-negative.
\item $\rho_{\hbar,N}(t,\cdot)\in \mathcal{P}(\mathbb{R}
^{3N})\cap L^{1}(\mathbb{R}^{3N})$ is the density of $R_{\hbar,N}(t)$. 
\end{itemize}
The basic idea underpinning the proof of the semiclassical limit and semiclassical mean-field limit is to study the evolution of $\mathcal{E}_{\hbar}(t)$ and $\mathcal{E}_{\hbar,N}(t)$ with the aim of obtaining a Gr\"onwall estimate on these quantities. This would lead to the vanishing of $\underset{t\in [0,T]}{\sup}\mathcal{E}_{\hbar}(t)$ and $\underset{t\in [0,T]}{\sup}\mathcal{E}_{\hbar,N}(t)$   as $\hbar\rightarrow 0$ and $\hbar+\frac{1}{N}\rightarrow 0$, respectively. In what follows we impose the following technical assumption on $A$, borrowed from \cite{luhrmann2012mean}: 

\begin{itemize}
    \item[\textbf{A1}] $A\in C^{\infty}(\mathbb{R}^{3};\mathbb{R}^{3})$. 
    \item[\textbf{A2}] There is some $\varepsilon>0$ such that for all $\alpha\in \mathbb{N}^{3}$ it holds that 
    \begin{align*}
     \left\vert \partial^{\alpha}\mathrm{curl}_{x}A(x)\right\vert\leq C_{\alpha}(1+\left\vert x\right\vert)^{-(1+\varepsilon)}\ \mbox{and}\ \vert \partial_{\alpha}A\vert\leq C_{\alpha}.     \end{align*}
\end{itemize}  
Note that in all the results we prove we allow the magnetic field to be spatially non-uniform. At the same time, any linear magnetic vector potential (which corresponds to a uniform magnetic field) verifies assumption \textbf{A1}-\textbf{A2}.  Our first theorem proves the derivation of \eqref{Magnetized Euler Poisson velocity density intro} from \eqref{SP equation intro} in the limit as $\hbar \rightarrow 0$ and is stated below. We denote by $H^{k}_{A}(\mathbb{R}^{3})$ the magnetic Sobolev space of order $k$ (see \S\ref{Magnetized SP AND VP SEC} for the definition). 
\begin{thm}\label{main thm semi classical}
Suppose that: 
\begin{itemize}
    \item $A:\mathbb{R}^{3}\rightarrow \mathbb{R}^{3}$ is a given vector field satisfying \textup{\textbf{A1}-\textbf{A2}}. 
    \item $\psi^{\mathrm{in}}_{\hbar}\in H^{2}_{A}(\mathbb{R}^{3})$ is such that $\int_{\mathbb{R}^{3}}\left\vert \psi_{\hbar}^{\mathrm{in}}\right\vert^{2}(x)\ \dd x=1$.   
    \item $(\rho^{\mathrm{in}},u^{\mathrm{in}})\in (H^{3}(\mathbb{R}^{3})\cap \mathcal{P}(\mathbb{R}^
    {3}))\times H^{4}(\mathbb{R}^{3})$.
\end{itemize}
Let $(\rho,u)$ be the unique solution to \eqref{Magnetized Euler Poisson velocity density intro} with initial data $(\rho^{\mathrm{in}},u^{\mathrm{in}})$ on $[0,T]$ (as guaranteed by Theorem \ref{well posedness of EPA}). Let $\psi_{\hbar}$ be the unique solution of \eqref{SP equation intro} with initial data $\psi^{\mathrm{in}}_{\hbar}$ (as guaranteed by Theorem \ref{Magnetized SP well posed}). Set 
\begin{align}
J_{\hbar}\coloneqq\,\hbar\mathrm{Im}\big(\overline{\psi_{\hbar}}(t,x)\nabla_{x}\psi_{\hbar}(t,x)\big)-A\left\vert \psi_{\hbar}\right\vert^{2}(t,x) \ \mbox{and} \ \rho_{\hbar}(t,x)\coloneqq\left\vert \psi_{\hbar}\right\vert^{2}(t,x).
\label{single body current mag}      
\end{align}
Then, it holds that 
\begin{align}
\underset{t\in [0,T]}{\sup}\mathcal{E}_{\hbar}(t)\underset{\hbar \rightarrow 0}{\rightarrow}0 \ \mbox{provided}\ \mathcal{E}_{\hbar}(0)\underset{\hbar \rightarrow 0}{\rightarrow}0. \label{vanishing of quantum modulated energy cocnlusion}    
\end{align}
Consequently, it holds that 
\begin{align}
\underset{t\in [0,T]}{\sup}\left\Vert (\rho_{\hbar}-\rho)(t,\cdot)\right\Vert_{\dot H^{-1}(\mathbb{R}^{3})}\underset{\hbar \rightarrow 0}{\rightarrow}0 \ \mbox{and}\ \underset{t\in [0,T]}{\sup}\left\Vert (J_{\hbar}-\rho u)(t,\cdot)\right\Vert_{W^{-1,\infty}(\mathbb{R}^{3};\mathbb{R}^{3})}\underset{\hbar \rightarrow 0}{\rightarrow}0. \label{convergence conclusion 1d}       
\end{align}
\end{thm}
Our second main theorem proves the derivation of \eqref{Magnetized Euler Poisson velocity density intro} from \eqref{von Neumman magnetic intro}. To the best of our knowledge, this is the first mean-field limit result for quantum many body dynamics which are simultaneously singular and magnetized . As has already been clarified in \cite{golse2022mean}, this limit is more complicated in comparison to the single body problem, as it necessitates the use of the commutator estimates discovered in \cite{Serfaty2020MeanField}. The definition \eqref{non mag current def} readily extends to the magnetized settings: we denote by  $J_{\hbar,N:1}$ the unique signed vector-valued Radon measure on $\mathbb{R}^{3}$ such that for all $a \in W^{1,\infty}(\mathbb{R}^{3};\mathbb{R}^{3})$ it holds that 
 \begin{align}
\int_{\mathbb{R}^{3}}a(x)\cdot J_{\hbar,N:1}(\dd x)=-\frac{1}{2}\mathrm{tr}\big( (i\hbar\nabla_{x}+A)\vee aR_{N:1}\big). \label{defintion of N body current}
 \end{align}
 Setting 
\begin{align}
\mathscr{K}_{\hbar}^{N}\coloneqq \frac{1}{2}\sum_{\ell=1}^{N}(i\hbar\nabla_{x^{\ell}}+A(x^{\ell}))^{2} \label{def of Nkinetic part}   
\end{align} and 
\begin{align}
\mathscr{V}^{N}\coloneqq \frac{1}{2N}\sum_{\ell\neq m}V(x^{\ell}-x^{m}) \label{def of Ninter part}  
\end{align} we may write concisely  
\begin{align*}
\mathscr{H}_{\hbar,A}^{N}=\mathscr{K}_{\hbar}^{N}+\mathscr{V}^{N}.     
\end{align*}
\begin{thm}\label{N body mainthm}
Suppose that: 
\begin{itemize}
    \item $A:\mathbb{R}^{3}\rightarrow \mathbb{R}^{3}$ is a given vector field satisfying assumptions \textup{\textbf{A1}-\textbf{A2}}. 
    \item $R^{\mathrm{in}}_{\hbar, N}\in \mathcal{D}_{s}(\mathfrak{H}^{\otimes N})$ is such that 
    \begin{align*}
     \mathrm{tr}\left(\left(I+\mathscr{K}_{\hbar}^{N}\right)^{2}R_{\hbar,N}^{\mathrm{in}}\right)<\infty.    
    \end{align*}
    \item $(\rho^{\mathrm{in}},u^{\mathrm{in}})\in (H^{3}(\mathbb{R}^{3})\cap \mathcal{P}(\mathbb{R}^{3}))\times H^{4}(\mathbb{R}^{3})$.
\end{itemize}
Let $(\rho,u)$ be the unique solution to \eqref{Magnetized Euler Poisson velocity density intro} with initial data $(\rho^{\mathrm{in}},u^{\mathrm{in}})$ on $[0,T]$ (as guaranteed by Theorem \ref{well posedness of EPA}). Let $R_{\hbar,N}(t)$ be the unique solution to \eqref{von Neumman magnetic intro} with initial data $R^{\mathrm{in}}_{\hbar,N}$ (as guaranteed by Theorem \ref{well posedness of magnetic von Neumann}). Then, it holds that 
\begin{align}
\underset{t\in [0,T]}{\sup}\mathcal{E}_{\hbar,N}(t)\underset{\hbar \rightarrow 0}{\rightarrow}0 \ \mbox{provided}\ \mathcal{E}_{\hbar,N}(0)\underset{\hbar \rightarrow 0}{\rightarrow}0. \label{N body N body energy convergence}   \end{align}
Consequently, it holds that 
\begin{align}
\underset{t\in [0,T]}{\sup}\left\Vert (\rho_{\hbar,N:1}-\rho)(t,\cdot)\right\Vert_{H^{-1}(\mathbb{R}^{3})}\underset{\hbar +\frac{1}{N}\rightarrow 0}{\rightarrow}0 \ \mbox{and}\ \underset{t\in [0,T]}{\sup}\left\Vert (J_{\hbar,N:1}-\rho u)(t,\cdot)\right\Vert_{W^{-1,\infty}(\mathbb{R}^{3};\mathbb{R}^{3})}\underset{\hbar+\frac{1}{N}\rightarrow 0}{\rightarrow}0.     \label{N body final convergence}  
\end{align}
\end{thm}
The inclusion of a magnetic field necessitates a modification of the argument in \cite{golse2022mean} from several viewpoints: first, we need to explain how to formally obtain \eqref{Magnetized Euler Poisson velocity density intro} from the magnetized Vlasov-Poisson equation. Next, the local well posedness theory of \eqref{Magnetized Euler Poisson velocity density intro} needs to be adapted to the magnetized settings—this will be the content of section \S\ref{sec formal derivation}. The (global) well posedness of \eqref{SP equation intro} and \eqref{von Neumman magnetic intro} also requires adjustments, and has been  treated to a large extent in \cite{luhrmann2012mean}—we will review and complement the relevant results in \S\ref{Magnetized SP AND VP SEC} (see also \cite{cazenave2003semilinear, yajima1991schrodinger} for related results on the well-posedness of magnetized Schr\"odinger equations). Perhaps most importantly, one has to correctly identify the cancellations of the magnetic terms in the calculation of the time derivative of $\mathcal{E}_{\hbar}(t)$ and $\mathcal{E}_{\hbar,N}(t)$. These calculations will be performed in \S\ref{semiclassical limit section} and \S\ref{Mean-field semiclassical limit sec}. Finally, admissible initial data, i.e. wave functions $\psi^{\mathrm{in}}_{\hbar}$ and density operators $R^{\mathrm{in}}_{\hbar,N}$ witnessing the convergence $\mathcal{E}_{\hbar}(0)\underset{\hbar\rightarrow 0}{\rightarrow} 0$ and $\mathcal{E}_{\hbar,N}(0)\underset{\hbar+\frac{1}{N}\rightarrow 0}{\rightarrow} 0$, will be constructed in \S\ref{initial data sec}. From a more conceptual point of view, we are hopeful that the derivations offered in the present work may pave the path to more ambitious problems on magnetized mean-field semiclassical limits, most notably the derivation of monokinetic PDEs with self-consistent magnetic fields from appropriate quantum many-body dynamics.
\section{The magnetized Schr\"odinger-Poisson equation and magnetized von-Neumann equation}\label{Magnetized SP AND VP SEC}
We start by recalling the global well-posedness of \eqref{SP equation intro}. Denote $\Pi_{j}\coloneqq i\hbar\partial_{x_{j}}+A_{j}$ and define the $k$-th order magnetic Sobolev space ($k\in \mathbb{N}$) by 
\begin{align*}
H^{k}_{A,\hbar}(\mathbb{R}^{d})\coloneqq \left\{\psi\in L^{2}(\mathbb{R}^{d};\mathbb{C})\vert \left\Vert \Pi_{1}^{\alpha_{1}}\dots \Pi_{d}^{\alpha_{d}}\psi\right\Vert_{2}<\infty \ \mbox{for all}\ \alpha \in \mathbb{N}_{0}^{d}\ \mbox{with}\ \left\vert \alpha\right\vert\leq k \right\}.    
\end{align*}
We equip $H^{k}_{A,\hbar}(\mathbb{R}^{d})$ with the norm  
\begin{align*}
\left\Vert \psi\right\Vert_{H^{k}_{A,\hbar}}^{2}\coloneqq \sum_{\left\vert \alpha\right\vert\leq k }\left\Vert \Pi_{1}^{\alpha_{1}}\dots \Pi_{d}^{\alpha_{d}}\right\Vert_{2}^{2}.       
\end{align*}
\begin{thm}\textup{(Proposition 2.1, \cite{luhrmann2012mean})}\label{Magnetized SP well posed}
Suppose that: 
\begin{itemize}
    \item $A$ satisfies assumptions \textup{\textbf{A1}-\textbf{A2}}.  
    \item $V\in L^{\frac{3}{2}}(\mathbb{R}^{3})+L^{\infty}(\mathbb{R}^{3})$ is real valued and even. 
    \item $\psi^{\mathrm{in}}_{\hbar}\in H^{2}_{A,\hbar}(\mathbb{R}^{3})$ is such that $\int_{\mathbb{R}^{3}}\left\vert \psi^{\mathrm{in}}_{\hbar}\right\vert^{2}\ \dd x=1$.
\end{itemize}
Then, there exist a unique solution $\psi_{\hbar}\in C(\mathbb{R}_{+};H^{2}_{A,\hbar}(\mathbb{R}^{3}))\cap C^{1}(\mathbb{R}_{+};H^{-1}(\mathbb{R}^{3}))$ to \eqref{SP equation intro} with initial data $\psi^{\mathrm{in}}_{\hbar}$. Moreover, conservation of energy holds, i.e. 
\begin{align*}
\mathcal{F}_{\hbar}(t)=\mathcal{F}_{\hbar}(0) \ \mbox{for all}\ t\in \mathbb{R}_{+}.    
\end{align*}
\end{thm}
\begin{rem}
It is worth mentioning that Theorem \ref{Magnetized SP well posed} was stated and proved in \cite{luhrmann2012mean} specifically for $\hbar=1$. However, note that by direct calculation $\psi$ is a solution to $i\partial_{t}\psi=\frac{1}{2}(i\nabla_{x}+A)^{2}\psi+V\ast \left\vert \psi\right\vert^{2}\psi$ iff $\psi_{\hbar}=\psi(\hbar^{-2}t,\hbar^{-1}x)$ is a solution to \eqref{SP equation intro}. Therefore, the global well-posedness  of \eqref{SP equation intro} for arbitrary $\hbar$ follows immediately from the case $\hbar=1$, and vice versa. Note moreover that the magnetic Sobolev space $H^{k}_{A,\hbar}(\mathbb{R}^{d})$  does not coincide with the usual Sobolev space $H^{k}(\mathbb{R}^{d})$, unless one assumes further that $A$ is bounded. We avoid  imposing boundedness on $A$ in order to be able to cover the case of constant magnetic fields. This is the reason why we introduce the spaces $H^{k}_{A,\hbar}(\mathbb{R}^{d})$.        
\end{rem}
The well-posedness of the abstract ODE \eqref{von Neumman magnetic intro} is based on Kato's perturbation theory, recalled for convenience: 
\begin{thm}[\cite{teschl2014mathematical}, Theorem 6.4]
\label{Kato's theorem}
Let $T,\mathfrak{D}(T)\subset\mathfrak{H}$ be an {\rm(}essentially{\rm)}
self-adjoint operator and $S,\mathfrak{D}(S)\subset\mathfrak{H}$ a symmetric
operator such that $\mathfrak{D}(T)\subset \mathfrak{D}(S)$. 
Suppose that there exist $0<a<1$ and $b>0$
such that, for each $\varphi\in \mathfrak{D}(T),$
\begin{equation}
\|S\varphi\|^{2}\leq a\|T\varphi\|^{2}+b\|\varphi\|^{2}.\label{eq:-25-1}
\end{equation}
Then $T+S$  is {\rm(}essentially{\rm)} self-adjoint and $\mathfrak{D}(T+S)=\mathfrak{D}(T)$. In the case
when $T$ is essentially self-adjoint, $\mathfrak{D}(\overline{T})\subset \mathfrak{D}(\overline{S})$
and $\overline{T+S}=\overline{T}+\overline{S}$, 
where $\overline{T}$ and $\overline{S}$ stand for the closures of $T$ and $S$ respectively.
\end{thm}
\begin{thm}\label{well posedness of magnetic von Neumann}
Suppose that: 
\begin{itemize}
\item $A$ satisfies assumptions \textup{\textbf{A1}-\textbf{A2}}. 
\item $R^{\mathrm{in}}_{\hbar,N}\in \mathcal{D}_{s}(\mathfrak{H}^{\otimes N})$ is such that 
\begin{align*}
\mathrm{tr}\left((I+\mathscr{K}_{\hbar}^{N})^{2}R^{\mathrm{in}}_{\hbar,N}\right)<\infty.     
\end{align*}
\end{itemize}
Then there exist a unique solution $R_{\hbar,N}(t)\in \mathcal{D}_{s}(\mathfrak{H}^{\otimes N})$ to \eqref{von Neumman magnetic intro} such that 
\begin{align}
\underset{t\in \mathbb{R}_{+}}{\sup}\mathrm{tr}\left(\mathscr{H}_{\hbar,A}^{N}R_{\hbar,N}^{\mathrm{in}}(t)\right)<\infty. \label{trace of operator against Hamiltonian} 
\end{align}
\end{thm} 
\begin{proof}
By 2.2 in \cite{helffer2006introduction} the operator $\mathscr{K}_{\hbar}^{N}$ is self-adjoint with domain 
\begin{align*}
\mathfrak{D}(\mathscr{K}_{\hbar}^{N})=\left\{\psi_{N}\in H^{1}_{A,\hbar}(\mathbb{R}^{3N})\vert \mathscr{K}_{\hbar}^{N}\psi_{N}\in L^{2}(\mathbb{R}^{3N})\right\}.    
\end{align*}
Denote by $\Delta_{A,x^{\ell}}$ the single body magnetic Laplacian defined by    $\Delta_{A,x^{\ell}}\coloneqq (i\hbar \nabla_{x^{\ell}}+A(x^{\ell}))^{2}$. 
Fix $1\leq i,j\leq N$ with $i<j$ and consider $V_{ij}\coloneqq V(x_{i}-x_{j})$, viewed as a multiplication operator on $L^{2}(\mathbb{R}^{3N})$. 
Given $\psi_{N}\in \mathfrak{D}(\mathscr{K}_{\hbar}^{N})$ we have 
\begin{align}
&\int_{\mathbb{R}^{3N}}\left\vert V_{ij}\psi_{N}\right\vert^{2}(x^{1},\dots,x^{N})\ \dd X^{N} =\int_{\mathbb{R}^{3N}}\left\vert V\right\vert^{2} (x_{i}-x_{j})\left\vert \psi_{N}\right\vert^{2}(x^{1},\dots,x^{N})\ \dd X^{N} \notag\\
&=\int_{\mathbb{R}^{3N}}\left\vert V\right\vert^{2}(x_{i})\left\vert \psi_{N}\right\vert^{2}(x^{1},\dots x^{i}+x^{j},\dots,x^{j},\dots,x^{N})\ \dd X^{N} \notag\\
&=\frac{1}{16\pi^{2}}\int_{\mathbb{R}^{3(N-1)}}\left(\int_{\mathbb{R}^{3}}\frac{\left\vert \psi_{N}\right\vert^{2}(x^{1},\dots,x^{i}+x^{j},\dots,x^{j},\dots,x^{N})}{\left\vert x^{i}\right\vert^{2}}\ \dd x^{i}\right)\dd x^{1}\dots \dd \widehat{x}^{i}\dots \dd x^{N}.  \label{Vij squared}
\end{align}
Recall that by the diamagnetic inequality (see for instance \cite{lieb2001analysis}) we have the pointwise inequality 
\begin{align}
\left\vert \nabla_{x} \left\vert \psi\right\vert \right\vert \leq \frac{1}{\hbar}\left\vert (i\hbar\nabla_{x}+A)\psi\right\vert.  \label{diamagnetic inequality}   
\end{align}
Applying the 3D Hardy inequality and the diamagnetic inequality \eqref{diamagnetic inequality} for the inner integral in the right-hand side of \eqref{Vij squared} we get 
\begin{align}
\int_{\mathbb{R}^{3}}\frac{\left\vert \psi_{N}\right\vert^{2}(x^{1},\dots,x^{i}+x^{j},\dots,x^{N})}{\left\vert x^{i}\right\vert^{2}}\ \dd x^{i}&\leq 4\int_{\mathbb{R}^{3}}\left\vert \nabla_{x^{i}}\left\vert \psi_{N}\right\vert\right\vert^{2}(x^{1},\dots,x^{i}+x^{j},\dots,x^{N})\ \dd x^{i} \notag\\
&=4\int_{\mathbb{R}^{3}}\left\vert \nabla_{x^{i}}\left\vert \psi_{N}\right\vert\right\vert^{2}(x^{1},\dots,x^{i},\dots,x^{N})\ \dd x^{i} \notag\\
&\leq \frac{4}{\hbar^{2}}\int_{\mathbb{R}^{3}}\left\vert (i\hbar\nabla_{x^{i}}+A(x^{i}))\psi_{N}\right\vert^{2}(x^{1},\dots,x^{N})\ \dd x^{i}.\label{Apllication of Hardy and diamagnetic}
\end{align}
Inserting \eqref{Apllication of Hardy and diamagnetic} inside \eqref{Vij squared} we obtain 
\begin{align*}
\int_{\mathbb{R}^{3N}}\left\vert V_{ij}\psi_{N}\right\vert^{2}(x^{1},\dots,x^{N})\ \dd X^{N}&\leq \frac{1}{4\hbar^{2}\pi^{2}}\int_{\mathbb{R}^{3N}}\left\vert (i\hbar\nabla_{x_{i}}+A(x^{i}))\psi_{N}\right\vert^{2}(x^{1},\dots,x^{N})\ \dd X^{N}\\
&\leq \frac{1}{4\hbar^{2}\pi^{2}} \sum_{\ell=1}^{N}\int_{\mathbb{R}^{3N}}\left\vert (i\hbar\nabla_{x^{\ell}}+A(x^{\ell}))\psi_{N}\right\vert^{2}(x^{1},\dots,x^{N})\ \dd X^{N}\\
&=-\frac{1}{4\hbar^{2}\pi^{2}}\sum_{\ell=1}^
{N}\int_{\mathbb{R}^{3N}}\overline{\psi_{N}}\Delta_{A,x^{\ell}}\psi_{N}(x^{1},\dots,x^{N})\ \dd X^{N}\\
&=-\frac{1}{4\hbar^{2}\pi^{2}}\int_{\mathbb{R}^{3N}}\overline{\psi_{N}}\mathscr{K}_{\hbar}^{N}\psi_{N}(x^{1},\dots,x^{N})\ \dd X^{N}. 
\end{align*}
In view of Young's inequality we deduce that for any arbitrarily small $\eta>0$ it holds that  
\begin{align}
\int_{\mathbb{R}^{3N}}\left\vert V_{ij}\psi_{N}\right\vert^{2}(x^{1},\dots,x^{N})\ \dd X^{N}\leq \frac{\eta}{4\hbar^{2}\pi^{2}}\left\Vert \mathscr{K}_{\hbar}^{N}\psi_{N}\right\Vert_{2}^{2}+\frac{1}{4\eta\hbar^{2}\pi^{2}}\left\Vert \psi_{N}\right\Vert_{2}^
{2}. \label{inf bound for Vij}
\end{align}
Now, note that by the Cauchy-Schwartz inequality we have  
\begin{align*}
\left\Vert \mathscr{V}^{N}\psi_{N}\right\Vert_{2}^{2} &=\int_{\mathbb{R}^{3N}}\left\vert \mathscr{V}^{N}\psi_{N}\right\vert^{2}(x_{1},\dots,x_{N})\ \dd X^{N}\\
&\leq \frac{(N-1)}{4N}\sum_{i\neq j}\int_{\mathbb{R}^{3N}}\left\vert  V_{ij}\psi_{N}\right\vert^{2}(x^{1},\dots,x^{N})\ \dd X^{N}.    \end{align*}
Consequently, in view of \eqref{inf bound for Vij} it follows that 
\begin{align*}
\left\Vert \mathscr{V}^{N}\psi_{N}\right\Vert_{2}^{2}\leq \frac{\eta (N-1)^{2}}{4\hbar^{2}\pi^{2}} \left\Vert \mathscr{K}_{\hbar}^{N}\psi_{N}\right\Vert_{2}^{2} +\frac{(N-1)^{2}}{4\eta\hbar^{2}\pi^{2}}\left\Vert \psi_{N}\right\Vert_{2}^{2}.     
\end{align*}
Thus, for 
$\eta=\eta(\hbar,N)$ sufficiently small we get 
\begin{align*}
\left\Vert \mathscr{V}^{N}\psi_{N}\right\Vert_{2}^{2}\leq a\left\Vert \mathscr{K}_{\hbar}^{N}\psi_{N}\right\Vert_{2}^{2}+b\left\Vert \psi_{N}\right\Vert_{2}^{2}   
\end{align*}
for some $0<a<1$ and some $b>0$. By Theorem \eqref{Kato's theorem}, we conclude that $\mathscr{H}_{\hbar,A}^{N}$ is self-adjoint on $\mathfrak{D}(\mathscr{K}_{\hbar}^{N})$. By Stone's theorem, $U_{\hbar, N}(t)\coloneqq e^{-\frac{it\mathscr{H}_{\hbar,A}^{N}}{\hbar}}$ is a strongly continuous one-parameter unitary group. The operator $R_{\hbar,N}(t)=U_{\hbar,N}^{\ast}(t)R^{\mathrm{in}}_{\hbar,N}U_{\hbar,N}(t)$ is the asserted solution.  We prove \eqref{trace of operator against Hamiltonian}. Using the cyclicity of the trace and that $U_{\hbar,N}$ is unitary and commutes with $\mathscr{H}_{\hbar,A}^{N}$ we see that  
\begin{align*}
\mathrm{tr}\left(\mathscr{H}_{\hbar,A}^{N}R_{\hbar,N}(t)\right)
=\mathrm{tr}\left(U_{\hbar,N}^{\ast}(t)\mathscr{H}_{\hbar,A}^{N}R^{\mathrm{in}}_{\hbar,N}U_{\hbar,N}(t)\right)=\mathrm{tr}\left(\mathscr{H}_{\hbar,A}^{N}R^{\mathrm{in}}_{\hbar,N}\right).     
\end{align*}
Let $\{\psi_{\ell}\}_{\ell=1}^{\infty}$ be the eigenfunction decomposition of $R^{\mathrm{in}}_{\hbar,N}$, i.e. $\{\psi_{\ell}\}_{\ell=1}^{\infty}\subset \mathfrak{H}^{\otimes N}$ is a complete system such that  
\begin{align*}
R_{\hbar,N}^{\mathrm{in}}=\sum_{\ell=1}^{\infty}\lambda_{\ell}\vert \psi_{\ell
}\rangle\vert \langle \psi_{\ell}\vert \qquad \mbox{with}\    \lambda_{\ell}\geq 0,\ \sum_{\ell=1}^{\infty}\lambda_{\ell}=1. \end{align*}
Then, we have 
\begin{align*}
\mathrm{tr}\left(\mathscr{H}_{\hbar,A}^{N}R^{\mathrm{in}}_{\hbar,N}\right)&=\sum_{\ell=1}^{N}\lambda_{\ell}\langle \psi_{\ell}\vert  \mathscr{K}_{\hbar}^{N}\vert\psi_{\ell}\rangle+\sum_{\ell=1}^{N}\lambda_{\ell}\langle \psi_{\ell}\vert \mathscr{V}^{N}\vert\psi_{\ell}\rangle\\
&\leq \sum_{\ell=1}^{N}\lambda_{\ell}(\left\Vert \psi_{\ell}\right\Vert^{2}_{2}+\left\Vert \mathscr{K}_{\hbar}^{N}\psi_{\ell}\right\Vert_{2}^{2})+\sum_{\ell=1}^{N}\lambda_{\ell}(\left\Vert \psi_{\ell}\right\Vert_{2}^{2}+\left\Vert \mathscr{V}^{N}\psi_{\ell}\right\Vert_{2}^{2}). 
\end{align*}
By what we proved, it holds that $\left\Vert \mathscr{V}^{N}\psi_{\ell}\right\Vert_{2}^{2}\leq \frac{1}{2}\left\Vert \mathscr{K}_{\hbar}^{N}\psi_{\ell}\right\Vert_{2}^{2}+2\left\Vert \psi_{\ell}\right\Vert_{2}^{2}$ and hence 
\begin{align*}
\mathrm{tr}\left(\mathscr{H}_{\hbar,A}^{N}R^{\mathrm{in}}_{\hbar,N}\right)\leq C\left(\sum_{\ell=1}^{\infty}\lambda_{\ell}(\left\Vert \psi_{\ell}\right\Vert_{2}^{2}+\left\Vert \mathscr{K}_{\hbar}^{N}\psi_{\ell}\right\Vert_{2}^{2})  \right)\leq C\mathrm{tr}\left((I+\mathscr{K}_{\hbar}^{N})^{2}R^{\mathrm{in}}_{\hbar,N}\right)<\infty    
\end{align*}
by assumption. 
\end{proof}
\begin{rem}
A particular corollary of Theorem \ref{well posedness of magnetic von Neumann} is that the modulated energy \eqref{Def of renormalized modulated energy} is well defined. This is evident upon noticing
$$\underset{t\in \mathbb{R}_{+}}{\sup}\mathrm{tr}(\mathscr{K}_{\hbar}^{N}R_{\hbar,N}(t))\leq \underset{t\in \mathbb{R}_{+}}{\sup}\mathrm{tr}\left(\mathscr{H}_{\hbar,A}^{N}R_{\hbar,N}(t)\right)<\infty$$
and
$$\underset{t\in \mathbb{R}_{+}}{\sup}\int_{\mathbb{R}^{3N}}\mathscr{V}^{N}\rho_{\hbar,N}(t,\dd X^{N})=\underset{t\in \mathbb{R}_{+}}{\sup}\mathrm{tr}\left(\mathscr{V}^{N} R_{\hbar,N}(t)\right)\leq \underset{t\in \mathbb{R}_{+}}{\sup}\mathrm{tr}(\mathscr{H}_{\hbar,A}^{N}R_{\hbar,N}(t))<\infty.$$
\end{rem}
\section{The magnetized Euler-Poisson equation}\label{sec formal derivation}
\subsection{Formal Derivation.} We start by discussing the formal derivation of the magnetized Euler-Poisson equation from the magnetized Vlasov-Poisson equation. Recall that we denote by $\left\{\cdot,\cdot\right\}$ the Poisson brackets defined by 
\begin{align*}
\left\{f,g\right\}=\nabla_{\xi}f\cdot \nabla_{x}g-\nabla_{x}f\cdot \nabla_{\xi}g.     
\end{align*}
Given a vector field $u=(u_{1},\dots,u_{d}):\mathbb{R}^{d}\rightarrow \mathbb{R}^{d}$ we denote by $D_{x}u$ the Jacobian of $u$ defined by $$D_{x}u\coloneqq \left(\frac{\partial u_{j}}{\partial x_{i}}\right)_{1\leq i,j\leq d}.$$ 
\begin{prop}
Suppose that: 
\begin{itemize}
    \item $\mathrm{div}_{x}A=0$
    \item $f$ is a smooth solution
    to the magnetized Vlasov-Poisson equation 
\begin{align} \tag{VPA}
\label{magnetized vlasov Poisson}
\partial_{t}f+\left\{\frac{1}{2}\left\vert \xi-A(x)\right\vert^{2}+V\ast \rho_{f},f\right\}=0,\  \rho_{f}(t,x)=\int_{\mathbb{R}^{3}}f(t,x,\xi)\ \dd \xi.     
\end{align}
\end{itemize}
Let $J_{A,f}:\mathbb{R}_
{+}\times \mathbb{R}^{3}\rightarrow \mathbb{R}^{3}$ be given by 
\begin{align}
 J_{A,f}(t,x)\coloneqq\int_{\mathbb{R}^{3}}(\xi-A(x))f(t,x,\xi)\ \dd  \xi.     
\end{align}
Then $(\rho_{f},J_{A,f})$ satisfies the following equations: 
\begin{align}
\begin{cases}
\begin{array}{lc}
\partial_{t}\rho_{f}+\mathrm{div}_{x}J_{A,f}=0\\
\partial_{t}J_{A,f}+\mathrm{div}_{x}\left(\int_{\mathbb{R}^{3}}(\xi-A)\otimes (\xi-A)f\ \dd \xi\right)+J_{A,f}\mathbf{J}+\rho_{f}\nabla_{x}V\ast \rho_{f}=0
\end{array}\end{cases} \label{equation for 0 and 1st moments}   
\end{align}
where $\mathbf{J}=D_{x}A-D_{x}^{T}A$. 
\end{prop}
\begin{proof}
Integrating in $\xi$ \eqref{magnetized vlasov Poisson} we get 
\begin{align*}
\partial_{t}\rho_{f}+\mathrm{div}_{x}J_{A,f}=0.
\end{align*}
We concentrate on the derivation of the equation for $J_{A,f}$. Multiplying \eqref{magnetized vlasov Poisson} by $\xi-A$ and integrating in $\xi$ we get 
\begin{align}
\partial_{t}J_{A,f}+\int_{\mathbb{R}^{3}}(\xi-A) \left\{\frac{1}{2}\left\vert \xi-A\right\vert^
{2}+V\ast \rho_{f},f\right\}\ \dd \xi=0. \label{component eq for current}  
\end{align}
We compute that 
\begin{align}
&\int_{\mathbb{R}^{3}}(\xi-A)\left\{\frac{1}{2}\left\vert \xi-A\right\vert^{2},f\right\}\ \dd \xi \notag\\
&=\int_{\mathbb{R}^{3}}(\xi-A)\left((\xi-A)\cdot \nabla_{x}f\right)\ \dd \xi+\int_{\mathbb{R}^{3}}(\xi-A)D_{x}A(\xi-A)\cdot \nabla_{\xi}f\ \dd \xi \notag\\
&= \int_{\mathbb{R}^{3}}(\xi-A)\otimes (\xi-A)\nabla_{x}f\ \dd \xi+\int_{\mathbb{R}^{3}}(\xi-A)\otimes (D_{x}A(\xi-A
))\nabla_{\xi}f\ \dd \xi. \label{kinetic p bracket}
\end{align}
Using that $\mathrm{div}
_{x}A=0$ we compute that 
\begin{align*}
\mathrm{div}_{x}((\xi-A)\otimes (\xi-A)f)&=\mathrm{div}_{x}((\xi-A)f)(\xi-A)-D_{x}^{T}A(\xi-A)f\\
&=((\xi-A)\cdot \nabla_{x}f)(\xi-A)-D_{x}^{T}A(\xi-A)f\\
&= (\xi-A)\otimes (\xi-A)\nabla_{x}f-D_{x}^{T}A(\xi-A)f,\end{align*}
so that 
\begin{align}
\int_{\mathbb{R}^{3}}(\xi-A)\otimes (\xi-A)\nabla_{x}f \ \dd \xi&=\mathrm{div}_{x}\left(\int_{\mathbb{R}^{3}}(\xi-A)\otimes (\xi-A)f\ \dd \xi\right)+D_{x}^{T}A\int_{\mathbb{R}^
{3}}(\xi-A)f\ \dd \xi \notag\\
&=\mathrm{div}_{x}\left(\int_{\mathbb{R}^{3}}(\xi-A)\otimes (\xi-A)f \ \dd \xi\right)+D_{x}^{T}AJ_{A,f}.\label{integral2} 
\end{align}
In addition, integrating by parts and using again that $\mathrm{div}_{x}A=0$ we get  
\begin{align}
-\int_{\mathbb{R}^{3}}(\xi-A)\otimes (D_{x}A (\xi-A))\nabla_{\xi}f\ \dd \xi
=&\int_{\mathbb{R}^{3}}\mathrm{div}_{\xi}\left((\xi-A)\otimes D_{x}A(\xi-A)\right)f \ \dd \xi \notag\\
=&-\int_{\mathbb{R}^{3}}\mathrm{div}_{x}(A)(\xi-A)f\ \dd \xi-\int_{\mathbb{R}^{3}} D_{x}A(\xi-A)f\ \dd \xi \notag\\
=&-D_{x}A\int_{\mathbb{R}^{3}}(\xi+A)f\ \dd \xi=-D_{x}AJ_{A,f}. \label{integral1} 
\end{align}
Substituting \eqref{integral2}-\eqref{integral1} inside \eqref{kinetic p bracket} we deduce that 
\begin{align}
\int_{\mathbb{R}^{3}}(\xi-A)\left\{\frac{1}{2}\left\vert \xi-A\right\vert^{2},f\right\}\ \dd\xi=\mathrm{div}_{x}\left(\int_{\mathbb{R}^{3}}(\xi-A)\otimes (\xi-A)f\ \dd \xi\right)+(D_{x}^{T}A-D_{x}A)J_{A,f}. \label{first part}    
\end{align}
In addition we compute that 
\begin{align}
\int_{\mathbb{R}^{3}}(\xi-A)\left\{V\ast \rho_{f},f\right\}\ \dd \xi=-\int_{\mathbb{R}^{3}}(\xi-A)\nabla_{x}V\ast \rho_{f}\cdot\nabla_{\xi}f\ \dd \xi=\rho_{f}\nabla_{x}V\ast \rho_{f}. \label{second part} 
\end{align}
Combining \eqref{first part} with \eqref{second part} we conclude that  
\begin{align*}
\partial_{t}J_{A,f}+\mathrm{div}_{x}\left(\int_{\mathbb{R}^{3}}(\xi-A)\otimes (\xi-A)f\ \dd \xi\right)+(D_{x}^{T}A-D_{x}A)J_{A,f}+\rho_{f}\nabla_{x}V\ast \rho_{f}=0.  \end{align*}
This establishes the second equation in \eqref{equation for 0 and 1st moments}. 
\end{proof}
The derivation of the Euler-Poisson equation now follows by taking $f$ to be the monokinetic ansatz, i.e.  
$f(t,x,\xi)=\rho_{f}(t,x)\delta(\xi-A-\frac{J_{A,f}}{\rho_{f}})$. For this $f$ we obtain from \eqref{equation for 0 and 1st moments} the  following  magnetized Euler-Poisson system written in terms of the current and the density
\begin{align}
\begin{cases}
\begin{array}{lc}
\partial_{t}\rho_{f}+\mathrm{div}_{x}J_{A,f}=0\\
\partial_{t}J_{A,f}+\mathrm{div}_{x}\left(\frac{J_{A,f}\otimes J_{A,f}}{\rho_{f}}\right)+J_{A,f}\mathbf{J}+\rho_{f}\nabla_{x}V\ast \rho_{f}=0. 
\end{array}\end{cases} \label{Magnetized Euler Poisson current density}   
\end{align}
We conclude by explaining how to write \eqref{Magnetized Euler Poisson current density} in terms of the density and the velocity field, which would lead to  \eqref{Magnetized Euler Poisson velocity density intro}. 
\begin{lem}
Let $(\rho,J)$ be a smooth solution to \eqref{Magnetized Euler Poisson current density} and let $u=\frac{J}{\rho}$. Then $(\rho,u)$ is a solution to \eqref{Magnetized Euler Poisson velocity density intro}. 
\end{lem}
\begin{proof}
The first equation in \eqref{Magnetized Euler Poisson velocity density intro} follows immediately from the first equation in \eqref{Magnetized Euler Poisson current density}. In addition, substituting the relation $J=\rho u$ inside the second equation of \eqref{Magnetized Euler Poisson current density} we get 
\begin{align*}
\partial_{t}\rho u+\rho\partial_{t}u+\mathrm{div}_{x}(\rho u\otimes u)+\rho u\mathbf{J}+\rho \nabla_{x}V\ast \rho=0.     
\end{align*}
Dividing by $\rho$ yields 
\begin{align*}
\frac{\partial_{t}\rho }{\rho}u+\partial_{t}u+\frac{\mathrm{div}_{x}(\rho u\otimes u)}{\rho}+u\mathbf{J}+\nabla_{x}V\ast \rho=0,   
\end{align*}
which is recast as
\begin{align}
-\frac{\mathrm{div}_{x}(\rho u)}{\rho}u+\partial_{t}u+\frac{\mathrm{div}_{x}(\rho u\otimes u)}{\rho}+u\mathbf{J}+\nabla_{x}V\ast \rho=0. 
\label{First expansion of Euler Poisson}
\end{align}
Note that 
\begin{align*}
-\mathrm{div}_{x}(\rho u)u=-\nabla_{x}\rho (u\otimes u)-\rho (\mathrm{div}_{x} u)u\label{fromula} \end{align*}
and that 
\begin{align*}
\mathrm{div}_{x}(\rho u\otimes u)=\nabla_{x}\rho (u\otimes u)+\rho\mathrm{div}_{x}(u\otimes u)=\nabla_{x}\rho (u\otimes u)+\rho\mathrm{div}_{x}(u)u+\rho uD_{x}u.     
\end{align*}
We therefore conclude that 
\begin{align}
-\frac{\mathrm{div}_{x}(\rho u)}{\rho}u+\frac{\mathrm{div}_{x}(\rho u\otimes u)}{\rho}=uD_{x}u. \end{align}
Substituting the above equation inside \eqref{First expansion of Euler Poisson} we obtain 
\begin{align*}
\partial_{t}u+uD_{x}u+u\mathbf{J}+\nabla_{x}V\ast \rho=0.     
\end{align*}
\end{proof}
\begin{rem}
The magnetized Vlasov-Poisson equation \eqref{magnetized vlasov Poisson} will not appear again in the sequel. We introduced it in order to clarify how to obtain the monokinetic equation \eqref{Magnetized Euler Poisson velocity density intro} from averaged quantities associated with an appropriate kinetic equation, which in turn makes the relation between  \eqref{SP equation intro} and \eqref{Magnetized Euler Poisson velocity density intro} more visible.   
\end{rem}
\begin{rem}
We may write the system \eqref{Magnetized Euler Poisson velocity density intro} component-wise as 
\begin{align*}
\partial_{t}u_{j}+\sum_{k=1}^{3}u_{k}\partial_{x_{k}}u_{j}+\sum_{k=1}^{3}u_{k}(\partial_{x_{k}}A_{j}-\partial_{x_{j}}A_{k})+\partial_{x_{j}}V\ast \rho=0,\ 1\leq j\leq  3. 
\end{align*}
\end{rem}
\subsection{Local well-posedness}
In this part we prove the local in time well-posedness of \eqref{Magnetized Euler Poisson velocity density intro}. Short time well-posedness for Euler-Poisson is natural since solutions may develop blowup in finite time for general initial data, see \cite{yuen2011blowup}.  In what follows we shall make free use of the following interpolation inequalities: for any $u:\mathbb{R}^{3}\rightarrow \mathbb{R}^{3}$ it holds that
\begin{align*}
 \left\Vert u\right\Vert_{H^{4}}\leq C\left(\left\Vert u\right\Vert_{2}+\left\Vert \Delta^{2}_{x}u\right\Vert_{2}  \right)     
\end{align*}
and 
\begin{align*}
\left\Vert u\right\Vert_{H^{3}}\leq C\left(\left\Vert u\right\Vert_{2}+\left\Vert \nabla_{x}\Delta_{x}u\right\Vert_{2}\right).      
\end{align*}
We will need the following Lemma. 
\begin{lem}\label{Stability H-1}
Suppose that: 
\begin{itemize}
    \item $v_{1},v_{2}\in L^{\infty}([0,T];H^{4}(\mathbb{R}^{3}))$ are given vector fields and for $i=1,2$ it holds that 
\begin{align*}
\left\Vert v_{i}\right\Vert_{L^{\infty}_{t}L^{2}_{x}}^{2} +\left\Vert \Delta^{2}_{x}v_{i}\right\Vert_{L^{\infty}_{t}L^{2}_{x}}^{2}\le M\ \mbox{for some} \ M>1.       
\end{align*}
\item $\rho^{\mathrm{in}}\in H^{3}(\mathbb{R}^{3})\cap \mathcal{P}(\mathbb{R}^{3}). $
\end{itemize}
Let $\rho_{1}$ and $\rho_{2}$ be solutions  to the linear transport equation
\begin{align*}
\partial_{t}\rho+\mathrm{div}_{x}(\rho v)=0,\ \rho(0,\cdot)=\rho^{\mathrm{in}}\end{align*}
with $v=v_{1}$ and $v=v_{2}$ respectively. 
Then, there is some $C>0$ such that it holds that 
\begin{align*}
\underset{t\in [0,T]}{\mathrm{ess}\sup}\left\Vert(\rho_{1}-\rho_{2})(t,\cdot) \right\Vert_{\dot{H}^{-1}}^{2}\leq& (3MT+2Te^{MCT})  \underset{t\in [0,T]}{\mathrm{ess}\sup}\left\Vert (\rho_{1}-\rho_{2})(t,\cdot)\right\Vert_{\dot{H}^{-1}}^{2}\\
&+2Te^{MCT}\left\Vert \rho^{\mathrm{in}}\right\Vert_{\infty} \underset{t\in [0,T]}{\mathrm{ess}\sup}\left\Vert(v_{1}-v_{2})(t,\cdot) \right\Vert_{2}^{2}.\end{align*}
\end{lem}
\begin{proof}
Set 
\begin{align*}
E(t)\coloneqq \int_{\mathbb{R}^{3}}V\ast (\rho_{1}-\rho_{2})(t,x)(\rho_{1}-\rho_{2})(t,x)\ \dd x.    
\end{align*}
Note that by passing to Fourier we have 
\begin{align*}
E(t)\coloneqq \left\Vert (\rho_{1}-\rho_{2})(t,\cdot)\right\Vert_{\dot{H}^{-1}}^{2}.      
\end{align*}
We compute 
\begin{align}
\frac{\dd E(t)}{\dd t}=&2\int_{\mathbb{R}^{3}}V\ast(\rho_{1}-\rho_{2})\partial_{t}(\rho_{1}-\rho_{2})(t,x)\ \dd x
 \notag\\=&2\int_{\mathbb{R}^{3}}V\ast(\rho_{1}-\rho_{2})(t,x)\mathrm{div}_{x}(\rho_{2}v_{2}-\rho_{1}v_{1})(t,x)\ \dd x \notag\\
=&2\int_{\mathbb{R}^{3}}V\ast(\rho_{1}-\rho_{2})(t,x)\mathrm{div}_{x}((\rho_{2}-\rho_{1})v_{2})(t,x) \ \dd x \notag\\
&+2\int_{\mathbb{R}^{3}}V\ast(\rho_{1}-\rho_{2})(t,x)\mathrm{div}_{x}(\rho_{1}(v_{2}-v_{1}))(t,x)\ \dd x \notag\\
=& -2\int_{\mathbb{R}^{3}}v_{2}(t,x)\cdot \nabla_{x} V\ast (\rho_{1}-\rho_{2})(t,x)(\rho_{2}-\rho_{1})(t,x)\ \dd x \notag\\
&-2\int_{\mathbb{R}^{3}}(v_{2}-v_{1})(t,x)\cdot\nabla_{x} V\ast (\rho_{1}-\rho_{2})(t,x)\rho_{1}(t,x) \ \dd x\coloneqq F_{1}+F_{2}. \label{time der of E} 
\end{align}
Using that $-\Delta_{x}V=\delta_{0}$ we recast $F_{1}$ as follows 
\begin{align}
F_{1}=&2\int_{\mathbb{R}^{3}}v_{2}(t,x)\cdot\nabla_{x} V\ast (\rho_{1}-\rho_{2})(t,x)\mathrm{div}_{x}(\nabla_{x} V\ast(\rho_{1}-\rho_{2}))(t,x) \ \dd x \notag\\
=&-2\int_{\mathbb{R}^{3}}D_{x}\nabla _{x}V\ast(\rho_{1}-\rho_{2})(t,x)\nabla_{x} V\ast (\rho_{1}-\rho_{2})(t,x)\cdot v_{2}(t,x)\ \dd x \notag\\
&-2\int_{\mathbb{R}^{3}}\nabla_{x} V\ast (\rho_{1}-\rho_{2})\nabla_{x} V\ast (\rho_{1}-\rho_{2})D_{x} v_{2}(t,x) \ \dd x \notag\\
=&-\int_{\mathbb{R}^{3}}\nabla_{x}\left\vert \nabla_{x} V\ast (\rho_{1}-\rho_{2})\right\vert^{2}(t,x)\cdot v_{2}(t,x)\ \dd x \notag\\
&-2\int_{\mathbb{R}^{3}}\nabla_{x} V\ast (\rho_{1}-\rho_{2})(t,x)\nabla_{x} V\ast (\rho_{1}-\rho_{2})(t,x)D_{x}v_{2}(t,x)\ \dd x \notag\\
=&\int_{\mathbb{R}^{3}}\left\vert \nabla_{x} V\ast (\rho_{1}-\rho_{2})\right\vert^{2}\mathrm{div}_{x}(v_{2}(t,x))\ \dd x \notag\\
&-2\int_{\mathbb{R}^{3}}\nabla_{x} V\ast (\rho_{1}-\rho_{2})(t,x)\nabla_{x} V\ast (\rho_{1}-\rho_{2})(t,x)D_{x}v_{2}(t,x)\ \dd x. 
\label{F1 calculation}
\end{align}
Inserting \eqref{F1 calculation} inside \eqref{time der of E} we obtain 
\begin{align}
\frac{\dd E(t)}{\dd t}&=  \int_{\mathbb{R}^{3}}\left\vert \nabla_{x} V\ast (\rho_{1}-\rho_{2})\right\vert^{2}\mathrm{div}_{x}(v_{2}(t,x))\ \dd x \notag\\
&-2\int_{\mathbb{R}^{3}}\nabla_{x} V\ast (\rho_{1}-\rho_{2})(t,x)\nabla_{x} V\ast (\rho_{1}-\rho_{2})(t,x)D_{x}v_{2}(t,x)\ \dd x \notag\\
&-2\int_{\mathbb{R}^{3}} (v_{2}-v_{1})(t,x) \cdot\nabla_{x} V\ast (\rho_{1}-\rho_{2})(t,x)\rho_{1}(t,x)\ \dd x. \label{calculation of E_{2}} 
\end{align}
Using the Sobolev embedding $L^{\infty}(\mathbb{R}^{3})\hookrightarrow H^{2}(\mathbb{R}^{3})$ we have $\underset{t\in [0,T]}{\mathrm{ess}\sup}\left\Vert D_{x}v_{2}(t,\cdot)\right\Vert_{\infty}\leq CM$.  Therefore the  first term in the right-hand side of \eqref{calculation of E_{2}} is bounded by 
\begin{align}
\underset{t\in [0,T]}{\mathrm{ess}\sup}\left\Vert \mathrm{div}_{x}v_{2}(t,\cdot)\right\Vert_{\infty}
\int_{\mathbb{R}^{3}}\left\vert \nabla_{x}V\ast(\rho_{1}-\rho_{2})\right\vert^{2}(t,x)\ \dd x\leq CM\left\Vert (\rho_{1}-\rho_{2})(t,\cdot)\right\Vert_{\dot{H}^{-1}}^{2}. \label{estfirst}   
\end{align}
The second term in the right-hand side of \eqref{calculation of E_{2}} is bounded by 
\begin{align}
2\underset{t\in [0,T]}{\mathrm{ess}\sup}\left\Vert D_{x}v_{2}(t,\cdot)\right\Vert_{\infty} \int_{\mathbb{R}^{3}}\left\vert \nabla_{x}V\ast (\rho_{1}-\rho_{2})\right\vert^{2}(t,x)\ \dd x\leq CM \left\Vert (\rho_{1}-\rho_{2})(t,\cdot)\right\Vert_{\dot{H}^{-1}}^{2}.\label{estsecond}      
\end{align}
By classical estimates for linear transport equations we have 
\begin{align*}
\underset{t\in [0,T]}{\mathrm{ess}\sup}\left\Vert \rho_{1}(t,\cdot)\right\Vert_{\infty}\leq \exp\Big({\underset{t\in [0,T]}{\mathrm{ess}\sup}\left\Vert \mathrm{div}_{x}v_{1}(t,\cdot)\right\Vert_{\infty}}t\Big)\left\Vert \rho^{\mathrm{in}}\right\Vert_{\infty}      
\end{align*}
and thus
$$\underset{t\in [0,T]}{\mathrm{ess}\sup}\left\Vert \rho_{1}(t,\cdot)\right\Vert_{\infty}\leq e^{CMT}\left\Vert \rho^{\mathrm{in}}\right\Vert_{\infty}.$$ 
Therefore the third term in the right-hand side of  \eqref{calculation of E_{2}} is bounded by 
\begin{align}
&\underset{t\in [0,T]}{\mathrm{ess}\sup}\left\Vert \rho_{1}(t,\cdot)\right\Vert_{\infty}\left(\int_{\mathbb{R}^
{3}}\left\vert \nabla_{x}V\ast (\rho_{1}-\rho_{2})\right\vert^{2}(t,x)\ \dd x+\left\Vert (v_{1}-v_{2})(t,\cdot)\right\Vert_{2}^{2}  \right) \notag
\\ &\leq e^{CMT}\left\Vert\rho^{\mathrm
{in}} \right\Vert_{\infty}\left(\left\Vert (\rho_{1}-\rho_{2})(t,\cdot)\right\Vert_{\dot{H}^{-1}}^{2}+\left\Vert (v_{1}-v_{2})(t,\cdot)\right\Vert_{2}^{2} \right). \label{third termest}     
\end{align}
Substituting the estimates \eqref{estfirst}-\eqref{third termest} in \eqref{calculation of E_{2}} we get  
\begin{align*}
\frac{\dd}{\dd t}\left\Vert (\rho_{1}-\rho_{2})(t,\cdot)\right\Vert_{\dot{H}^{-1}}^{2}=\frac{\dd E(t)}{\dd t}&\leq 3MC\left\Vert (\rho_{1}-\rho_{2})(t,\cdot)\right\Vert_{\dot{H}^{-1}}^{2}\\
&+2e^{CMT}\left\Vert \rho^{\mathrm{in}}\right\Vert_{\infty} \left(\left\Vert (\rho_{1}-\rho_{2})(t,\cdot)\right\Vert_{\dot{H}^{-1}}^{2}+\left\Vert (v_{1}-v_{2})(t,\cdot)\right\Vert_{2}^{2} \right).     
\end{align*}
Therefore integrating in time and maximizing over $[0,T]$ it follows that 
\begin{align*}
\underset{t\in [0,T]}{\mathrm{ess}\sup}\left\Vert (\rho_{1}-\rho_{2})(t,\cdot)\right\Vert_{\dot{H}^{-1}}^{2}&\leq (3MT+2Te^{MCT})  \underset{t\in [0,T]}{\mathrm{ess}\sup}\left\Vert (\rho_{1}-\rho_{2})(t,\cdot)\right\Vert_{\dot{H}^{-1}}^{2}\\
&+2Te^{MCT}\left\Vert \rho^{\mathrm{in}}\right\Vert_{\infty} \underset{t\in [0,T]}{\mathrm{ess}\sup}\left\Vert(v_{1}-v_{2})(t,\cdot) \right\Vert_{2}^{2}.     
\end{align*}
\end{proof}
\begin{lem}
Suppose that: 
\begin{itemize}
    \item $v:[0,T]\times \mathbb{R}^{3}\rightarrow \mathbb{R}^{3}$ is a given vector field and suppose there is some $M>1$ such that 
    \begin{align*}
\left\Vert v\right\Vert_{L^{\infty}_{t}L^{2}_{x}}^{2}+\left\Vert \Delta^{2}_{x}v\right\Vert_{L^{\infty}_{t}L^{2}_{x}}\leq M. \end{align*}
\item $\rho^{\mathrm{in}}\in H^{3}(\mathbb{R}^{3})\cap \mathcal{P}(\mathbb{R}^{3})$. 
\end{itemize}
Let $\rho$ be the solution to the linear transport equation
\begin{align}
\partial_{t}\rho+\mathrm{div}_{x}(\rho v)=0, \ \rho(0,\cdot)=\rho^{\mathrm{in}}.  \label{transporteq1}  
\end{align}
Then, there is some constant $C>0$ such that it holds that 
\begin{align*}
\left\Vert \rho(t,\cdot)\right\Vert_{H^{3}}^{2} \leq e^{CMt}\left\Vert \rho^{\mathrm{in}}\right\Vert_{H^{3}}
^{2}.       
\end{align*}
\label{est of H3 norm of rho}
In particular, it holds that $\left\Vert \rho(t,\cdot)\right\Vert_{L^{\infty}\cap L^{1}} \leq e^{CMt}\left\Vert \rho^{\mathrm{in}}\right\Vert_{H^{3}}$. 
\end{lem}
\begin{proof}
In what follows $C>0$ stands for an effective constant which may vary between the different terms. Note that by Sobolev embedding we have $H^{4}(\mathbb{R}^{3})\hookrightarrow W^{2,\infty}(\mathbb{R}^{3}) $ with the bound $\left\Vert v\right\Vert_{W^{2,\infty}}\leq C\left\Vert v\right\Vert_{H^{4}}$. We shall freely use this bound in what follows. First, multiplying \eqref{transporteq1} by $\rho$ and integrating we arrive at the estimate  
\begin{align}
\frac{\dd}{\dd t}\left\Vert \rho(t,\cdot)\right\Vert_{2}^{2} \leq CM\left\Vert \rho(t,\cdot)\right\Vert_{2}^{2}.  \label{L2 derivative ofro}
\end{align}
Taking the 
$\nabla_{x}\Delta_{x}$ in \eqref{transporteq1} we get 
\begin{align}
\partial_{t}\nabla_{x}\Delta_{x}\rho+\nabla_{x}\Delta_{x}\mathrm{div}_{x}(\rho v)=0.     
\end{align}
We shall employ the following formula 
\begin{align} \nabla_{x}\Delta_{x}\operatorname{div}(\rho v) ={}& \nabla_{x}\bigl(v\cdot\nabla_{x}\Delta_{x}\rho\bigr) +2\nabla_{x}\bigl(D_{x}\nabla_{x}\rho:D_{x} v\bigr) +\nabla_{x}\bigl(\nabla_{x}\rho\cdot\Delta_{x} v\bigr) \notag\\ &+ \nabla_{x}\bigl((\Delta_{x}\rho)\operatorname{div}_{x}v\bigr) +2\nabla_{x}\bigl( \nabla_{x}\rho\cdot\nabla_{x}\operatorname{div}_{x}v \bigr) +\nabla_{x}\bigl( \rho\,\Delta_{x}\operatorname{div}_{x}v \bigr). \end{align}  
Multiplying by $\nabla_{x}\Delta_{x}\rho$ we get 
\begin{align*}
\frac{1}{2}\frac{\dd}{\dd t}\left\Vert \nabla_{x}\Delta_{x}\rho\right\Vert_{2}^{2}=&-\int_{\mathbb{R}^{3}}\nabla_{x}(v\cdot \nabla_{x}\Delta_{x}\rho)\cdot \nabla_{x}\Delta_{x}\rho\ \dd x-2\int_{\mathbb{R}^{3}}\nabla_{x}(D_{x}\nabla_{x}\rho:D_{x}v)\cdot \nabla_{x} \Delta_{x}\rho \ \dd x\\
&-\int_{\mathbb{R}^{3}}\nabla_{x}(\nabla_{x}\rho\cdot \Delta_{x}v)\cdot \nabla_{x}\Delta_{x}\rho \ \dd x-\int_{\mathbb{R}^{3}}\nabla_{x}(\Delta_{x}\rho \mathrm{div}_{x}v)\cdot \nabla_{x}\Delta_{x}\rho \ \dd x\\
&-2\int_{\mathbb{R}^{3}}\nabla_{x}(\nabla_{x}\rho \cdot \nabla_{x}\mathrm{div}_{x}(v))\cdot \nabla_{x}\Delta_{x}\rho\ \dd x-\int_{\mathbb{R}^{3}}\nabla_{x}(\rho\Delta_{x}\mathrm{div}_{x}v)\cdot \nabla_{x}\Delta_{x}\rho\ \dd x\\
\coloneqq&\sum_{k=1}^{6}I_{k}.  
\end{align*}
We estimate each of the $I_{k}$. To estimate $I_{1}$, note that
\begin{align}
I_{1}&=-\int_{\mathbb{R}^{3}}D_{x}v\nabla_{x}\Delta_{x}\rho\cdot \nabla_{x}\Delta_{x}\rho\ \dd x-\frac{1}{2}\int_{\mathbb{R}^{3}}v\cdot \nabla_{x}\left\vert\nabla_{x}\Delta_{x} \rho\right\vert^{2}\ \dd x \notag \\
&=-\int_{\mathbb{R}^{3}} D_{x}v\nabla_{x}\Delta_{x}\rho \cdot \nabla_{x}\Delta_{x}\rho+\frac{1}{2}\int_{\mathbb{R}^{3}}\mathrm{div}_{x}(v)\left\vert \nabla_{x}\Delta_{x}\rho\right\vert^{2}\ \dd x \notag\\
&\leq C\left\Vert D_{x}v\right\Vert_{\infty}\int_{\mathbb{R}^{3}}\left\vert \nabla_{x}\Delta_{x} \rho\right\vert^{2}\ \dd x  \leq CM\left\Vert \nabla_{x}\Delta_{x}\rho(t,\cdot)\right\Vert_{2}^{2}. \label{I1 est H3}  
\end{align}
To estimate $I
_{2}$, note that we have 
\begin{align*}
\nabla_{x}(D_{x}\nabla_{x}\rho:D_{x} v)&=\nabla_{x}\left(\sum_{1\leq i,j\leq 3}\partial_{x_{i}x_{j}}\rho \partial_{x_{j}}v_{i}\right)\\
&=\sum_{1\leq i,j\leq 3}\nabla_{x}\partial_{x_{i}x_{j}}\rho \partial_{x_{j}}v_{i}+\sum_{1\leq i,j\leq 3}\partial_{x_{i}x_{j}}\rho\nabla_{x}\partial_{x_{j}}v_{i}    
\end{align*} 
and therefore 
\begin{align}
I_{2}\leq \left\Vert D_{x}v\right\Vert_{\infty}\left\Vert \rho\right\Vert_{H
^{3}}^{2}+\left\Vert D_{x}^{2}v\right\Vert_{\infty}\left\Vert\rho \right\Vert_{H^{3}}^{2}\leq CM(\left\Vert \rho\right\Vert^{2}_{2}+ \left\Vert \nabla_{x}\Delta_{x}\rho\right\Vert_{2}^{2}).  \label{I2 est H3}      
\end{align}
To estimate $I_{3}$ we note that 
\begin{align}
I_{3}&=-\int_{\mathbb{R}^{3}}D_{x}\nabla_{x}\rho\Delta_{x}v\cdot \nabla_{x}\Delta_{x}\rho\ \dd x-\int_{\mathbb{R}^{3}}D_{x}\Delta_{x}v\nabla_{x}\rho\cdot \nabla_{x}\Delta_{x}\rho \ \dd x \notag\\
&\leq C\left\Vert D_{x}^{2}v\right\Vert_{\infty} \left(\left\Vert D_{x}\nabla_{x}\rho\right\Vert_{2}^{2}+\left\Vert \nabla_{x}\Delta_{x}\rho\right\Vert_{2}^{2}  \right)+\left\Vert \nabla_{x}\Delta_{x}\rho\right\Vert_{2}^{2}+\left\Vert \nabla_{x}\rho\right\Vert_{\infty}^{2} \left\Vert D_{x}\Delta_{x}v\right\Vert_{2}^{2} \notag\\
&\leq CM\left(\left\Vert \rho\right\Vert_{2}^{2} +\left\Vert \nabla_{x}\Delta_{x}\rho\right\Vert_{2}^{2} \right). \label{I3 est H3} 
\end{align}
To estimate $I_{4}$, we have 
\begin{align}
I_{4}&=-\int_{\mathbb{R}^{3}}\vert\nabla_{x}\Delta_{x}\rho\vert^{2}\mathrm{div}_{x}(v)\ \dd x-\int_{\mathbb{R}^{3}}\Delta_{x}\rho\nabla_{x}\mathrm{div}_{x}(v)\cdot \nabla_{x}\Delta_{x}\rho\ \dd  x \notag\\
&\leq \left\Vert \mathrm{div}_{x}v\right\Vert_{\infty}\left\Vert \nabla_{x}\Delta_{x}\rho(t,\cdot)\right\Vert_{2}^{2}+\left\Vert \nabla_{x}\mathrm{div}_{x}v\right\Vert_{\infty}\left(\left\Vert \Delta_{x}\rho(t,\cdot)\right\Vert_{2}^{2} +\left\Vert \nabla_{x}\Delta_{x}\rho(t,\cdot)\right\Vert_{2}^{2} \right)   
\notag\\
&\leq CM\left(\left\Vert \nabla_{x}\Delta_{x}\rho(t,\cdot)\right\Vert_{2}^{2}+\left\Vert \Delta_{x}\rho(t,\cdot)\right\Vert_{2}^{2}  \right)\leq CM\left(\left\Vert \rho(t,\cdot)\right\Vert_{2}^{2}
+\left\Vert \nabla_{x}\Delta_{x}\rho(t,\cdot)\right\Vert_{2}^{2} \right).    \label{I4 est H3} 
\end{align}
By the same token we arrive at the following estimates for $I_{5}$ and $I_{6}$:
\begin{align}
I_{5}\leq CM(\left\Vert \rho(t,\cdot)\right\Vert_{2}^{2}+\left\Vert \nabla_{x}\Delta_{x}\rho(t,\cdot)\right\Vert_{2}^{2}), \ I_{6}\leq CM(\left\Vert \rho(t,\cdot)\right\Vert_{2}^{2}+\left\Vert \nabla_{x}\Delta_{x}\rho(t,\cdot)\right\Vert_{2}^{2}).\label{I5I6 est H3}   
\end{align}
Gathering \eqref{L2 derivative ofro} and \eqref{I1 est H3}-\eqref{I5I6 est H3} we conclude that 
\begin{align*}
\frac{\dd}{\dd t}\left(\left\Vert \rho(t,\cdot)\right\Vert_{2}^{2}+\left\Vert \nabla_{x}\Delta_{x}\rho(t,\cdot)\right\Vert_{2}^{2}  \right)\leq CM \left(\left\Vert \rho(t,\cdot)\right\Vert_{2}^{2}+\left\Vert \nabla_{x}\Delta_{x}\rho(t,\cdot)\right\Vert_{2}^{2}  \right),   
\end{align*}
which in view of Gr\"onwall's lemma yields the announced inequality.
\end{proof}
We are now ready to prove the main theorem of this section, stated below. 
\begin{thm}\label{well posedness of EPA}
Suppose that: 
\begin{itemize}
    \item  $\mathbf{J}$ is given by  \eqref{def of antisymmetrized gradient} where $A$ satisfies   \textup{\textbf{A1}-\textbf{A2}}. 
    \item $(\rho^{\mathrm{in}},u^{\mathrm{in}})\in (H^{3}(\mathbb{R}^{3})\cap\mathcal{P}(\mathbb{R}^{3})) \times H^{4}(\mathbb{R}^{3})$. 
\end{itemize}
Then, there exist some $T >0$ such that the Cauchy problem\begin{align}
\begin{cases}
\begin{array}{lc}
\partial_{t}\rho+\mathrm{div}_{x}(\rho u)=0, \ \rho(0,\cdot)=\rho^{\mathrm{in}}\\
\partial_{t}u+uD_{x}u+u\mathbf{J}+\nabla_{x}V\ast \rho=0, \ u(0,\cdot)=u^{\mathrm{in}}. 
\end{array}\end{cases} \label{Magnetized Euler Poisson Cauchy problem}    
\end{align}
has a unique solution 
$(\rho,u)\in L^{\infty}([0,T];H^{3}(\mathbb{R}^{3})\cap \mathcal{P}(\mathbb{R}^{3}))\times L^{\infty}([0,T];H^{4}(\mathbb{R}^{3}))$.
\end{thm}
\begin{proof}
\textbf{Step 1.} Fix \(M>1\) and \(T>0\) to be chosen later and set 
\begin{align*}
\mathfrak{X}_{T,M}
=
\Big\{
v(0)=u^{\mathrm{in}}, \left\Vert v\right\Vert_{L^{\infty}([0,T];L^{2}(\mathbb{R}^{3}))}^{2}+\left\Vert \Delta_{x}^{2}v\right\Vert_{L^{\infty}([0,T];L^{2}(\mathbb{R}^{3}))}^{2}  \leq M
\Big\}\subset L^{\infty}([0,T];H^{4}(\mathbb{R}^{3})).
\end{align*}
Equip $\mathfrak{X}_{T,M}$  with the metric 
\begin{equation}
D(v_1,v_2)
\coloneqq
\left\Vert v_{1}-v_{2}\right\Vert_{L^{\infty}([0,T];L^{2}(\mathbb{R}^{3}))} .
\tag{3}
\end{equation}
Note that $(\mathfrak{X}_{T,M},D)$ is a complete metric space.
Given $v\in \mathfrak{X}_{T,M}$ denote by $\rho[v]\in L^{\infty}([0,T];H^{3}(\mathbb{R}^{3}))$ the solution of  
\begin{align}
\begin{cases}
\partial_t\rho+\mathrm{div}_{x}(\rho v)=0,\\
\rho(0,\cdot)=\rho^{\mathrm{in}}.
\end{cases} \label{transporteq}    
\end{align}
Define $\Phi[v]\coloneqq V\ast \rho[v]$ and consider the solution $w[v]$ of the linear equation
\begin{align}
\partial_{t}w+vD_{x} w+w\mathbf{J}+\nabla_{x}\Phi[v]=0, \ w(0,\cdot)=u^{\mathrm{in}}. 
\label{eq for w}    \end{align}
We claim that $v\mapsto w[v]$ is a contraction from $\mathfrak{X}_{T,M}$ to itself for appropriate choice of $T$ and $M$.\\
\textbf{Step 2.}  In this step we establish that $w[v]\in \mathfrak{X}_{T,M}$ for well chosen $T,M$. For brevity, we set $w=w[v]$ and $\rho=\rho[v]$. In the sequel $C$ stands for any constant independent of $T,M$ and may vary between the different terms.  Multiplying by $w $ equation \eqref{eq for w} and integrating  we get 
\begin{align*}
\frac{\dd}{\dd t}\left\Vert w(t,\cdot)\right\Vert_{2}^{2}=-\int_{\mathbb{R}^{3}}v\cdot \nabla_{x}\left\vert w\right\vert^{2}(t,x)\ \dd x-
2\int_{\mathbb{R}^{3}}\nabla_{x}\Phi[v]\cdot w(t,x)\ \dd x. 
\end{align*}
Note that we used that $w\mathbf{J}\cdot w=0$. Therefore, after integration by parts and using that $\left\Vert \mathrm{div}_{x}v\right\Vert_{\infty}\leq CM$ and that $\left\Vert \nabla_{x}V\ast \rho(t,\cdot)\right\Vert_{2}\leq C\left\Vert \rho(t,\cdot)\right\Vert_{\frac{6}{5}}$ we get 
\begin{align}
 \frac{\dd}{\dd t}\left\Vert w(t,\cdot)\right\Vert_{2}^{2}&\leq CM\left\Vert w(t,\cdot)\right\Vert_{2}^{2}+\left\Vert w(t,\cdot)\right\Vert_{2}^{2}+\left\Vert \nabla_{x}\Phi[v](t,\cdot)\right\Vert_{2}^{2} \notag\\
 &\leq CM(\left\Vert w(t,\cdot)\right\Vert_{2}^
{2}+\left\Vert \rho(t,\cdot)\right\Vert_{\frac{6}{5}}^{2})\leq CM(\left\Vert w(t,\cdot)\right\Vert_{2}^{2}+e^{MCT}). \label{time der of Lp norm of w} 
\end{align}
In the last inequality we applied Lemma \ref{est of H3 norm of rho}. 
We proceed by calculating time derivative of $\left\Vert \Delta_{x}^{2}w(t,\cdot)\right\Vert_{2}^{2}$. We will make use of the following formula for the Laplacian: 
\begin{align}
\Delta^2_{x}(vD_{x}w)
={}& (\Delta^2_{x} v)D_{x}w + v\,\Delta^2_{x} D_{x}w \notag\\
&+ 4\sum_{i=1}^3 (\partial_{x_{i}} v)(\partial_{x_i} \Delta_{x} D_{x}w)
+ 4\sum_{i=1}^3 (\partial_{x_i} \Delta_{x} v)(\partial_{x_{i}} D_{x}w) \notag\\
&+ 2(\Delta_{x} v)(\Delta_{x} D_{x}w)
+ 4\sum_{1\leq i,j\leq 3} (\partial_{x_{i}x_{j}}v)(\partial_{x_{i}x_{j}}D_{x}w).
\end{align}
Similarly, it holds that 
\begin{align}
\Delta^2_{x}(w\mathbf{J})
={}& (\Delta^2_{x} w)\mathbf{J} + w\,\Delta^2_{x} \mathbf{J}  \notag\\
&+ 4\sum_{i=1}^3 (\partial_{x_{i}} w)(\partial_{x_i} \Delta_{x} \mathbf{J})
+ 4\sum_{i=1}^3 (\partial_{x_{i}} \Delta_{x} w)(\partial_{x_{i}} \mathbf{J}) \notag \\
&+ 2(\Delta_{x} w)(\Delta_{x}\mathbf{J})
+ 4\sum_{1\leq i,j\leq 3} (\partial_{x_ix_j}w)(\partial_{x_ix_j}\mathbf{J}). \label{Delta uJ formula}   
\end{align}
Moreover, using that $-\Delta_{x}V=\delta_{0}$ we see that 
\begin{align*}
\Delta_{x}^{2}\nabla_{x}V\ast \rho=-\nabla_{x}\Delta_{x}\rho.  \end{align*}
Consequently, taking the bi-Laplacian $\Delta^{2}_{x}$ in \eqref{eq for w} we obtain 
\begin{align*}
&\partial_{t}\Delta_{x}^{2}w+ (\Delta^2_{x} v)D_{x}w + v\,\Delta^2_{x} D_{x}w \notag\\
&+ 4\sum_{i=1}^3 (\partial_{x_{i}} v)(\partial_{x_i} \Delta_{x} D_{x}w)
+ 4\sum_{i=1}^3 (\partial_{x_i} \Delta_{x} v)(\partial_{x_{i}} D_{x}w) \notag\\
&+ 2(\Delta_{x} v)(\Delta_{x} D_{x}w)
+ 4\sum_{1\leq i,j\leq 3} (\partial_{x_{i}x_{j}}v)(\partial_{x_{i}x_{j}}D_{x}w)\\
&+(\Delta^2_{x} w)\mathbf{J} + w\,\Delta^2_{x} \mathbf{J}  \notag\\
&+ 4\sum_{i=1}^3 (\partial_{x_{i}} w)(\partial_{x_i} \Delta_{x} \mathbf{J})
+ 4\sum_{i=1}^3 (\partial_{x_{i}} \Delta_{x} w)(\partial_{x_{i}} \mathbf{J}) \notag \\
&+ 2(\Delta_{x} w)(\Delta_{x}\mathbf{J})
+ 4\sum_{1\leq i,j\leq 3} (\partial_{x_ix_j}w)(\partial_{x_ix_j}\mathbf{J})-\nabla_{x}\Delta_{x}\rho=0. 
\end{align*}
Multiplying by $ \Delta^{2}_{x}w$  the above equation we get 
\begin{align}
 &\frac{1}{2}\frac{\dd}{\dd t}\left\Vert \Delta_{x}^{2}w(t,\cdot)\right\Vert_{2}^{2} \notag\\=&-\int_{\mathbb{R}^{3}}\Delta_{x}^{2}vD_{x}w\cdot \Delta_{x}^{2}w\ \dd x-\frac{1}{2}\int_{\mathbb{R}^{3}}v\cdot \nabla_{x}\left\vert \Delta^{2}_{x}w\right\vert^{2}\ \dd x \notag\\
 &-4\sum_{i=1}^{3}\int_{\mathbb{R}^{3}}(\partial_{x_{i}}v)\cdot (\partial_{x_{i}}\Delta_{x}D_{x}w)\Delta^{2}_{x}w\ \dd x-4\sum_{i=1}^{3}\int_{\mathbb{R}^{3}}(\partial_{x_{i}}\Delta_{x}v)(\partial_{x_{i}}D_{x}w)\Delta^{2}_{x}w\ \dd x \notag\\
&-2\int_{\mathbb{R}^{3}}(\Delta_{x}v)(\Delta_{x}D_{x}w)\Delta^{2}_{x}w \ \dd x-4\sum_{i=1}^{3}\int_{\mathbb{R}^{3}}(\partial_{x_{i}x_{j}}v)\cdot (\partial_{x_{i}x_{j}}D_{x}w)\Delta^{2}_{x}w \ \dd x \notag\\
&-\int_{\mathbb{R}^{3}}w\Delta_{x}^{2}\mathbf{J}\cdot \Delta^{2}_{x}w \ \dd x -4\sum_{i=1}^{3}\int_{\mathbb{R}^{3}}(\partial_{x_{i}}w)(\partial_{x_{i}}\Delta_{x} \mathbf{J})\cdot \Delta_{x}^{2}w\ \dd x-4\sum_{i=1}^{3}\int_{\mathbb{R}^{3}}(\partial_{x_{i}}\Delta_{x}w)(\partial_{x_{i}}\mathbf{J})\cdot \Delta_{x}^{2}w\ \dd x \notag\\
&-2\int_{\mathbb{R}^{3}}(\Delta_{x}w)\Delta_{x}\mathbf{J}\cdot\Delta_{x}^{2}w \ \dd x-4\sum_{1\leq i,j\leq 3}\int_{\mathbb{R}^{3}}(\partial_{x_{i}x_{j}}w)\partial_{x_{i}x_{j}}\mathbf{J}\cdot\Delta_{x}^{2}w \ \dd x+\int_{\mathbb{R}^{3}}\nabla_{x}\Delta_{x}\rho\cdot \Delta^{2}_{x}w \ \dd x.  \label{rhs time der of bilap}
\end{align}
We abbreviate by $\{I_{k}\}_{k=1}^{6}$ the first $6$ terms in \eqref{rhs time der of bilap} and by $\{J_{k}\}_{k=1}^{6}$ the 6 last terms in \eqref{rhs time der of bilap}. Then, we may write concisely 
\begin{align*}
\frac{1}{2}\frac{\dd}{\dd t}\left\Vert \Delta^{2}_{x}w(t,\cdot)\right\Vert_{2}^{2}=\sum_{k=1}^{6}I_{k}+\sum_{k=1}^{6}J_{k}.      
\end{align*}
We proceed by estimating the terms $I_{k}$ and $J_{k}$. \\ 
\textbf{Estimate on the $I_{k}$'s.} Thanks to the Sobolev embedding $H^{2}(\mathbb{R}^{3})\hookrightarrow L^{\infty}(\mathbb{R}^{3})$ and the interpolation estimate 
$\left\Vert w\right\Vert_{H^{4}} \leq C\left(\left\Vert w\right\Vert_{2}+\left\Vert \Delta^{2}_{x}w\right\Vert_{2} \right)$
we have 
\begin{align*}
\left\Vert D_{x}w(t,\cdot)\right\Vert_{\infty}\leq C\left\Vert w(t,\cdot)\right\Vert_{H^{3}}\leq C\left(\left\Vert w(t,\cdot)\right\Vert_{2}+\left\Vert \Delta^{2}_{x}w(t,\cdot)\right\Vert_{2}  \right).       
\end{align*}
Therefore we can estimate $I_{1}$ as \begin{align}
I_{1}&\leq \left\Vert \Delta_{x}^{2}v(t,\cdot)\right\Vert_{2}^{2}\left\Vert D_{x}w(t,\cdot)\right\Vert_{\infty}^{2}+\left\Vert \Delta^{2}_{x}w(t,\cdot)\right\Vert_{2}^{2} \notag\\
&\leq CM(\left\Vert w(t,\cdot)\right\Vert_{2}^{2}+\left\Vert \Delta_{x}^{2}w(t,\cdot)\right\Vert_{2}^{2})+\left\Vert \Delta_{x}^{2}w(t,\cdot)\right\Vert_{2}^{2}\leq CM\left(\left\Vert w(t,\cdot)\right\Vert_{2}^{2}+\left\Vert \Delta^{2}_{x}w(t,\cdot)\right\Vert_{2}^{2}  \right) \label{I1 est}.       
\end{align}
To estimate $I_{2}$ and $I_{3}$ we use that $\left\Vert D_{x}v\right\Vert_{\infty}\leq C\left\Vert v\right\Vert_{L^{\infty}_{t}H^{3}_{x}}\leq CM$ to find that 
\begin{align}
\int_{\mathbb{R}^{3}}v\cdot \nabla_{x}\left\vert \Delta_{x}^{2}w\right\vert^{2}\ \dd x=\int_{\mathbb{R}^{3}}\mathrm{div}_{x}(v)\left\vert \Delta^{2}_{x}w\right\vert^{2}\ \dd x\leq \left\Vert \mathrm{div}_{x}v\right\Vert_{\infty}\left\Vert \Delta_{x}^{2}w(t,\cdot)\right\Vert_{2}^{2}\leq CM\left\Vert \Delta_{x}^{2}w(t,\cdot)\right\Vert_{2}^{2}  \label{I2 est}        
\end{align}
and 
\begin{align}
I_{3}&\leq C\left\Vert D_{x}v\right\Vert_{\infty}\left(\left\Vert D_{x}^{2}w(t,\cdot)\right\Vert_{2}^{2} +\left\Vert \Delta_{x}^{2}w(t,\cdot)\right\Vert_{2}^{2} \right) \notag\\
&\leq C\left\Vert D_{x}v\right\Vert_{\infty}\left(\left\Vert w(t,\cdot)\right\Vert_{2}^{2}+\left\Vert \Delta^{2}_{x}w(t,\cdot)\right\Vert_{2}^{2}  \right)\leq CM\left(\left\Vert w(t,\cdot)\right\Vert_{2}^{2}+\left\Vert\Delta^{2}_{x}w(t,\cdot) \right\Vert^{2}_{2}  \right). \label{I3 est}      
\end{align}
To estimate $I_{4}$, we have 
\begin{align*}
I_{4}&\leq C\left(\left\Vert \Delta_{x}^{2}w(t,\cdot)\right\Vert_{2}^{2}+\left\Vert D_{x}^{2}w\right\Vert_{\infty}^{2}\left\Vert D_{x}\Delta_{x}v(t,\cdot)\right\Vert_{2}^{2}\right)\\
&\leq C\left(\left\Vert \Delta_{x}^
{2}w(t,\cdot)\right\Vert_{2}^{2}+CM(\left\Vert w(t,\cdot)\right\Vert_{2}^{2}+\left\Vert \Delta^{2}_{x}w(t,\cdot)\right\Vert_{2}^{2})\right)\leq CM\left(\left\Vert w(t,\cdot)\right\Vert_{2}^{2}+\left\Vert \Delta_{x}^{2}w(t,\cdot)\right\Vert_{2}^{2}  \right).    
\end{align*}
As for $I_{5}$ and $I_{6}$, they are estimated by 
\begin{align}
I_{5}&\leq C\left\Vert \Delta_{x}v\right\Vert_{\infty}\int_{\mathbb{R}^{3}}\left\vert \Delta_{x}D_{x}w\right\vert\left\vert \Delta^{2}_{x}w\right\vert \ \dd x \notag \\
&\leq CM(\left\Vert \Delta_{x}D_{x}w(t,\cdot)\right\Vert_{2}^{2}+\left\Vert \Delta^{2}_{x}w(t,\cdot)\right\Vert_{2}^{2} )\leq CM(\left\Vert w(t,\cdot)\right\Vert_{2}^{2}+\left\Vert \Delta_{x}^{2}w(t,\cdot)\right\Vert_{2}^{2}) \label{I5 est}     
\end{align}
and similarly 
\begin{align}
I_{6}\leq CM\left(\left\Vert w(t,\cdot)\right\Vert_{2}^{2}+\left\Vert \Delta^{2}_{x}w(t,\cdot)\right\Vert_{2}^{2}\right). \label{I6 est}    
\end{align}
Gathering \eqref{I1 est}-\eqref{I6 est} we obtain 
\begin{align}
\sum_{k=1}^{6}I_{k}\leq CM\left(\left\Vert w(t,\cdot)\right\Vert_{2}^{2}+\left\Vert \Delta_{x}^{2}w(t,\cdot)\right\Vert_{2}^{2}  \right).\label{final est on sum of Ik}     
\end{align}
\textbf{Estimate on the $J_{k}$'s.} We have 
\begin{align}
J_{1}\leq \left\Vert \Delta_{x}^{2}\mathbf{J}\right\Vert_{\infty}\int_{\mathbb{R}^{3}}\left\vert w \right\vert\left\vert \Delta^{2}_{x}w\right\vert\ \dd x&\leq \frac{1}{2}\left\Vert \Delta^{2}_{x}\mathbf{J}\right\Vert_{\infty}\left(\left\Vert w(t,\cdot)\right\Vert_{2}^{2}+\left\Vert \Delta^{2}_{x}w(t,\cdot)\right\Vert^{2}_{2}   \right) \notag\\
&\leq C\left(\left\Vert w(t,\cdot)\right\Vert_{2}^{2}+\left\Vert \Delta^{2}_{x}w(t,\cdot)\right\Vert^{2}_{2}   \right).\label{J1est}         
\end{align}
To bound $J_{2}$, we have
\begin{align}
J_{2}&\leq C\left\Vert D_{x}\Delta_{x}\mathbf{J}\right\Vert_{\infty} \int_{\mathbb{R}^{3}}\left\vert D_{x}w\right\vert\left\vert \Delta_{x}^{2}w\right\vert\ \dd x \notag\\
&\leq C\left\Vert D_{x}\Delta_{x}\mathbf{J}\right\Vert_{\infty}\left(\left\Vert D_{x}w(t,\cdot)\right\Vert_{2}^{2}+\left\Vert \Delta^{2}_{x}w(t,\cdot)\right\Vert_{2}^{2}  \right)\leq C\left(\left\Vert w(t,\cdot)\right\Vert_{2}^{2}+\left\Vert 
\Delta^{2}_{x}w(t,\cdot)\right\Vert^{2}_{2}  \right). \label{J2 est}     
\end{align}
We can estimate $J_{3},J_{4},J_{5}$ similarly. We get  
\begin{align}
J_{3}\leq \left\Vert D_{x}\mathbf{J}\right\Vert_{\infty}\left(\left\Vert D_{x}\Delta_{x}w(t,\cdot)\right\Vert_{2}^{2}+\left\Vert \Delta_{x}^{2}w(t,\cdot)\right\Vert_{2}^{2}  \right)\leq C\left(\left\Vert w(t,\cdot)\right\Vert_{2}^{2}+\left\Vert \Delta_{x}^{2}w(t,\cdot)\right\Vert_{2}^{2}\right),     
\end{align}
\begin{align}
 J_{4}\leq \left\Vert \Delta_{x}\mathbf{J}\right\Vert_{\infty}\left(\left\Vert \Delta_{x}w(t,\cdot)\right\Vert_{2}^{2}+\left\Vert \Delta^{2}_{x}w(t,\cdot)\right\Vert_{2}^{2}  \right)\leq C\left(\left\Vert w(t,\cdot)\right\Vert_{2}^{2} +\left\Vert \Delta_{x}^{2}w(t,\cdot)\right\Vert_{2}^{2} \right) \label{J_{4} est}   
\end{align}
and 
\begin{align}
J_{5}\leq C\left\Vert D_{x}^{2}\mathbf{J}\right\Vert_{\infty} \left(\left\Vert D_{x}^{2}w(t,\cdot)\right\Vert_{2}^{2}+\left\Vert \Delta^{2}_{x}w(t,\cdot)\right\Vert_{2}^{2}  \right)\leq C\left(\left\Vert w(t,\cdot)\right\Vert_{2}^{2}+\left\Vert \Delta^{2}_{x}w(t,\cdot)\right\Vert^{2}_{2}  \right). \label{J5 est}    
\end{align}
Finally, to estimate $J_{6}$ we note that  
\begin{align*}
J_{6}\leq \frac{1}{2}\left(\left\Vert \nabla_{x}\Delta_{x}\rho(t,\cdot)\right\Vert_{2}^{2}+\left\Vert \Delta^{2}_{x}w(t,\cdot)\right\Vert_{2}^{2}  \right).     
\end{align*}
Thus, in view of Lemma \ref{est of H3 norm of rho} we get 
\begin{align}
J_{6}\leq Ce^{MCT}+\left\Vert \Delta_{x}^{2}w(t,\cdot)\right\Vert_{2}^{2}.  \label{est J6}    
\end{align}
Gathering \eqref{J1est}-\eqref{est J6} we get 
\begin{align}
\sum_{k=1}^{6}J_{k}\leq C\left(\left\Vert w(t,\cdot)\right\Vert_{2}^{2}+\left\Vert\Delta^{2}_{x}w(t,\cdot) \right\Vert_{2}^{2}  \right)+Ce^{MCT}.  \label{final est on sum of Jk}   
\end{align}
Set $S(t)\coloneqq \left\Vert w(t,\cdot)\right\Vert_{2}^{2}+\left\Vert \Delta_{x}^{2}w(t,\cdot)\right\Vert_{2}^{2}$. Combining \eqref{time der of Lp norm of w},\eqref{final est on sum of Ik} and \eqref{final est on sum of Jk} we get the following estimate 
\begin{align*}
\frac{\dd}{\dd t}S(t)\leq CMS(t)+Ce^{MCT}.     
\end{align*}
Integrating in time yields 
\begin{align*}
S(t)\leq S(0)+MC\int_{0}^{t}S(\tau)\ \dd \tau+TMCe^{MCT}.    
\end{align*}
Maximizing in time, it follows that 
\begin{align*}
\underset{t\in [0,T]}{\mathrm{ess\sup}}\ S(t)\leq S(0)+MCT\underset{t\in [0,T]}{\mathrm{ess\sup}}\ S(t)+TMCe^{MCT}.     
\end{align*}
Choose $T>0$ so small and $M$ so large so that 
\begin{itemize}
    \item $M>4 S(0)$.
    \item $MCT\leq \frac{1}{16}$ and $TCe^{MCT}\leq \frac{1}{16}$.
\end{itemize}
For this choice we get 
\begin{align*}
\frac{1}{2}\underset{t\in [0,T]}{\mathrm{ess\sup}}S(t)\leq S(0)+\frac{1}{16}M\leq \frac{M}{4}+\frac{M}{16}=\frac{5}{16}M,     
\end{align*}
so that 
\begin{align*}
\left\Vert w\right\Vert_{L^{\infty}([0,T];L^{2}(\mathbb{R}^{3}))}^{2}+\left\Vert \Delta^{2}_{x} w(t,\cdot)\right\Vert^{2}_{L^{\infty}([0,T];L^{2}(\mathbb{R}^{3}))}
=\underset{t\in [0,T]}{\mathrm{ess\sup}}\ S(t)\leq \frac{5}{8}M.     
\end{align*}
Hence, it follows that for this choice of $T$ and $M$ we have $w[v]\in \mathfrak{X}_{T,M}$. 
\\ 
\textbf{Step 3}. In this step we establish that $v\mapsto w[v]$ is a contraction. Given $v_{1},v_{2}\in \mathfrak{X}_{T,M}$ (with $T$ and $M$ as in step 1) denote $w_{1}\coloneqq w[v_{1}],w_{2}\coloneqq w[v_{2}]$ and $\rho_{1}\coloneqq \rho[v_{1}],\rho_{2}=\rho[v_{2}]$. The equation for the difference $w_{1}-w_{2}$ reads 
\begin{align*}
\partial_{t}(w_{1}-w_{2})+v_{1} D_{x}w_{1}-v_{2} D_{x}w_{2}+(w_{1}-w_{2})\mathbf{J}+\nabla_{x}V\ast (\rho_{1}-\rho_{2})=0.     
\end{align*}
Multiplying by $w_{1}-w_{2}$ and integrating we obtain 
\begin{align*}
\frac{\dd}{\dd t}\frac{1}{2}\left\Vert (w_{1}-w_{2})(t,\cdot)\right\Vert_{2}^{2}=&-\frac{1}{2}\int_{\mathbb{R}^{3}}v_{1}\cdot\nabla_{x}\left\vert w_{1}-w_{2}\right\vert^{2}(t,x)\ \dd x\\
&-\int_{\mathbb{R}^{3}}(v_{1}-v_{2}) D_{x}w_{2}\cdot(w_{1}-w_{2})(t,x)\ \dd x    \\
&-\int_{\mathbb{R}^{3}}\nabla_{x}V\ast (\rho_{1}-\rho_{2})\cdot (w_{1}-w_{2})(t,x)\ \dd x\\
=&\frac{1}{2}\int_{\mathbb{R}^{3}}\mathrm{div}_{x}(v_{1})\left\vert w_{1}-w_{2}\right\vert^{2}(t,x)\ \dd x\\
&-\int_{\mathbb{R}^{3}}(v_{1}-v_{2})D_{x}w_{2}\cdot (w_{1}-w_{2})(t,x)\ \dd x\\
&-\int_{\mathbb{R}^{3}}\nabla_{x}V\ast (\rho_{1}-\rho_{2})\cdot (w_{1}-w_{2})(t,x)\ \dd x.
\end{align*}
Therefore, using that $\left\Vert \mathrm{div}_{x}(v_{1})\right\Vert_{\infty}\leq CM $ and that $\left\Vert \nabla_{x}V\ast(\rho_{1}-\rho_{2})(t,\cdot)\right\Vert_{2}^{2}=\left\Vert (\rho_{1}-\rho_{2})(t,\cdot)\right\Vert_{\dot{H}^{-1}}^{2}$ we obtain 
\begin{align*}
\frac{\dd}{\dd t}\left\Vert (w_{1}-w_{2})(t,\cdot)\right\Vert_{2}^{2}&\leq CM\left\Vert (w_{1}-w_{2})(t,\cdot)\right\Vert_{2}^{2}\\
&+\left\Vert D_{x}w_{2}\right\Vert_{\infty} (\left\Vert (w_{1}-w_{2})(t,\cdot)\right\Vert_{2}^{2}+\left\Vert (v_{1}-v_{2})(t,\cdot)\right\Vert_{2}^{2})\\
&+\left\Vert (\rho_{1}-\rho_{2})(t,\cdot)\right\Vert_{\dot{H}^{-1}}^{2}+\left\Vert (w_{1}-w_{2})(t,\cdot)\right\Vert_{2}^{2}.
\end{align*}
Integrating in time and maximizing over $[0,T]$ we get 
\begin{align}
\left\Vert w_{1}-w_{2}\right\Vert^{2}_{L^{\infty}_{t}L^2_{x}}&\leq CMT\left\Vert w_{1}-w_{2}\right\Vert_{L^{\infty}_{t}L^2_{x}}^{2} \notag\\
&+T\left\Vert D_{x}w_{2}\right\Vert_{\infty} \left(\left\Vert w_{1}-w_{2}\right\Vert_{L^{\infty}_{t}L^{2}_{x}}^{2} +\left\Vert v_{1}-v_{2}\right\Vert_{L^{\infty}_{t}L^2_{x}}^{2}\right) \notag\\
&+T\left\Vert \rho_{1}-\rho_{2}\right\Vert_{L^{\infty}_{t}\dot{H}^{-1}_{x}}^{2}+T\left\Vert w_{1}-w_{2}\right\Vert_{L^{\infty}_{t}L^2_{x}}^{2}. \label{L2 norm w1w2}     
\end{align}
By Lemma \ref{Stability H-1} we have that $$T\left\Vert \rho_{1}-\rho_{2}\right\Vert_{L^{\infty}_{t}\dot H^{-1}_{x}}^{2}\leq CT(MT+Te^{MCT})\left\Vert \rho_{1}-\rho_{2}\right\Vert_{L^{\infty}_{t}\dot{H}^{-1}_{x}}^{2} +CT^{2}e^{MCT}\left\Vert v_{1}-v_{2}\right\Vert_{L^{\infty}_{t}L^2_{x}}^{2}.$$
Therefore for $T$ sufficiently small (for instance 
$T$ such that $C(MT+Te^{MCt})\leq \frac{1}{2}$) we get 
\begin{align}
T\left\Vert \rho_{1}-\rho_{2}\right\Vert_{L^{\infty}_{t}\dot{H}^{-1}_{x}}^{2}\leq CT^{2}e^{MCT}\left\Vert v_{1}-v_{2}\right\Vert_{L^{\infty}_{t}L^2_{x}}^{2}.  \label{ine for diff of rho}     
\end{align}
Moreover, by step 2 and Sobolev embedding we have 
\begin{align}
\left\Vert D_{x}w_{2}\right\Vert_{\infty}\leq C\left\Vert w_{2}\right\Vert_{L^{\infty}_{t}H^{4}_{x}}\leq CM . \label{ine for Dxw2}    
\end{align}
Substituting \eqref{ine for diff of rho}-\eqref{ine for Dxw2} inside \eqref{L2 norm w1w2} we obtain 
\begin{align*}
\left\Vert w_{1}-w_{2}\right\Vert_{L^{\infty}_{t}L^2_{x}}^{2}&\leq CMT\left\Vert w_{1}-w_{2}\right\Vert^{2}_{L^{\infty}_{t}L^{2}_{x}}+CMT\left\Vert v_{1}-v_{2}\right\Vert_{L^{\infty}_{t}L^{2}_{x}}^{2}\\
&+CT^{2}e^{MCT}\left\Vert v_{1}-v_{2}\right\Vert_{L^{\infty}_{t}L^{2}_{x}}^{2}. 
\end{align*}
It follows that for $T$ sufficiently small it holds that 
\begin{align*}
\left\Vert w_{1}-w_{2}\right\Vert_{L^{\infty}_{t}L^2_{x}}^{2}\leq (CT^{2}e^{MCT}+CMT)\left\Vert v_{1}-v_{2}\right\Vert_{L^{\infty}_{t}L^{2}_{x}}^{2}.      
\end{align*}
Further shrinking $T$ if necessary as to ensure that $CT^{2}e^{MCT}+CMT<1$ we get 
\begin{align*}
\left\Vert w_{1}-w_{2}\right\Vert_{L^{\infty}_{t}L^2_{x}}\leq \lambda \left\Vert v_{1}-v_{2}\right\Vert_{L^{\infty}_{t}L^2_{x}}  \ \mbox{for }\ 0<\lambda<1.     
\end{align*}
By the Banach contraction principle we deduce there exist a unique fixed point  $u\in \mathfrak{X}_{T,M}$. So $(\rho[u],u)$ is the asserted solution.
\end{proof}
\begin{rem}
Due to the interpolation inequality $\left\Vert \Delta_{x} u\right\Vert_{\infty}\leq C\left(\left\Vert \Delta_{x} u\right\Vert_{2} +\left\Vert \Delta^{2}_{x}u\right\Vert_{2} \right)$ we see that the solution of Theorem \ref{well posedness of EPA} satisfies $\Delta_{x}u\in L^{\infty}([0,T];L^{\infty}(\mathbb{R}^{3}))$. We will need the velocity field $u$ to have bounded Laplacian when proving the semiclassical and mean-field semiclassical limit, which is why it is convenient to work subject to the assumption $(\rho^{\mathrm{in}},u^{\mathrm{in}})\in H^{3}(\mathbb{R}^{3})\times H^{4}(\mathbb{R}^{3})$. Note also that since $\rho$ is governed by a transport equation, it remains a probability density for all times on which the solution is defined. 
\end{rem}

\section{The semiclassical limit}\label{semiclassical limit section}
We start by recalling that the quantum current $J_{\hbar}$ given in \eqref{single body current mag} can be equivalently defined by duality as the unique vector valued Radon measure such that for all $a\in W^{1,\infty}(\mathbb{R}^{3};\mathbb{R}^{3})$ it holds that   
\begin{align*}
\int_{\mathbb{R}^{3}}J_{\hbar}(t,x)\cdot a(x)\ \dd x=-\frac{1}{2}\mathrm{tr}\left((i\hbar\nabla_{x}+A)\vee a R_{\hbar}(t)\right)=
-\frac{1}{2}\sum_{j=1}^{3}\mathrm{tr}\left((i\hbar\partial_{x_{j}}+A_{j})\vee a_{j} R_{\hbar}(t)\right).     
\end{align*}
In fact, this definition motivated the definition of $J_{\hbar,N:1}$ in \eqref{defintion of N body current}. If $\psi_{\hbar}$ is a solution to \eqref{SP equation intro} consider the time-dependent density operator $R_{\hbar}(t)=\vert \psi_{\hbar}(t)\rangle  \vert\langle \psi_{\hbar}(t)\vert$. As explained in the introduction, the 1-body density operator $R_{\hbar}(t)$ is governed by Hartree's equation, which in the magnetic settings reads 
\begin{align}\tag{HEA}
i\hbar \partial_{t}R_{\hbar}=\left[\mathscr{H}_{\hbar,A},R_{\hbar}(t)\right]\ \mbox{where}\ \mathscr{H}_{\hbar,A}=\frac{1}{2}(i\hbar \nabla_{x}+A)^{2}+V\ast \left\vert \psi_{\hbar}\right\vert^{2}. \label{Hartree magnetic}     
\end{align}
It will be convenient to work with Hartree's equation \eqref{Hartree magnetic} when calculating the time derivative of $\mathcal{E}_{\hbar}(t)$. Furthermore note that the quantum modulated energy may be recast as
\begin{align*}
\mathcal{E}_{\hbar}(t)=\mathrm{tr}\left((i\hbar\nabla_{x}+A+u  )^{2}R_{\hbar}(t)\right)+\int_{\mathbb{R}^{3}}V\ast(\rho_{\hbar}-\rho)(\rho_{\hbar}-\rho)(t,x)\ \dd x.      
\end{align*}
\begin{rem}\label{remark about eq qunatum density}
By exactly the same calculation in Lemma 3.3 in \cite{golse2022mean} one proves that $$\partial_{t}\rho_{\hbar}+\mathrm{div}_{x}J_{\hbar}=0.$$
By the same token, it holds that  
\begin{align*}
\partial_{t}\rho_{\hbar,N:1}+\mathrm{div}_{x}J_{\hbar,N:1}=0.     
\end{align*}
These equations will be used when calculating $\frac{\dd}{\dd t}\mathcal{E}_{\hbar}(t)$ and $\frac{\dd}{\dd t}\mathcal{E}_{\hbar,N}(t)$.     
\end{rem}
We aim to calculate the time derivative of the quantum modulated energy. This is the main step in proving the semiclassical limit. We henceforth denote by $\Pi_{j}$ the $j$-th component of the magnetic momentum operator, i.e.  
\begin{align}
\Pi_{j}\coloneqq i\hbar\partial_{x_{j}}+A_{j}. \label{magnetic momentum operator def}    
\end{align}
The following lemma is a key component in ensuring the terms contributed by $A$ will eventually cancel out in the calculation of $\mathcal{E}_{\hbar}(t)$. 
\begin{lem} \label{cancellation lemma}
Suppose that: 
\begin{itemize}
    \item $x\mapsto\mathbf{J}(x)\in M_{3\times 3}(\mathbb{R})$ is anti-symmetric for all $x\in \mathbb{R}^{3}$.
 \item  $u:\mathbb{R}^{3}\rightarrow \mathbb{R}^{3}$ is a given vector field. 
 \item $\Pi_{j}$ is given by \eqref{magnetic momentum operator def}.
\end{itemize}
Then, it holds that 
\begin{align*}
\sum_{1\leq j,k\leq 3}(\Pi_{k}\vee \mathbf{J}_{jk})\vee u_{j}=\sum_{1\leq j,k\leq 3} \Pi_{j}\vee (\mathbf{J}_{kj}\vee u_{k})    
\end{align*}
\end{lem}
\begin{proof}
Since $\mathbf{J}$ is anti-symmetric we have 
\begin{align}
\sum_{1\leq j,k\leq 3}(\Pi_{k}\vee \mathbf{J}_{jk})\vee u_{j}=\sum_{1\leq j,k\leq 3}(\Pi_{j}\vee \mathbf{J}_{kj})\vee u_{k}=-\sum_{1\leq j,k\leq 3}(\Pi_{j}\vee \mathbf{J}_{jk})\vee u_{k}. \label{identity1}    
\end{align}
On the other hand, recall the following formula for the anti-commutator 
\begin{align*}
T\vee (S\vee Q)=(T\vee S)\vee Q+\left[S,[Q,T]\right].     
\end{align*}
Therefore, we have that  
\begin{align*}
\Pi_{j}\vee (\mathbf{J}_{kj}\vee u_{k})=(\Pi_{j}\vee \mathbf{J}_{kj})\vee u_{k} +\left[\mathbf{J}_{kj},[u_{k},\Pi_{j}]\right]=-(\Pi_{j}\vee \mathbf{J}_{jk})\vee u_{k}+\left[\mathbf{J}_{kj},[u_{k},\Pi_{j}]\right].\end{align*}
Since $\left[u_{k},\Pi_{j}\right]=-i\hbar\partial_{x_{j}}u_{k}$ it follows that $\left[\mathbf{J}_{kj},\left[u_{k},\Pi_{j}\right]\right]=0$, as a commutator of multiplication operators. Hence we get 
\begin{align}
\sum_{1\leq j,k\leq 3}\Pi_{j}\vee (\mathbf{J}_{kj}\vee u_{k})=-\sum_{1\leq j,k\leq 3}(\Pi_{j}\vee \mathbf{J}_{jk})\vee u_{k}. \label{identity2}    
\end{align}
Comparing \eqref{identity1} with \eqref{identity2} yields the announced result. 
\end{proof}
\begin{thm}\label{thm time derivative of modulated energy}
Let the assumptions of Theorem \ref{main thm semi classical} hold. Let $\psi_{\hbar}$ be the unique solution to \eqref{SP equation intro} with initial data $\psi_{\hbar}^{\mathrm{in}}$ and let  $(\rho,u)$ be the unique solution to \eqref{Magnetized Euler Poisson velocity density intro} with initial data $(\rho^{\mathrm{in}},u^{\mathrm{in}})$ on $[0,T]$. Let $\mathcal{E}_{\hbar}(t)$ be given by \eqref{quantum modulated energy}.  Then, for all $t\in [0,T]$ it holds that 
\begin{align*}
\frac{\dd}{\dd t}\mathcal{\mathcal{E}}_{\hbar}(t)=&-\frac{1}{2}\sum_{1\leq j,k\leq 3}\mathrm{tr}\big((\Pi_{j}+u_{j})\vee((\Pi_{k}+u_{k})\vee(\partial_{x_{k}}u_{j}))R_{\hbar}(t)\big)\\
&+2\int_{\mathbb{R}^{3}}u(t,x)\cdot \nabla_{x} V\ast(\rho_{\hbar}-\rho)(t,x)(\rho_{\hbar}-\rho)(t,x)\ \dd x.     
\end{align*}
\end{thm}
\begin{proof}
\textbf{1}.
Note that
\begin{align*}
\mathcal{E}_{\hbar}(t)=\mathcal{F}_{\hbar}(t)&+\sum_{j=1}^{3}\mathrm{tr}\left((i\hbar\partial_{x_{j}}+A_{j})\vee u_{j}R_{\hbar}(t)\right)+\int_{\mathbb{R}^{3}}\left\vert u\right\vert^{2}\rho_{\hbar}(t,x)\ \dd x\\
&+\int_{\mathbb{R}^{3}}V\ast \rho (t,x)\rho(t,x)\ \dd x-2\int_{\mathbb{R}^{3}}V\ast \rho_{\hbar}(t,x)\rho(t,x)\ \dd x. 
\end{align*}
Therefore, by  conservation of energy (equation \eqref{Conservation of energy}),
we have 
\begin{align*}
\frac{\dd}{\dd t}\mathcal{E}_{\hbar}(t)=&\frac{\dd}{\dd t}\sum_{j=1}^{3}\mathrm{tr}\left((i\hbar\partial_{x_{j}}+A_{j})\vee u_{j}R_{\hbar}
(t)\right)+\frac{\dd}{\dd t}\int_{\mathbb{R}^{3}}\left\vert u\right\vert^{2}\rho_{\hbar}(t,x)\ \dd x\\
&+\frac{\dd}{\dd t}\int_{\mathbb{R}^{3}}V\ast \rho(t,x)\rho(t,x)\ \dd x-2\frac{\dd}{\dd t}\int_{\mathbb{R}^{3}}V\ast \rho_{\hbar}(t,x)\rho(t,x)\ \dd x\coloneqq \sum_{j=1}^{4}I^{j}(t). 
\end{align*}
\textbf{2}.  First, we assert that the following identity holds:
\begin{align}
I^{1}(t)+I^{2}(t)&=\sum_{j=1}^{3}\mathrm{tr}\big((u_{j}+\Pi_{j})\vee(\partial_{t}u_{j}+(uD_{x}u )_{j})R_{\hbar}(t)\big) \notag \\
&-\frac{1}{2}\sum_{1\leq j,k\leq 3}\mathrm{tr}\big((u_{j}+\Pi_{j})\vee((u_{k}+\Pi_{k})\vee\partial_{x_{k}}u^{j})R_{\hbar}(t)\big).    \notag\\
&+\frac{1}{2}\sum_{1\leq j,k\leq 3}\mathrm{tr}\left((\Pi_{k}\vee \mathbf{J}_{jk})\vee u_{j} R_{\hbar}(t)\right)
\notag\\&+2\int_{\mathbb{R}^{3}}\rho_{\hbar}(t,x)\,\nabla_{x} V\ast \rho_{\hbar}(t,x)\cdot u(t,x) \,\dd x.\label{eq:-2}
\end{align}
To prove \eqref{eq:-2} we compute: 
\begin{align}
I^{1}+I^{2}&=\frac{\dd}{\dd t}\sum_{j=1}^{3}\mathrm{tr}\left(\Pi_{j}\vee u_{j}R_{\hbar}(t)\right)+\frac{\dd}{\dd t}\mathrm{tr}\Big(\left\vert u\right\vert^{2}R_{\hbar}(t)\Big) \notag\\
&=\mathrm{tr}\Big(\Big(\sum_{j=1}^{3}\Pi_{j}\vee u_{j}+\left\vert u\right\vert^{2}\Big)\partial_{t}R_{\hbar}(t) \Big) \notag+\mathrm{tr}\Big(\partial_{t}\Big(\sum_{j=1}^{3}\Pi_{j}\vee u_{j}+\left\vert u\right\vert^{2} \Big)R_{\hbar}(t)\Big) \notag\\
&=\frac{1}{i\hbar}\mathrm{tr}\Big(\Big(\sum_{j=1}^{3}\Pi_{j}\vee u_{j}
+\left\vert u\right\vert^{2} \Big)\left[\mathscr{H}_{\hbar,A},R_{\hbar}(t)\right]\Big) \notag+\mathrm{tr}\Big(\partial_{t}\Big(\sum_{j=1}^{3}\Pi_{j}\vee u_{j}+\left\vert u\right\vert^{2} \Big)R_{\hbar}(t)\Big) \notag\\
&=\frac{i}{\hbar}\mathrm{tr}\Big(\Big[\mathscr{H}_{\hbar,A},\sum_{j=1}^{3}\Pi_{j}\vee u_{j}+\left\vert u\right\vert^{2} \Big]R_{\hbar}(t)\Big)
+\mathrm{tr}\Big(\partial_{t}\Big(\sum_{j=1}^{3}\Pi_{j}\vee u_{j}+\left\vert u\right\vert^{2} \Big)R_{\hbar}(t)\Big). \label{first calculation of I1I2}
\end{align}
The right-hand side of \eqref{first calculation of I1I2} reads 
\begin{align*}
&\sum_{j=1}^{3}\mathrm{tr}\Big(\Big(\partial_{t}+\frac{i}{\hbar}\Big[\frac{1}{2}\sum_{k=1}^{3}\Pi_{k}^{2},\cdot\Big]\Big)\Big(\Pi_{j}\vee u_{j}+\frac{1}{2}u_{j}\vee u_{j}\Big)R_{\hbar}(t)\Big
)\\
&+ \sum_{j=1}^{3}\mathrm{tr}\Big(\frac{i}{\hbar}\Big[V\ast \rho_{\hbar},\cdot\Big]\Big(\Pi_{j}\vee u_{j}+\frac{1}{2}u_{j}\vee u_{j} \Big)R_{\hbar}(t)\Big)\\
&\coloneqq \sum_{j=1}^{3}\mathrm{tr}\left(L_{1j}R_{\hbar}(t)\right)+\sum_{j=1}^{3}\mathrm{tr}\left(L_{2j}R_{\hbar}(t)\right). 
\end{align*}
We start with the term $L_{1j}$. Using the  Leibniz rule: $\left[A,B\vee C\right]=\left[A,B\right]\vee C+\left[A,C\right]\vee B,$ we can write $L_{1j}$ as follows: 
\begin{align}
L_{1j}=&\Big(\partial_{t}+\frac{i}{\hbar}\Big[\frac{1}{2}\sum_{k=1}^{3}\Pi_{k}^{2},\cdot\Big]\Big)\Big((\frac{1}{2}u_{j}+\Pi_{j})\vee u_{j}\Big) \notag\\
=&\Big(\Big(\partial_{t}+\frac{i}{\hbar}\Big[\frac{1}{2}\sum_{k=1}^{3}\Pi_{k}^{2},\cdot\Big]\Big)(\frac{1}{2}u_{j}+\Pi_{j})\Big)\vee u_{j} \notag\\
&+(\frac{1}{2}u_{j}+\Pi_{j})\vee \Big(\Big(\partial_{t}+\frac{i}{\hbar}\Big[\frac{1}{2}\sum_{k=1}^{3}\Pi_{k}^{2},\cdot\Big]\Big)u_{j}\Big) \notag\\
=&\Big(\Big(\partial_{t}+\frac{i}{\hbar}\Big[\frac{1}{2}\sum_{k=1}^{3}\Pi_{k}^{2},\cdot\Big]\Big)u_{j}\Big)\vee u_{j}+\Pi_{j}\vee \Big(\Big(\partial_{t}+\frac{i}{\hbar}\Big[\frac{1}{2}\sum_{k=1}^{3}\Pi_{k}^{2},\cdot\Big]\Big)u_{j}\Big) \notag\\
&+\frac{i}{\hbar}\Big[\frac{1}{2}\sum_{k=1}^{3}\Pi_{k}^{2},\Pi_{j}\Big]\vee u_{j}. \label{L1j rhs}
\end{align}
We compute that 
\begin{align*}
\left[\Pi_{k}^{2},\Pi_{j}\right]=\Pi_{k}\vee\left[\Pi_{k},\Pi_{j}\right]=\Pi_{k}\vee \left[i\hbar\partial_{x_{k}}+A_{k},i\hbar \partial_{x_{j}}+A_{j}\right]=\Pi_{k}\vee (i\hbar \partial_{x_{k}}A_{j}-i\hbar\partial_{x_{j}}A _{k}), 
\end{align*}
and hence we get   
\begin{align}
\frac{i}{\hbar}\Big[\frac{1}{2}\sum_{k=1}^{3}\Pi_{k}^{2},\Pi_{j}\Big]=\frac{1}{2}\sum_{k=1}^{3}\Pi_{k}\vee(\partial_{x_{j}}A_{k}-\partial_{x_{k}}A_{j})=\frac{1}{2}\sum_{k=1}^{3}\Pi_{k}\vee \mathbf{J}_{jk}.  \label{last term 1body}     
\end{align}
Inserting \eqref{last term 1body} inside \eqref{L1j rhs} we get 
\begin{align}
L_{1j}&=(u_{j}+\Pi_{j})\vee \Big(\Big(\partial_{t}+\frac{i}{\hbar}\left[\frac{1}{2}\sum_{k=1}^{3}\Pi_{k}^{2},\cdot\right]\Big)u_{j}\Big)+\frac{1}{2}\sum_{k=1}^{3}(\Pi_{k}\vee \mathbf{J}_{jk})\vee u_{j} \notag\\
&=(u_{j}+\Pi_{j})\vee \Big(\partial_{t}u_{j}+\sum_{k=1}^{3}u_{k}\partial_{x_{k}}u_{j}\Big)+(u_{j}+\Pi_{j})\vee \Big(\frac{i}{\hbar}\Big[\frac{1}{2}\sum_{k=1}^{3}\Pi_{k}^{2},u_{j}\Big]-\sum_{k=1}^{3}u_{k}\partial_{x_{k}}u_{j}\Big) \notag\\
&+\frac{1}{2}\sum_{k=1}^{3}(\Pi_{k}\vee \mathbf{J}_{jk})\vee u_{j}. \label{second right hand side L1j} 
\end{align}
Note that 
\begin{align*}
\frac{i}{2\hbar}\left[\Pi_{k}^{2},u_{j}\right]=\frac{i}{2\hbar}\Pi_{k}\vee \left[i\hbar\partial_{x_{k}},u_{j}\right]=-\frac{1}{2}\Pi_{k}\vee \partial_{x_{k}}u_{j}
\end{align*}
so that the second term in the right-hand side of \eqref{second right hand side L1j} reads
\begin{align*}
&-\sum_{k=1}^{3}(u_{j}+\Pi_{j})\vee (\frac{1}{2}\Pi_{k}\vee \partial_{x_{k}}u_{j}+u_{k}\partial_{x_{k}}u_{j})\\
&=-\sum_{k=1}^{3}(u_{j}+\Pi_{j})\vee (\frac{1}{2}\Pi_{k}\vee \partial_{x_{k}}u_{j}+\frac{1}{2}u_{k}\vee \partial_{x_{k}}u_{j})\\
&=-\frac{1}{2}\sum_{k=1}^{3}(u_{j}+\Pi_{j})\vee ((u_{k}+\Pi_{k})\vee \partial_{x_{k}}u_{j}).     
\end{align*}
So we find that 
\begin{align*}
L_{1j}=&(u_{j}+\Pi_{j})\vee \Big(\partial_{t}u_{j}+\sum_{k=1}^{3}u_{k}\partial_{x_{k}}u_{j}\Big)-\frac{1}{2}\sum_{k=1}^{3}(u_{j}+\Pi_{j})\vee \left((u_{k}+\Pi_{k})\vee \partial_{x_{k}}u_{j}\right)\\
+&\frac{1}{2}\sum_{k=1}^{3}(\Pi_{k}\vee \mathbf{J}_{jk})\vee u_{j}. 
\end{align*}
Therefore we conclude that
\begin{align}
\sum_{j=1}^{3}\mathrm{tr}\big(L_{1j}R_{\hbar}(t)\big)
&=\sum_{j=1}^{3}\mathrm{tr}\big((u_{j}+\Pi_{j})\vee(\partial_{t}u_{j}+(uD_{x}u )_{j})R_{\hbar}(t)\big) \notag \\
&-\frac{1}{2}\sum_{1\leq j,k\leq 3}\mathrm{tr}\big((u_{j}+\Pi_{j})\vee((u_{k}+\Pi_{k})\vee\partial_{x_{k}}u^{j})R_{\hbar}(t)\big) \notag\\
&+\frac{1}{2}\sum_{1\leq j,k\leq 3}\mathrm{tr}\left((\Pi_{k}\vee \mathbf{J}_{jk})\vee u_{j} R_{\hbar}(t)\right).
\label{eq for J1 1body}
\end{align}
As for $L_{2j}$, note the following identities
\begin{align*}
[V\ast \rho_{\hbar},u_{j}\vee u_{j}]=0,\qquad [V\ast \rho_{\hbar},A_{j}u_{j}]=0, \qquad  
-[V\ast \rho_{\hbar},\partial_{x_{j}}\vee u^{j}]=2\partial_{x_{j}}V\ast \rho_{\hbar}u^{j}.
\end{align*}
Consequently, we have 
\begin{align*}
L_{2j}=2\int_{\mathbb{R}^{3}}\rho_{\hbar}(t,x)\partial_{x_{j}}V\ast \rho_{\hbar}(t,x)\,u_j(t,x)\,\dd x.  \end{align*}
So it follows that 
\begin{align}
\sum_{j=1}^{3}\mathrm{tr}\big(L_{2,j}R_{\hbar}(t)\big)=2\int_{\mathbb{R}^{3}}\rho_{\hbar}(t,x)\nabla_{x} V\ast \rho_{\hbar}(t,x)\cdot u(t,x)\,\dd x.
\label{eq for J2 1body}
\end{align}
Combining \eqref{eq for J1 1body} with \eqref{eq for J2 1body}, we obtain \eqref{eq:-2}. 

\smallskip
\textbf{3}. Using Remark \ref{remark about eq qunatum density} we compute that 
\begin{align}
I^{3}(t)+I^{4}(t)&= \frac{\dd}{\dd t}\int_{\mathbb{R}^{3}}V\ast \rho(t,x)\rho(t,x)\ \dd x-2\frac{\dd}{\dd t}\int_{\mathbb{R}^{3}}V\ast \rho_{\hbar}(t,x)\rho(t,x)\ \dd x \notag\\
&=2\int_{\mathbb{R}^{3}}V\ast\rho(t,x)\partial_{t}\rho(t,x)\ \dd x-2\int_{\mathbb{R}^{3}}V\ast \rho_{\hbar}(t,x)\partial_{t}\rho(t,x)\ \dd x \notag\\
&-2\int_{\mathbb{R}^{3}}V\ast \partial_{t}\rho_{\hbar}(t,x)\rho(t,x)\ \dd x=2\int_{\mathbb{R}^{3}}\nabla_{x} V\ast \rho(t,x)\cdot\rho u(t,x)\ \dd x \notag\\
&-2\int_{\mathbb{R}^{3}}\nabla_{x} V\ast \rho_{\hbar}(t,x)\cdot \rho u(t,x) \ \dd x+2\int_{\mathbb{R}^{3}}V\ast \mathrm{div}_{x}(J_{\hbar})(t,x)\rho(t,x)\ \dd x. \label{I3I4 1body} 
\end{align}
In view of \eqref{eq:-2} and \eqref{I3I4 1body} we conclude that 
\begin{align}
I^{1}+I^{2}+I^{3}+I^{4}=&-\sum_{j=1}^{3}\mathrm{tr}((u_{j}+\Pi_{j})\vee (\partial_{x_{j}}V\ast \rho+(u\mathbf{J})_{j})R_{\hbar}(t))
\notag\\
&-\frac{1}{2}\sum_{1\leq j,k\leq 3}\mathrm{tr}\left((u_{j}+\Pi_{j})\vee ((\Pi_{k}+u_{k})\vee \partial_{x_{k}}u_{j})R_{\hbar}(t)\right)\notag\\
&+\frac{1}{2}\sum_{1\leq j,k\leq 3}\mathrm{tr}\left((\Pi_{k}\vee \mathbf{J}_{jk})\vee u_{j} R_{\hbar}(t)\right) \notag
\\
&+2\int_{\mathbb{R}^{3}}\rho_{\hbar}(t,x)\nabla_{x}V\ast \rho_{\hbar}(t,x)\cdot u(t,x)\ \dd x\notag\\
&+2\int_{\mathbb{R}^{3}}\nabla_{x}V\ast \rho(t,x)\cdot \rho u(t,x)\ \dd x\notag\\
&-2\int_{\mathbb{R}^{3}}\nabla_{x}V\ast \rho_{\hbar}(t,x)\cdot \rho u(t,x)\ \dd x+2\int_{\mathbb{R}^{3}}V\ast \mathrm{div}_{x}(J_{\hbar})(t,x)\rho(t,x)\ \dd x.
\label{I1I4 sum}
\end{align}
Since $\mathbf{J}$ is anti-symmetric we have $u\cdot (u\mathbf{J})=0$. Therefore, by Lemma \ref{cancellation lemma} we have   
\begin{align*}
&-\sum_{j=1}^{3}(u_{j}+\Pi_{j})\vee (u\mathbf{J})_{j}+\frac{1}{2}\sum_{1\leq j,k\leq 3}(\Pi_{k}\vee\mathbf{J}_{jk})\vee u_{j}\\
&=-\frac{1}{2}\sum_{1\leq j,k\leq 3}\Pi_{j}\vee (\mathbf{J}_{kj}\vee u_{k})+\frac{1}{2}\sum_{1\leq j,k\leq 3}(\Pi_{k}\vee \mathbf{J}_{jk})\vee u_{j}=0. 
\end{align*}
Hence, we get the following cancellation 
\begin{align}
-\sum_{1\leq j,k\leq 3}\mathrm{tr}\left((u_{j}+\Pi_{j})\vee (u\mathbf{J})_{j}R_{\hbar}(t)\right)+\frac{1}{2}\sum_{1\leq j,k\leq 3}\mathrm{tr}\left((\Pi_{k}\vee \mathbf{J}_{jk})\vee u_{j}R_{\hbar}(t)\right)=0. \label{Cancelation single}    
\end{align}
Plugging in \eqref{Cancelation single} inside \eqref{I1I4 sum} we deduce that 
\begin{align}
I^{1}+I^{2}+I^{3}+I^{4}=&-\sum_{j=1}^{3}\mathrm{tr}\left((u_{j}+\Pi_{j})\vee \partial_{x_{j}}V\ast \rho R_{\hbar}(t)\right) \notag\\
&-\frac{1}{2}\sum_{1\leq j,k\leq 3}\mathrm{tr}\left((u_{j}+\Pi_{j})\vee ((u_{k}+\Pi_{k})\vee \partial_{x_{k}}u_{j})R_{\hbar}(t)\right) \notag\\
&+2\int_{\mathbb{R}^{3}}\rho_{\hbar}(t,x)\nabla_{x}V\ast \rho_{\hbar}(t,x)\cdot u(t,x)\ \dd x \notag\\
&+2\int_{\mathbb{R}^{3}}\nabla_{x}V\ast \rho(t,x)\cdot \rho u(t,x)\ \dd x \notag\\
&-2\int_{\mathbb{R}^{3}}\nabla_{x}V\ast \rho_{\hbar}(t,x)\cdot \rho u(t,x)\ \dd x+2\int_{\mathbb{R}^{3}}V\ast \mathrm{div}_{x}(J_{\hbar})(t,x)\rho(t,x)\ \dd x. \label{sum of I1I4}
\end{align}
Integrating by parts we see that
\begin{align*}
2\int_{\mathbb{R}^{3}}V\ast\mathrm{div}_{x}(J_{\hbar})(t,x)\rho(t,x)\ \dd x&=-2\int_{\mathbb{R}^{3}}\nabla_{x} V\ast \rho(t,x)\cdot J_{\hbar}(t,x)\ \dd x\\
&=\sum_{j=1}^{3}\mathrm{tr}\left(\Pi_{j}\vee \partial_{x_{j}}V\ast \rho R_{\hbar}(t)\right).     
\end{align*}
Therefore we have 
\begin{align}
&-\sum_{j=1}^{3}\mathrm{tr}\left(\Pi_{j}\vee \partial_{x_{j}}V\ast \rho R_{\hbar}(t)\right)+2\int_{\mathbb{R}^{3}}V\ast \mathrm{div}_{x}(J_{\hbar})(t,x)\rho(t,x)\ \dd x=0. 
\label{vanishing}   \end{align}
In addition, observe  that 
\begin{align*}
-\sum_{j=1}^{3}\mathrm{tr}\left(u_{j}\vee \partial_{x_{j}}V\ast \rho R_{\hbar}(t)\right)=-2\int_{\mathbb{R}^{3}}\nabla_{x} V\ast \rho(t,x)\cdot u(t,x)\rho_{\hbar}(t,x)\ \dd x.    
\end{align*}
Consequently, we obtain  
\begin{align}
&\sum_{j=1}^{3}\mathrm{tr}\left(u_{j}\vee (-\partial_{x_{j}}V\ast \rho) R_{\hbar}(t)\right)+2\int_{\mathbb{R}^{3}}\rho_{\hbar}(t,x)\,\nabla_{x} V\ast \rho_{\hbar}(t,x)\cdot u(t,x) \,\dd x \notag\\
&+2\int_{\mathbb{R}^{3}}\nabla_{x} V\ast \rho(t,x)\cdot \rho u(t,x)\ \dd x
-2\int_{\mathbb{R}^{3}}\nabla_{x} V\ast \rho_{\hbar}(t,x)\cdot \rho u(t,x) \ \dd x \notag\\
&=-2\int_{\mathbb{R}^{3}}\nabla_{x} V\ast \rho(t,x) \cdot u(t,x)\rho_{\hbar}(t,x) \ \dd x+2\int_{\mathbb{R}^{3}}\rho_{\hbar}(t,x)\,\nabla_{x} V\ast \rho_{\hbar}(t,x)\cdot u(t,x) \,\dd x \notag\\
&+2\int_{\mathbb{R}^{3}}\nabla
_{x}V\ast\rho \cdot \rho u(t,x)\ \dd x
-2\int_{\mathbb{R}^{3}}\nabla_{x} V\ast \rho_{\hbar}(t,x)\cdot \rho u(t,x) \ \dd x \notag\\
&=2\int_{\mathbb{R}^{3}}u(t,x)\cdot \nabla_{x} V\ast(\rho_{\hbar}-\rho)(t,x)(\rho_{\hbar}-\rho)(t,x)\ \dd x. \label{interaction part of time derivative}      
\end{align}
Substituting \eqref{vanishing}-\eqref{interaction part of time derivative} inside \eqref{sum of I1I4} we obtain 
\begin{align*}
\frac{\dd}{\dd t}\mathcal{\mathcal{E}}_{\hbar}(t)=&-\frac{1}{2}\sum_{1\leq j,k\leq 3}\mathrm{tr}\big((\Pi_{j}+u_{j})\vee((\Pi_{k}+u_{k})\vee(\partial_{x_{k}}u_{j}))R_{\hbar}(t)\big)\\
&+2\int_{\mathbb{R}^{3}}u(t,x)\cdot \nabla_{x} V\ast(\rho_{\hbar}-\rho)(t,x)(\rho_{\hbar}-\rho)(t,x)\ \dd x.     
\end{align*}
\end{proof}
We will now use Theorem \ref{thm time derivative of modulated energy} to conclude the proof of Theorem \ref{main thm semi classical}. For compactness of notation, we denote by $\mathcal{K}_{\hbar}(t)$ and $\mathcal{V}_{\hbar}(t)$ the kinetic and interaction parts of $\mathcal{E}_{\hbar}(t)$, i.e. 
\begin{align}
\mathcal{K}_{\hbar}(t)\coloneqq \mathrm{tr}\left((i\hbar\nabla_{x}+A+u  )^{2}R_{\hbar}(t)\right), \qquad \mathcal{V}_{\hbar}(t)\coloneqq\int_{\mathbb{R}^{3}}V\ast(\rho_{\hbar}-\rho) (\rho_{\hbar}-\rho)(t,x)\ \dd x.\label{kinetic and int parts def}     
\end{align}
\textbf{Proof of Theorem \ref{main thm semi classical}}. Our ultimate aim is to prove that $\underset{t\in [0,T]}{\sup}\mathcal{E}_{\hbar}(t)\underset{\hbar\rightarrow 0}{\rightarrow}0$.  
Set 
\begin{align*}
\mathcal{D}_{1}(t)\coloneqq -\frac{1}{2}\sum_{1\leq j,k\leq 3}\mathrm{tr}\big((\Pi_{j}+u_{j})\vee((\Pi_{k}+u_{k})\vee(\partial_{x_{k}}u_{j}))R_{\hbar}(t)\big)     
\end{align*} 
and 
\begin{align*}
\mathcal{D}_{2}(t)\coloneqq 2\int_{\mathbb{R}^{3}}u(t,x)\cdot \nabla_{x} V\ast(\rho_{\hbar}-\rho)(t,x)(\rho_{\hbar}-\rho)(t,x)\ \dd x.    
\end{align*}
According to Theorem \ref{thm time derivative of modulated energy} we have $\frac{\dd}{\dd t}\mathcal{E}_{\hbar}(t)=\mathcal{D}_{1}(t)+\mathcal{D}_{2}(t)$. We proceed by estimating separately $\mathcal{D}_{1}$ and $\mathcal{D}_{2}$.\\
\textbf{Estimate on $\mathcal{D}_{1}(t)$.} Set $P_{j}\coloneqq \Pi_{j}+u_{j}$. We have 
\begin{align}
P_{j}\vee (P_{k}\vee \partial_{x_{k}}u_{j})&=P_{j}\vee (P_{k}\partial_{x_{k}}u_{j}+\partial_{x_{k}}u_{j}P_{k}) \notag\\
&=P_{j}P_{k}\partial_{x_{k}}u_{j}+P_{j}\partial_{x_{k}}u_{j}P_{k}+P_{k}\partial_{x_{k}}u_{j}P_{j}+\partial_{x_{k}}u_{j}P_{k}P_{j}. \label{formula for anticommutator of Pj}  \end{align}
Therefore we have 
\begin{align*}
\mathcal{D}_{1}(t)=&-\frac{1}{2}\sum_{1\leq j,k\leq 3}\langle \psi\vert P_{j}P_{k}\partial_{x_{k}}u_{j}\vert \psi\rangle -\frac{1}{2}\sum_{1\leq j,k\leq 3}\langle \psi \vert  P_{j}(\partial_{x_{k}}u_{j}+\partial_{x_{j}}u_{k})P_{k} \vert \psi\rangle\\
&-\frac{1}{2}\sum_{1\leq j,k\leq 3}\langle \psi \vert  \partial_{x_{k}}u_{j}P_{k}P_{j}\vert  \psi\rangle \coloneqq D_{1}+D_{2}+D_{3}.        
\end{align*}
To estimate $D_{1}$, we apply the self-adjointness of $P_{j}$ to find
\begin{align*}
\sum_{k=1}^{3}\langle \psi\vert   P_{j}P_{k}\partial_{x_{k}}u_{j}\vert\psi\rangle  &=\sum_{k=1}^{3}\langle P_{j}\psi\vert   P_{k}\partial_{x_{k}}u_{j}\vert \psi \rangle\\
&=\sum_{k=1}^{3} (\langle P_{j}\psi  \vert  i\hbar \partial_{x_{k}x_{k}}u_{j}\vert \psi\rangle +\langle P_{j}\psi\vert  \partial_{x_{k}}u_{j}P_{k}\vert \psi\rangle)\\
&=\langle P_{j}\psi\vert  i\hbar \Delta_{x} u_{j}\vert \psi\rangle+\sum_{k=1}^{3}\langle P_{j}\psi  \vert \partial_{x_{k}}u_{j}P_{k}\vert \psi \rangle.       
\end{align*}
Observe that
\begin{align*}
\langle P_{j}\psi\vert  \partial_{x_{k}}u_{j}P_{k}\vert \psi\rangle \leq \frac{1}{2}\left\Vert D_{x}u\right\Vert_{\infty}\left(\left\Vert P_{j}\psi(t,\cdot)\right\Vert_{2}^{2} +\left\Vert P_{k}\psi(t,\cdot)\right\Vert_{2}^{2} \right).     
\end{align*}
Furthermore we have 
\begin{align*}
\langle P_{j}\psi \vert i\hbar \Delta_{x}u_{j}\vert \psi\rangle \leq \left\Vert P_{j}\psi(t,\cdot) \right\Vert_{2}^{2}+\hbar^{2}\left\Vert  \Delta_{x}u_{j}\psi(t,\cdot)\right\Vert_{2}^{2}\leq \mathcal{K}_{\hbar}(t)+\hbar^{2}\left\Vert \Delta_{x} u\right\Vert_{\infty}^{2}.     
\end{align*}
Therefore, we obtain 
\begin{align}
D_{1}\leq 3\left\Vert D_{x}u\right\Vert_{\infty}\mathcal{K}_{\hbar}(t)+9\mathcal{K}_{\hbar}(t)+9\hbar^{2}\left\Vert \Delta_{x}u\right\Vert_{\infty}^{2}= (3\left\Vert D_{x}u\right\Vert_{\infty} +9)\mathcal{K}_{\hbar}(t)+9\hbar^{2}\left\Vert \Delta_{x}u \right\Vert_{\infty}^{2}.  \label{Est D1}    
\end{align}
To estimate $D_{2}$, denote by $D^{\mathbf{s}}_{x}u$ the symmetrized Jacobian matrix of $u$, i.e. $D^{\mathbf{s}}_{x}u=\frac{1}{2}(D_{x}u+D_{x}^{T}u)$. Then, we have 
\begin{align*}
\langle \psi\vert   P_{j}(\partial_{x_{k}}u_{j}+\partial_{x_{j}}u_{k})P_{k}\vert \psi\rangle&= \langle P_{j}\psi\vert  (\partial_{x_{k}}u_{j}+\partial_{x_{j}}u_{k})P_{k}\vert\psi\rangle\\
&\leq \left\Vert D^{\mathbf{s}}u\right\Vert_{\infty} \left(\frac{1}{2}\left\Vert P_{j}\psi(t,\cdot)\right\Vert_{2}^{2}+\frac{1}{2}\left\Vert P_{k}\psi(t,\cdot)\right\Vert_{2}^{2}\right).      
\end{align*}
Therefore it follows that 
\begin{align}
D_{2}\leq 3\left\Vert D^{\mathbf{s}}u\right\Vert_{\infty}\mathcal{K}_{\hbar}(t). \label{Est D2}     
\end{align}
To estimate $D_{3}$ we note that 
\begin{align*}
&\sum_{k=1}^{3}\langle \psi \vert  \partial_{x_{k}}u_{j}P_{k}P_{j}\vert \psi\rangle=\sum_{k=1}^{3}\langle P_{k}\partial_{x_{k}}u_{j}\psi\vert  P_{j}\vert \psi \rangle\\
&=\sum_{k=1}^{3}\langle \partial_{x_{k}}u_{j}P_{k}\psi\vert  P_{j}\vert \psi \rangle  +\sum_{k=1}^{3}\langle i\hbar \partial_{x_{k}x_{k}}u_{j}\psi\vert  P_{j}\vert \psi\rangle \\
&= \sum_{k=1}^{3}\langle \partial_{x_{k}}u_{j}P_{k}\psi\vert P_{j}\vert \psi\rangle+\langle i\hbar \Delta_{x}u_{j}\psi \vert P_{j}\vert \psi\rangle\\
&\leq \frac{\left\Vert D_{x}u\right\Vert_{\infty}}{2}\sum_{k=1}^{3} (\left\Vert P_{k}\psi\right\Vert_{2}^{2} +\left\Vert P_{j}\psi(t,\cdot)\right\Vert_{2}^{2})+\frac{\hbar^{2}}{2}\left\Vert \Delta_{x}u_{j}\psi(t,\cdot)\right\Vert_{2}^{2}+\frac{1}{2}\left\Vert P_{j}\psi\right\Vert_{2}^{2}.    
\end{align*}
Therefore it follows that 
\begin{align}
D_{3}\leq 3(\left\Vert D_{x}u\right\Vert_{\infty}+\frac{1}{2})\mathcal{K}_{\hbar}(t)+\frac{9\hbar^{2}}{2}\left\Vert \Delta_{x}u\right\Vert_{\infty}^{2}. \label{EST D3}      
\end{align}
Gathering \eqref{Est D1}-\eqref{EST D3} we conclude that 
\begin{align}
\mathcal{D}_{1}(t)\leq c_{1}\mathcal{K}_{\hbar}(t)+c_{2}\hbar^{2},    \label{final est D1} 
\end{align}
where $c_{1}=c_{1}(\left\Vert D_{x}u\right\Vert_{\infty})$ and $c_{2}=c_{2}(\left\Vert \Delta_{x}u\right\Vert_{\infty})$. \par\medskip
\textbf{Estimate on $\mathcal{D}_{2}$. } First, we recast $\mathcal{D}_{2}(t)$ as 
\begin{align*}
&-2\int_{\mathbb{R}^{3}}u(t,x)\cdot \nabla_{x}V\ast (\rho_{\hbar}-\rho)(t,x)\mathrm{div}_{x}(\nabla_{x}V\ast (\rho_{\hbar}-\rho))(t,x)\ \dd x\\
&=2\int_{\mathbb{R}^{3}}\nabla_{x}(u(t,x)\cdot \nabla_{x}V\ast (\rho_{\hbar}-\rho))(t,x)\cdot \nabla_{x}V\ast (\rho_{\hbar}-\rho)(t,x)\ \dd x\\
&=2\int_{\mathbb{R}^{3}}D_{x}u(t,x)\nabla_{x}V\ast (\rho_{\hbar}-\rho)(t,x)\cdot \nabla_{x}V\ast (\rho_{\hbar}-\rho)(t,x)\ \dd x\\
&+2\int_{\mathbb{R}^{3}}u(t,x)D_{x}\nabla_{x}V\ast(\rho_{\hbar}-\rho)(t,x)\cdot \nabla_{x}V\ast (\rho_{\hbar}-\rho)(t,x)\ \dd x\coloneqq\mathcal{D}_{2}^{1}(t)+\mathcal{D}_{2}^{2}(t).   
\end{align*}
To estimate $\mathcal{D}_{2}^{1}(t)$, we use the identity 
$$\int_{\mathbb{R}^{3}}\left\vert \nabla_{x} V\ast \mu\right\vert^{2}(x)\ \dd x=\int_{\mathbb{R}^{3}}\mu(x) V\ast \mu(x)\ \dd x $$
in order to find that 
\begin{align}
\mathcal{D}_{2}^{1}(t)&\leq 2\left\Vert D_{x}u\right\Vert_{\infty}\int_{\mathbb{R}^{3}}\left\vert \nabla_{x}V\ast (\rho_{\hbar}-\rho)\right\vert^{2}(t,x) \ \dd x\notag\\
&=2\left\Vert D_{x}u\right\Vert_{\infty}\int_{\mathbb{R}^{3}}V\ast (\rho_{\hbar}-\rho)(t,x)(\rho_{\hbar}-\rho)(t,x)\ \dd x=2\left\Vert D_{x}u\right\Vert_{\infty} \mathcal{V}_{\hbar}(t). \label{Est on D21}       
\end{align}
In addition, we have 
\begin{align}
\mathcal{D}_{2}^{2}(t)&=\int_{\mathbb{R}^{3}}u(t,x)\cdot \nabla_{x}\left\vert \nabla_{x}V\ast (\rho_{\hbar}-\rho)\right\vert^{2}(t,x)\ \dd x \notag\\
&=-\int_{\mathbb{R}^{3}}\mathrm{div}_{x}u(t,x)\left\vert \nabla_{x}V\ast (\rho_{\hbar}-\rho)\right\vert^{2}(t,x)\ \dd x \notag\\
&\leq \left\Vert \mathrm{div}_{x}u\right\Vert_{\infty}\int_{\mathbb{R}^{3}}\left\vert \nabla_{x}V\ast(\rho_{\hbar}-\rho)\right\vert^{2}(t,x)\ \dd x \notag\\
&=\left\Vert \mathrm{div}_{x}u\right\Vert_{\infty}\int_{\mathbb{R}^{3}}V\ast (\rho_{\hbar}-\rho)(t,x)(\rho_{\hbar}-\rho)(t,x)\ \dd x. \label{D22 est}
\end{align}
In view of \eqref{Est on D21}-\eqref{D22 est} we conclude that 
\begin{align}
\mathcal{D}_{2}(t)\leq (2\left\Vert D_{x}u\right\Vert_{\infty}+\left\Vert \mathrm{div}_{x}u\right\Vert_{\infty})\mathcal{V}_{\hbar}(t). \label{final est on d2} 
\end{align}
By \eqref{final est D1} and \eqref{final est on d2} we conclude that there are constants $C_{1}=C_{1}(\left\Vert D_{x}u\right\Vert_{\infty} )>0$ and $C_{2}=C_{2}(\left\Vert \Delta_{x}u\right\Vert_{\infty})$ such that 
\begin{align*}
\frac{\dd}{\dd t}\mathcal{E}_{\hbar}(t)\leq C_{1}\mathcal{E}_{\hbar}(t)+C_{2}\hbar^{2}.    
\end{align*}
Thanks to Gr\"onwall's lemma we conclude that 
\begin{align*}
\underset{t\in [0,T]}{\sup}\mathcal{E}_{\hbar}(t)\leq e^{C_{1}T}\left(\mathcal{E}_{\hbar}(0)+C_{2}T\hbar^{2}\right) ,   
\end{align*}
which proves \eqref{vanishing of quantum modulated energy cocnlusion}. 
\qed
\begin{rem}
Once the convergence $\underset{t\in [0,T]}{\sup}\mathcal{E}_{\hbar}(t)\underset{\hbar \rightarrow 0}{\rightarrow} 0$ has been established, proving the convergence \eqref{convergence conclusion 1d} is standard, and is not affected by the inclusion of a magnetic field. Exactly the same argument presented in \cite{golse2022mean} or \cite{ben2026quantum} can be applied in the present settings in order to deduce \eqref{convergence conclusion 1d}. The same remark applies for the $N$-body problem, i.e. the essence of the proof is obtaining a Gr\"onwall estimate for $\mathcal{E}_{\hbar,N}(t)$. \label{remark how to deduce convergence}
\end{rem}
\begin{rem}
To conclude the semiclassical limit we needed $u$ to be Lipschitz with bounded Laplacian, which is weaker than the regularity provided by our well-posedness Theorem \ref{well posedness of EPA} and is also slightly weaker than the original regularity requirement in \cite{golse2022mean}, which demanded an extra bounded derivative. It is an interesting question whether it is possible to relax the regularity on $u$. Note that already if we assume that $u\in L^{\infty}_{t}W^{2,p}_{x}$ ($p<\infty$) then it is not obvious how to make sure that terms of the form $\hbar^{2}\left\Vert \Delta_{x}u\psi_{\hbar}\right\Vert_{2}$ are negligible in $\hbar$, since this would necessitate uniform in $\hbar$ bounds on the $L^{p}$ norm of $\left\vert \psi_{\hbar}\right\vert ^{2}$. In the case where $A=0$, these bounds are a consequence of the propagation of quantum moments proved in \cite{lafleche2019propagation}. However, when $A$ is an arbitrary Lipschitz vector-field it seems non-trivial to extend this propagation of quantum moments. Another reason that propagation of quantum moments in the magnetized settings is interesting is because we expect it to be a key ingredient in deriving the kinetic equation \eqref{magnetized vlasov Poisson} as a semiclassical limit from \eqref{SP equation intro}.      
\end{rem}
\section{The $N$-body problem}\label{Mean-field semiclassical limit sec}
This section deals with the joint mean-field semiclassical limit for magnetized Coulombic many body quantum dynamics. For the reader's convenience we recall the following functional inequalities. We mention that these inequalities are part of a series of deep developments in mean-field limits for singular flows, a full review of which is beyond the scope of the current work. We refer to \cite{serfaty2024lectures} for an exhaustive overview on the subject. For brevity, we introduce the notation 
\begin{align*}
\mathbf{V}(X^{N},\mu)\coloneqq  \int_{\Delta^{c}}V(x-y)(\mu_{X^{N}}-\mu)^{\otimes2}(\dd x\dd y).
\end{align*}
We also introduce the abbreviation for the kinetic and interaction parts of $\mathcal{E}_{\hbar,N}(t)$, denoted by $\mathcal{K}_{\hbar,N}(t)$ and $\mathcal{V}_{\hbar,N}(t)$ respectively, and defined by 
\begin{align*}
\mathcal{K}_{\hbar,N}(t)\coloneqq\frac{1}{N}\sum^N_{\ell=1}\mathrm{tr}\big((i\hbar\nabla_{x
^{\ell}}+A(x^{\ell})+u(t,x^{\ell})
)^{2}R_{\hbar,N}(t)\big)    
\end{align*}
and 
\begin{align*}
\mathcal{V}_{\hbar,N}(t)\coloneqq \int_{\mathbb{R}^{3N}}\left(\int_{\Delta^{c}}V(x-y)(\mu_{X^{N}}-\rho)^{\otimes 2}(\dd x\dd y)\right)\rho_{\hbar,N}(t,\dd X^{N}).     
\end{align*}
Note that we can recast $\mathcal{V}_{\hbar,N}(t)$ as
\begin{align}
\mathcal{V}_{\hbar,N}(t)=\int_{\mathbb{R}^{3N}}\mathbf{V}(X^{N},\rho(t,\cdot))\rho_{\hbar,N}(t,\dd X^{N}). \label{VhN formula}     
\end{align}
\begin{lem}[\cite{Serfaty2020MeanField}] \label{duerinckslemma}
 Let $\mu\in L^{\infty}(\mathbb{R}^{3})\cap \mathcal{P}(\mathbb{R}^{3})$ and  ${X}^{N}\in \Delta_{N}^{c}$. 
 Then

 \smallskip
 \begin{enumerate}
     \item [\rm (i)]  $\mathbf{V}(X^{N},\mu) +\frac{C(\left\Vert \mu\right\Vert_{\infty} )}{N^{\frac{2}{3}}}\geq 0${\rm ;}

     \smallskip
     \item [\rm (ii)] There are some $\lambda,C>0$ such that, for all $\varphi\in W^{1,\infty}(\mathbb{R}^{3})\cap \dot H^{-1}(\mathbb{R}^{3}),$ 
\[
\bigg\vert \int_{\mathbb{R}^{3}}\left(\mu_{X^{N}}-\mu\right)\varphi(x)\,\dd x \bigg\vert 
\leq C\left\Vert \varphi\right\Vert _{W^{1,\infty}}N^{-\lambda}+\left\Vert \nabla_{x}\varphi\right\Vert _{2}\Big(\mathbf{V}(X^{N},\mu)+\frac{C(\left\Vert \mu\right\Vert_{\infty} )}{N^{\frac{2}{3}}}\Big)^{\frac{1}{2}}.
\]
 \end{enumerate}
 \label{reminder of positivity}
\end{lem}
\begin{prop} [\cite{Serfaty2020MeanField}] \label{commutator estimate} 
Let the assumptions of {\rm Lemma \ref{duerinckslemma}} hold, and assume further that $u:\mathbb{R}^{3}\rightarrow \mathbb{R}^{3}$ is Lipschitz. Then,
it holds that
 \begin{align*}
\bigg\vert \int_{\Delta^{c}
}\left(u(x)-u(y)\right) \cdot\nabla_{x} V(x-y)\left(\mu_{X^{N}}-\mu\right)^{\otimes2}(\dd x\dd y)\bigg\vert 
\leq C\Big(\mathbf{V}(X^{N},\mu)+\frac{1+\left\Vert \mu \right\Vert_{\infty}}{N^{\beta}}\Big),
\end{align*}
where $C=C\left(\left\Vert u \right\Vert_{W^{1,\infty}},\left\Vert \mu\right\Vert_{\infty} \right)$ and $\beta>0$. 
\end{prop}
Note that point (i) in Lemma \ref{duerinckslemma} ensures that $\mathcal{E}_{\hbar,N}(t)\geq 0$. Recall that   $\mathscr{H}_{\hbar,A}^{N}$ is the $N$-body magnetized quantum Hamiltonian defined by 
\begin{align*}
\mathscr{H}_{\hbar,A}^{N}=\frac{1}{2}\sum_{\ell=1}^{N}(i\hbar\nabla_{x^{\ell}}+A(x^{\ell})+u(x^{\ell}))^{2}+\frac{1}{2N}\sum_{\ell\neq m}V(x^{\ell}-x^{m}).     
\end{align*}
Finally, the quantum $N$-body total energy is given by   
\begin{align*}
\mathcal{F}_{\hbar,N}(t)=\frac{1}{N}\sum_{\ell=1}^{N}\mathrm{tr}\left((i\hbar\nabla_{x^{\ell}}+A(x^{\ell}))^{2}R_{\hbar,N}(t)\right)+\frac{1}{N}\sum_{\ell\neq m} V(x_{\ell}-x_{m}).    
\end{align*}
By direct calculation, conservation of energy holds for the $N$-body system as well, i.e. 
\begin{align}
\frac{\dd}{\dd t}\mathcal{F}_{\hbar,N}(t)=0. \label{Conservation of N-body}    
\end{align}
The following theorem is the $N$-body analogue of Theorem \ref{thm time derivative of modulated energy}.   \begin{thm}
Let the assumption of Theorem \ref{N body mainthm} hold. Let $R_{\hbar,N}(t)$ be the unique solution to   \eqref{von Neumman magnetic intro} with initial data $R^{\mathrm{in}}_{\hbar,N}$ and   $(\rho,u)$ be the unique solution of \eqref{Magnetized Euler Poisson velocity density intro} with initial data $(\rho^{\mathrm{in}},u^{\mathrm{in}})$. Let $\mathcal{E}_{\hbar,N}(t)$ be given by \eqref{Def of renormalized modulated energy}. Then, for all $t\in [0,T]$ it holds that 
\begin{align*}
\frac{\dd}{\dd t}\mathcal{E}_{\hbar,N}(t)&=-\frac{1}{2}\sum_{1\leq j,k\leq 3}\mathrm{tr}\left((\Pi_{j}+u_{j})\vee ((\Pi_{k}+u_{k})\vee (\partial_{x_{k}}u_{j}))R_{\hbar,N:1}(t)\right) 
\\
&+\int_{\mathbb{R}^{3N}}\int_{\Delta^{c}}(u(t,x)-u(t,y))\cdot \nabla_{x} V(x-y)(\mu_{X^{N}}-\rho(t,\cdot))^{\otimes 2}(\dd x\dd y)\rho_{\hbar,N}(t,\dd X^{N}).    
\end{align*}
\label{time derivative of N body mod energy}
\end{thm}
\begin{proof}
\textbf{1.} We have 
\begin{align*}
\mathcal{E}_{\hbar,N}(t)&=\mathcal{F}_{\hbar,N}(t)+\frac{1}{N}\sum_{\ell=1}^{N}\mathrm{tr}\left((i\hbar\nabla_{x^{\ell}}+A(x^{\ell}))\vee u(t,x^{\ell})R_{\hbar,N}(t)\right)\\
&+\frac{1}{N}\sum_{\ell=1}^{N}\mathrm{tr}\left( \left\vert u\right\vert^{2} (t,x^{\ell})R_{\hbar,N}(t)\right)
-\frac{2}{N}\sum_{\ell=1}^{N}\int_{\mathbb{R}^{3N}}V\ast\rho(t,x^{\ell})\rho_{\hbar,N}(t,\dd X^{N})\\
&+\int_{\mathbb{R}^{3}}\rho(t,x) V\ast \rho(t,x)\ \dd x.
\end{align*}
Due to conservation of energy (equation \eqref{Conservation of N-body}) it follows that 
\begin{align*}
\frac{\dd}{\dd t}\mathcal{E}_{\hbar,N}(t)=&\frac{\dd}{\dd t}\frac{1}{N}\sum_{\ell=1}^{N}\mathrm{tr}\left((i\hbar\nabla_{x^{\ell}}+A(x^{\ell}))\vee u(t,x^{\ell})R_{\hbar,N}(t)\right)\\
&+\frac{\dd}{\dd t}\frac{1}{N}\sum_{\ell=1}^{N}\mathrm{tr}\left( \left\vert u\right\vert ^{2}(t,x^{\ell})R_{\hbar,N}(t)\right)\\
&-\frac{\dd}{\dd t}\frac{2}{N}\sum_{\ell=1}^{N}\int_{\mathbb{R}^{3N}}V\ast \rho(t,x^{\ell})\rho_{\hbar,N}(t,\dd X^{N})+\frac{\dd}{\dd t}\int_{\mathbb{R}^{3}}\rho V\ast \rho(t,x)\ \dd x\coloneqq
\sum_{k=1}^{4}I^{k}. 
\end{align*}
\textbf{2.} We start with calculating $I^{1}(t)+I^{2}(t)$. Setting $\Pi_{j}^{\ell}\coloneqq i\hbar\partial_{x^{\ell}_{j}}+A_{j}(x^{\ell})$ we may write
\begin{align*}
I^{1}(t)=\frac{\dd}{\dd t}\frac{1}{N}\sum_{\ell=1}^{N}\sum_{j=1}^{3}\mathrm{tr}\left(\Pi_{j}^{\ell}\vee u_{j}(t,x^{\ell})R_{\hbar,N}(t)\right)
\end{align*}
and 
\begin{align*}
I^{2}(t)=\frac{\dd}{\dd t}\frac{1}{N}\sum_{\ell=1}^{N}\int_{\mathbb{R}^{3}}\left\vert u\right\vert^{2} (t,x^{\ell})\rho_{\hbar,N}(t,\dd X^{N}).     
\end{align*}
For brevity, we omit time dependency in the forthcoming calculations. We claim to have the identity 
\begin{align}
I^{1}+I^{2} &=\frac{1}{N}\sum_{\ell=1}^{N}\sum_{j=1}^{3}\mathrm{tr}\big((u_{j}(x^{\ell})+\Pi_{j}^{\ell})\vee(\partial_{t}u_{j}(x^{\ell})+(uD_{x}u )_{j}(x^{\ell}))R_{\hbar,N}\big) \notag\\
&-\frac{1}{2N}\sum_{\ell=1}^{N}\sum_{1\leq j,k\leq 3}\mathrm{tr}\big((\Pi_{j}^{\ell}+u_{j}(x^{\ell}))\vee((\Pi_{k}^{\ell}+u_{k}(x^
{\ell}))\vee\partial_{x_{k}}u^{j}(x^{\ell}))R_{\hbar,N}\big) \notag\\
&+\frac{1}{2N}\sum_{\ell=1}^{N}\sum_{1\leq j,k\leq 3}\mathrm{tr}\left((\Pi_{k}^{\ell}\vee \mathbf{J}_{jk}(x^{\ell}))\vee u_{j}(x^{\ell})R_{\hbar,N}\right) \notag\\
&+\frac{2}{N^{2}}\sum_{m\neq \ell}\int_{\mathbb{R}^{3N}}\nabla_{x^{\ell}}V(x^{\ell}-x^{m})\cdot u(t,x^{\ell})\rho_{\hbar,N}(t,\dd X^{N}). \label{I1I2 Nbody formula}
\end{align}
We compute:   
\begin{align}
&\frac{\dd}{\dd t}\sum_{j=1}^{3}\mathrm{tr}\left(\Pi_{j}^{\ell}\vee u_{j}(x^{\ell})R_{\hbar,N}\right)+\frac{\dd}{\dd t}\mathrm{tr}\left(\left\vert u\right\vert^{2}(x^{\ell})R_{\hbar,N}\right) \notag\\
=&\mathrm{tr}\Big(\Big(\sum_{j=1}^{3}\Pi_{j}^{\ell}\vee u_{j}(x^{\ell})+\left\vert u\right\vert^{2}(x^{\ell}) \Big)\partial_{t}R_{\hbar,N}\Big) \notag
+\mathrm{tr}\Big(\partial_{t}\Big(\sum_{j=1}^{3}\Pi_{j}^{\ell}\vee u_{j}(x^{\ell})+\left\vert u\right\vert^{2}(x^{\ell})  \Big)R_{\hbar,N}\Big) \notag\\
=&\frac{1}{i\hbar}\mathrm{tr}\Big(\Big(\sum_{j=1}^{3}\Pi_{j}^{\ell}\vee u_{j}(x^{\ell})+\left\vert u\right\vert^{2}(x^{\ell}) \Big)\left[\mathscr{H}_{\hbar,A}^{N},R_{\hbar,N}\right]\Big) \notag
\\
&+\mathrm{tr}\Big(\partial_{t}\Big(\sum_{j=1}^{3}\Pi_{j}^{\ell}\vee u_{j}(x^{\ell})+\left\vert u\right\vert^{2}(x^{\ell})  \Big)R_{\hbar,N}\Big) \notag\\
=&\frac{i}{\hbar}\mathrm{tr}\Big(\Big[\mathscr{H}_{\hbar,A}^{N},\sum_{j=1}^{3}\Pi_{j}^{\ell}\vee u_{j}(x^{\ell})+\left\vert u\right\vert^{2}(x^{\ell}) \Big]R_{\hbar,N}\Big) 
+\mathrm{tr}\Big(\partial_{t}\Big(\sum_{j=1}^{3}\Pi_{j}^{\ell}\vee u_{j}(x^{\ell})+\left\vert u\right\vert^{2}(x^{\ell})  \Big)R_{\hbar,N}\Big). 
\label{rhsnbody} 
\end{align}
Recalling the definition  \eqref{def of Nkinetic part} and \eqref{def of Ninter part}, the right-hand side of \eqref{rhsnbody} reads  
\begin{align*}
&\sum_{j=1}^{3}\mathrm{tr}\left(\left(\partial_{t}+\frac{i}{\hbar}\left[\mathscr{K}_{\hbar}^{N},\cdot\right]\right)\left(\Pi_{j}^{\ell}\vee u_{j}(x^{\ell})+\frac{1}{2}u_{j}(x^{\ell})\vee u_{j}(x^{\ell})\right)R_{\hbar,N}\right
)\\
&+ \sum_{j=1}^{3}\mathrm{tr}\left(\frac{i}{\hbar}\left[\mathscr{V}^{N},\cdot\right]\left(\Pi_{j}^{\ell}\vee u_{j}(x^{\ell})+\frac{1}{2}u_{j}(x^{\ell})\vee u_{j}(x^{\ell}) \right)R_{\hbar,N}\right)\\
&\coloneqq \sum_{j=1}^{3}\mathrm{tr}\left(L_{1j}R_{\hbar,N}\right)+\sum_{j=1}^{3}\mathrm{tr}\left(L_{2j}R_{\hbar,N}\right). 
\end{align*}
We start with the term $L_{1j}$. Using the  Leibniz rule: $\left[A,B\vee C\right]=\left[A,B\right]\vee C+\left[A,C\right]\vee B,$ we can write $L_{1j}$ as follows:
\begin{align}
L_{1j}&=\left(\partial_{t}+\frac{i}{\hbar}\left[\mathscr{K}_{\hbar}^{N},\cdot\right]\right)\left((\frac{1}{2}u_{j}(x^{\ell})+\Pi_{j}^{\ell})\vee u_{j}(x^{\ell})\right) \notag \\
&=\left(\left(\partial_{t}+\frac{i}{\hbar}\left[\mathscr{K}_{\hbar}^{N},\cdot\right]\right)(\frac{1}{2}u_{j}(x^{\ell})+\Pi_{j}^{\ell})\right)\vee u_{j}(x^{\ell}) \notag\\
&+(\frac{1}{2}u_{j}(x^{\ell})+\Pi_{j}^{\ell})\vee \left(\left(\partial_{t}+\frac{i}{\hbar}\left[\mathscr{K}_{\hbar}^{N},\cdot\right]\right)u_{j}(x^{\ell})\right) \notag\\
&=\left(\left(\partial_{t}+\frac{i}{\hbar}\left[\mathscr{K}_{\hbar}^{N},\cdot\right]\right)u_{j}(x^{\ell})\right)\vee u_{j}(x^{\ell})+\Pi_{j}^{\ell}\vee \left(\left(\partial_{t}+\frac{i}{\hbar}\left[\mathscr{K}_{\hbar}^{N},\cdot\right]\right)u_{j}(x^{\ell})\right) \notag\\
&+\frac{i}{\hbar}\left[\mathscr{K}_{\hbar}^{N},\Pi_{j}^{\ell}\right]\vee u_{j}(x^{\ell}). \label{L1j nbody} 
\end{align}
Note $\Pi_{j}^{\ell},\Pi_{k}^{m}$ commute whenever  $m\neq \ell$ and therefore  \begin{align*}
\Big[\mathscr{K}_{\hbar}^{N},\Pi_{j}^{\ell}\Big]=\frac{1}{2}\sum_{k=1}^{3}\Big[(\Pi_{k}^{\ell})^{2},\Pi_{j}^{\ell}\Big].     
\end{align*}
Furthermore, we compute that 
\begin{align*}
\left[(\Pi_{k}^{\ell})^{2},\Pi_{j}^{\ell}\right]&=\Pi_{k}^{\ell}\vee\left[\Pi_{k}^{\ell},\Pi_{j}^{\ell}\right]=\Pi_{k}^{\ell}\vee \left[i\hbar\partial_{x_{k}^{\ell}}+A_{k}(x^{\ell}),i\hbar \partial_{x_{j}^{\ell}}+A_{j}(x^{\ell})\right]\\
&=\Pi_{k}^{\ell}\vee (i\hbar \partial_{x_{k}^{\ell}}A_{j}(x^{\ell})-i\hbar\partial_{x_{j}^{\ell}}A_{k}(x^{\ell})), 
\end{align*}
and hence we get   
\begin{align}
\frac{i}{\hbar}\left[\mathscr{K}_{\hbar}^{N},\Pi_{j}^{\ell}\right]=\frac{1}{2}\sum_{k=1}^{3}\Pi_{k}^{\ell}\vee(\partial_{x_{j}^{\ell}}A_{k}(x^{\ell})-\partial_{x_{k}^{\ell}}A_{j}(x^{\ell}))=\frac{1}{2}\sum_{k=1}^{3}\Pi_{k}^{\ell}\vee \mathbf{J}_{jk}(x^{\ell}).    \label{last term}     
\end{align}
In addition, note that $\left[\mathscr{K}^{N}_{\hbar},u_{j}(x^{\ell})\right]=\left[\frac{1}{2}\sum_{k=1}^{3}(\Pi_{k}^{\ell})^{2},u_{j}(x^{\ell})\right]$.
Thus, inserting \eqref{last term} inside \eqref{L1j nbody} we get 
\begin{align*}
L_{1j}&=(u_{j}(x^{\ell})+\Pi_{j}^{\ell})\vee \left(\left(\partial_{t}+\frac{i}{\hbar}\left[\frac{1}{2}\sum_{k=1}^{3}(\Pi_{k}^{\ell})^{2},\cdot\right]\right)u_{j}(x^{\ell})\right)\\
&+\frac{1}{2}\sum_{k=1}^{3}(\Pi_{k}^{\ell}\vee \mathbf{J}_{jk}(x^{\ell}))\vee u_{j}(x^{\ell})\\
&=(u_{j}(x^{\ell})+\Pi_{j}^{\ell})\vee \left(\partial_{t}u_{j}(x^{\ell})+\sum_{k=1}^{3}u_{k}(x^{\ell})\partial_{x_{k}}u_{j}(x^{\ell})\right)\\
&+(u_{j}(x^{\ell})
+\Pi_{j}^{\ell})\vee \left(\frac{i}{\hbar}\left[\frac{1}{2}\sum_{k=1}^{3}(\Pi_{k}^{\ell})^{2},u_{j}(x^{\ell})\right]-\sum_{k=1}^{3}u_{k}(x_{\ell})\partial_{x_{k}}u_{j}(x^{\ell})\right)\\
&+\frac{1}{2}\sum_{k=1}^{3}(\Pi_{k}^{\ell}\vee \mathbf{J}_{jk}(x^{\ell})) \vee u_{j}(x^{\ell}). 
\end{align*}
Therefore we conclude that
\begin{align}
\sum_{j=1}^{3}\mathrm{tr}\big(L_{1j}R_{\hbar,N}(t)\big)
&=\sum_{j=1}^{3}\mathrm{tr}\big((u_{j}(x^{\ell})+\Pi_{j}^{\ell})\vee(\partial_{t}u_{j}(x^{\ell})+(uD_{x}u )_{j}(x^{\ell}))R_{\hbar,N}\big) \notag \\
&-\frac{1}{2}\sum_{1\leq j,k\leq 3}\mathrm{tr}\big((\Pi_{j}^{\ell}+u_{j}(x^{\ell}))\vee((\Pi_{k}^{\ell}+u_{k}(x^
{\ell}))\vee\partial_{x_{k}}u^{j}(x^{\ell}))R_{\hbar,N}\big) \notag\\
&+\frac{1}{2}\mathrm{tr}\Big(\sum_{1\leq j,k\leq 3}(\Pi_{k}^{\ell}\vee \mathbf{J}_{jk}(x^{\ell}))\vee u_{j}(x^{\ell}) R_{\hbar,N}\Big).
\label{eq for J1}
\end{align}
As for $L_{2j}$, we see that
\begin{align*}
[\mathscr{V}^{N},u_{j}(x^{\ell})\vee u_{j}(x^{\ell})]=0,\ [\mathscr{V}^{N},A_{j}(x^{\ell})]=0
\end{align*}
and that 
\begin{align*}
-[\mathscr{V}^{N},\partial_{x_{j}^{\ell}}\vee u_{j}(x^{\ell})]=-2[\mathscr{V}^{N}, u_{j}(x^{\ell})\partial_{x_{j}^{\ell}}]=2\partial_{x_{j}^{\ell}}\mathscr{V}^{N}u_{j}(x^{\ell}).  \end{align*}
We compute that 
\begin{align*}
&2\partial_{x_{j}^{\ell}}\mathscr{V}^{N}=\frac{1}{N}\sum_{m'\neq  m}\partial_{x_{j}^{\ell}}(V(x^{m'}-x^{m}))\\
&=\frac{1}{N}\sum_{m'\neq m}\delta_{\ell m'}\partial_{x_{j}^{\ell}}V(x^{m'}-x^{m})-\delta_{\ell m}\partial_{x_{j}^{\ell}}V(x^{m'}-x^{m})\\
&=\frac{1}{N}\sum_{m:m\neq \ell }\partial_{x_{j}^{\ell}}V(x^{\ell}-x^{m})-\frac{1}{N}\sum_{m':m'\neq \ell}\partial_{x_{j}^{\ell}}V(x^{m'}-x^{\ell})=\frac{2}{N}\sum_{m:m\neq \ell}\partial_{x_{j}^{\ell}}V(x^{\ell}-x^{m}). 
\end{align*}
Consequently, it follows that 
\begin{align}
\sum_{j=1}^{3}\mathrm{tr}\big(L_{2,j}R_{\hbar,N}\big)&=\frac{2}{N}\sum_{m:m\neq \ell}\sum_{j=1}^{3}\int_{\mathbb{R}
^{3N}}\partial_{x_{j}^{\ell}}V(x^{\ell}-x^{m})u_{j}(x^{\ell})\rho_{\hbar,N}(t,\dd X^{N}) \notag\\
&=\frac{2}{N}\sum_{m:m\neq \ell}\int_{\mathbb{R}^{3N}}\nabla_{x^{\ell}}V(x^{\ell}-x^{m})\cdot u(t,x^{\ell})\rho_{\hbar,N}(t,\dd X^{N}) 
\notag
\end{align}
so that 
\begin{align}
\frac{2}{N}\sum_{\ell=1}^{N}\sum_{j=1}^{3}\mathrm{tr}\left(L_{2,j}R_{\hbar,N}\right)=\frac{2}{N^{2}}\sum_{m\neq \ell}\int_{\mathbb{R}^{3N}}\nabla_{x^{\ell}}V(x^{\ell}-x^{m})\cdot u(t,x^{\ell})\rho_{\hbar,N}(t,\dd X^{N}).     \label{eq for J2}
\end{align}
Combining \eqref{eq for J1} with \eqref{eq for J2}, we obtain \eqref{I1I2 Nbody formula}.\\
\textbf{3.} In view of Remark \ref{remark about eq qunatum density} we compute that 
\begin{align}
I^{3}(t)+I^{4}(t)=&\frac{\dd}{\dd t}\int_{\mathbb{R}^{3}}V\ast \rho(t,x)\rho(t,x)\ \dd x-2\frac{\dd}{\dd t}\int_{\mathbb{R}^{3}}V\ast \rho(t,x)\rho_{\hbar,N:1}(t,x)\ \dd x \notag\\
=&2\int_{\mathbb{R}^{3}}V\ast \rho(t,x)\partial_{t}\rho(t,x)\ \dd x-2\int_{\mathbb{R}^{3}}V\ast \rho_{\hbar,N:1}(t,x)\partial_{t}\rho(t,x)\ \dd x \notag\\
&-2\int_{\mathbb{R}^{3}}V\ast \partial_{t}\rho_{\hbar,N:1}(t,x)\rho(t,x)\ \dd x \notag\\
=& 2\int_{\mathbb{R}^{3}}\nabla_{x}V\ast \rho(t,x)\cdot \rho u(t,x)\ \dd x-2\int_{\mathbb{R}^{3}} \nabla_{x}V\ast \rho_{\hbar,N:1}(t,x)\cdot \rho u(t,x)\ \dd x \notag\\
&+2\int_{\mathbb{R}^{3}}V\ast \mathrm{div}_{x}(J_{\hbar,N:1})(t,x)\rho(t,x)\ \dd x. \label{I3I4 N body}
\end{align}
In view of \eqref{I1I2 Nbody formula} and \eqref{I3I4 N body} we conclude that 
\begin{align}
I^{1}+I^{2}+I^{3}+I^{4} =&-\frac{1}{N}\sum_{\ell=1}^{N}\sum_{j=1}^{3}\mathrm{tr}\big((u_{j}(x^{\ell})+\Pi_{j}^{\ell})\vee(\partial_{x_{j}^{\ell}}V\ast \rho(x^{\ell})+(u\mathbf{J})_{j}(x^{\ell}))R_{\hbar,N}\big) \notag\\
&-\frac{1}{2N}\sum_{\ell=1}^{N}\sum_{1\leq j,k\leq 3}\mathrm{tr}\big((\Pi_{j}^{\ell}+u_{j}(x^{\ell}))\vee((\Pi_{k}^{\ell}+u_{k}(x^
{\ell}))\vee\partial_{x_{k}}u^{j}(x^{\ell}))R_{\hbar,N}\big) \notag\\
&+\frac{1}{2N}\sum_{\ell=1}^{N}\sum_{1\leq j,k\leq 3}\mathrm{tr}\left((\Pi_{k}^{\ell}\vee \mathbf{J}_{jk}(x^{\ell}))\vee u_{j}(x^{\ell})R_{\hbar,N}\right) \notag\\
&+\frac{2}{N^{2}}\sum_{m\neq \ell}\int_{\mathbb{R}^{3N}}\nabla_{x^{\ell}}V(x^{\ell}-x^{m})\cdot u(t,x^{\ell})\rho_{\hbar,N}(t,\dd X^{N}) \notag\\
&+2\int_{\mathbb{R}^{3}}\nabla_{x}V\ast \rho(t,x)\cdot \rho u(t,x)\ \dd x-2\int_{\mathbb{R}^{3}} \nabla_{x}V\ast \rho_{\hbar,N:1}(t,x)\cdot \rho u(t,x)\ \dd x \notag\\
&+2\int_{\mathbb{R}^{3}}V\ast \mathrm{div}_{x}(J_{\hbar,N:1})(t,x)\rho(t,x)\ \dd x. \label{sum of I1I4 Nbody}
\end{align}
Since $\mathbf{J}$ is anti-symmetric we have $u\cdot (u\mathbf{J})=0$. Therefore, by Lemma \ref{cancellation lemma} for any fixed $\ell$ we have   
\begin{align*}
&-\sum_{j=1}^{3}(u_{j}(x^{\ell})+\Pi_{j}^{\ell})\vee (u\mathbf{J})_{j}(x^{\ell})+\frac{1}{2}\sum_{1\leq j,k\leq 3}(\Pi_{k}^{\ell}\vee\mathbf{J}_{jk}(x^{\ell}))\vee u_{j}(x^{\ell})\\
&=-\frac{1}{2}\sum_{1\leq j,k\leq 3}\Pi_{j}^{\ell}\vee (u_{k}(x^{\ell})\vee \mathbf{J}_{kj}(x^{\ell}))+\frac{1}{2}\sum_{1\leq j,k\leq 3}(\Pi_{k}^{\ell}\vee \mathbf{J}_{jk}(x^{\ell}))\vee u_{j}(x^{\ell})=0. 
\end{align*}
Hence, we get the following cancellation 
\begin{align}
&-\frac{1}{N}\sum_{\ell=1}^{N}\sum_{1\leq j,k\leq 3}\mathrm{tr}\left((u_{j}(x^{\ell})+\Pi_{j}^{\ell})\vee (u\mathbf{J})_{j}(x^{\ell})R_{\hbar,N}\right) \notag\\
&+\frac{1}{2N}\sum_{\ell=1}^{N}\sum_{1\leq j,k\leq 3}\mathrm{tr}\left((\Pi_{k}^{\ell}\vee \mathbf{J}_{jk}(x^{\ell}))\vee u_{j}(x^{\ell})R_{\hbar,N}\right)=0. \label{Cancelation Nbody}    
\end{align}
Plugging in \eqref{Cancelation Nbody} inside \eqref{sum of I1I4 Nbody} we deduce that 
\begin{align}
I^{1}+I^{2}+I^{3}+I^{4}=& -\frac{1}{N}\sum_{\ell=1}^{N}\sum_{j=1}^{3}\mathrm{tr}\big((u_{j}(x^{\ell})+\Pi_{j}^{\ell})\vee(\partial_{x_{j}^{\ell}}V\ast \rho(x^{\ell}))R_{\hbar,N}\big)  \notag\\
&-\frac{1}{2N}\sum_{\ell=1}^{N}\sum_{1\leq j,k\leq 3}\mathrm{tr}\big((\Pi_{j}^{\ell}+u_{j}(x^{\ell}))\vee((\Pi_{k}^{\ell}+u_{k}(x^
{\ell}))\vee\partial_{x_{k}}u^{j}(x^{\ell}))R_{\hbar,N}\big) \notag\\
&+\frac{2}{N^{2}}\sum_{m\neq \ell}\int_{\mathbb{R}^{3N}}\nabla_{x^{\ell}}V(x^{\ell}-x^{m})\cdot u(t,x^{\ell})\rho_{\hbar,N}(t,\dd X^{N}) \notag\\
&+2\int_{\mathbb{R}^{3}}\nabla_{x}V\ast \rho(t,x)\cdot \rho u(t,x)\ \dd x-2\int_{\mathbb{R}^{3}} \nabla_{x}V\ast \rho_{\hbar,N:1}(t,x)\cdot \rho u(t,x)\ \dd x \notag\\
&+2\int_{\mathbb{R}^{3}}V\ast \mathrm{div}_{x}(J_{\hbar,N:1})(t,x)\rho(t,x)\ \dd x. \label{I1I4 prefinal formula}
\end{align}
Integrating by parts we see that
\begin{align*}
2\int_{\mathbb{R}^{3}}V\ast\mathrm{div}_{x}(J_{\hbar,N:1})(t,x)\rho(t,x)\ \dd x&=-2\int_{\mathbb{R}^{3}}\nabla_{x} V\ast \rho(t,x)\cdot J_{\hbar,N:1}(t,x)\ \dd x\\
&=\sum_{j=1}^{3}\mathrm{tr}\left(\Pi_{j}\vee \partial_{x_{j}}V\ast \rho R_{\hbar,N:1}\right).     
\end{align*}
Therefore we have 
\begin{align}
&-\frac{1}{N}\sum_{\ell=1}^{N}\sum_{j=1}^{3}\mathrm{tr}\big(\Pi_{j}^{\ell}\vee(\partial_{x_{j}^{\ell}}V\ast \rho(x^{\ell}))R_{\hbar,N}\big) +2\int_{\mathbb{R}^{3}}V\ast \mathrm{div}_{x}(J_{\hbar,N:1})(t,x)\rho(t,x)\ \dd x \notag\\
&=-\sum_{j=1}^{3}\mathrm{tr}\left(\Pi_{j}\vee \partial_{x_{j}}V\ast \rho R_{\hbar,N:1}\right)+\sum_{j=1}^{3}\mathrm{tr}\left(\Pi_{j}\vee \partial_{x_{j}}   V\ast \rho R_{\hbar,N:1}\right)=0. 
\label{vanishing N body}   
\end{align}
In addition, observe  that 
\begin{align*}
-\frac{1}{N}\sum_{\ell=1}^{N}\sum_{j=1}^{3}\mathrm{tr}\left(u_{j}(x^{\ell})\vee \partial_{x_{j}^{\ell}}V\ast \rho(x^{\ell}) R_{\hbar,N}\right)=-2\int_{\mathbb{R}^{3}}\nabla_{x} V\ast \rho(t,x)\cdot u(t,x)\rho_{\hbar,N:1}(t,x)\ \dd x.    
\end{align*}
Consequently, we obtain  
\begin{align}
&-\frac{1}{N}\sum_{\ell=1}^{N}\sum_{j=1}^{3}\mathrm{tr}\left(u_{j}(x^{\ell})\vee \partial_{x_{j}^{\ell}}V\ast \rho(x^{\ell})R_{\hbar,N}\right)\notag\\
&+\frac{2}{N^{2}}\sum_{m\neq \ell}\int_{\mathbb{R}^{3N}}\nabla_{x^{\ell}}V(x^{\ell}-x^{m})\cdot u(t,x^{\ell})\rho_{\hbar,N}(t,\dd X^{N})\notag\\
&+2\int_{\mathbb{R}^{3}}\nabla_{x}V\ast \rho(t,x)\cdot \rho u(t,x)\ \dd x-2\int_{\mathbb{R}^{3}} \nabla_{x}V\ast \rho_{\hbar,N:1}(t,x)\cdot \rho u(t,x)\ \dd x \notag\\
&=-2\int_{\mathbb{R}^{3}}u(t,x)\cdot \nabla_{x}V\ast \rho(t,x)\rho_{\hbar,N:1}(t,x)\ \dd x \notag\\
&+\frac{2}{N^{2}}\sum_{m\neq \ell}\int_{\mathbb{R}^{3N}}\nabla_{x^{\ell}}V(x^{\ell}-x^{m})\cdot u(t,x^{\ell})\rho_{\hbar,N}(t,\dd X^{N})\notag\\
&+2\int_{\mathbb{R}^{3}}\nabla_{x}V\ast \rho(t,x)\cdot \rho u(t,x)\ \dd x-2\int_{\mathbb{R}^{3}} \nabla_{x}V\ast \rho_{\hbar,N:1}(t,x)\cdot \rho u(t,x)\ \dd x. \label{rhs prefinal NBODY}
\end{align}
Now, note the following identities. First, we have 
\begin{align}
&\frac{2}{N^{2}}\sum_{m\neq \ell}\int_{\mathbb{R}^{3N}}\nabla_{x^{\ell}}V(x^{\ell}-x^{m})\cdot u(t,x^{\ell})\rho_{\hbar,N}(t,\dd X^{N}) \notag\\
&= 2\int_{\mathbb{R}^
{3N}}\int_{\Delta^{c}}\nabla_{x}V(x-y)\cdot u(t,x)\mu_{X^{N}}^{\otimes 2}(\dd x\dd y)\rho_{\hbar,N}(t,\dd X^{N}). \label{pure empirical terms}  
\end{align}
Furthermore it holds that  
\begin{align}
&-2\int_{\mathbb{R}^{3}}u(t,x)\cdot \nabla_{x}V\ast \rho(t,x)\rho_{\hbar,N:1}(t,x)\ \dd x \notag\\
&=-2\int_{\mathbb{R}^{3N}}\int_{\mathbb{R}^{3}}u(t,x)\cdot \nabla_{x}V\ast \rho(t,x)\mu_{X^{N}}(\dd x)\rho_{\hbar,N}(t,\dd X^{N}) \notag\\
&=-2\int_{\mathbb{R}^{3N}}\int_{\mathbb{R}^{3}\times \mathbb{R}^{3}}u(t,x)\cdot\nabla_{x}V(x-y)\rho(t,y)\ \dd y\mu_{X^{N}}(\dd x)\rho_{\hbar,N}(t,\dd X^{N}).\label{identity for mixed terms}
\end{align}
In addition, we have that  
\begin{align}
-&2\int_{\mathbb{R}^{3}}\nabla_{x}V\ast \rho_{\hbar,N:1}(t,x)\cdot \rho u(t,x)\ \dd x=2\int_{\mathbb{R}^{3}} V\ast \mathrm{div}_{x}(\rho u)(t,x)\rho_{\hbar,N:1}(t,x)\ \dd x \notag\\
=&\int_{\mathbb{R}^{3N}}\int_{\mathbb{R}^{3}\times \mathbb{R}^{3}}V(x-y)\mathrm{div}_{x}(\rho u)(t,y)\dd y\mu_{X^{N}}(\dd x)\rho_{\hbar,N}(t,\dd X^{N}) \notag\\
=&2\int_{\mathbb{R}^{3N}}\int_{\mathbb{R}^{3}\times \mathbb{R}^{3}}\nabla_{x}V(x-y)\cdot \rho u(t,y)\ \dd y\mu_{X^{N}}(\dd x)\rho_{\hbar,N}(t,\dd X^{N}) \notag\\
=&-2\int_{\mathbb{R}^{3N}}\int_{\mathbb{R}^{3}\times \mathbb{R}^{3}}u(t,x)\cdot \nabla_{x}V(x-y)\rho(t,x)\mu_{X^{N}}(\dd y)\rho_{\hbar,N}(t,\dd X^{N}). \label{identity for mixed terms2}
\end{align}
Inserting \eqref{pure empirical terms}-\eqref{identity for mixed terms2} inside \eqref{rhs prefinal NBODY} we conclude that the right-hand side of \eqref{rhs prefinal NBODY} writes 
\begin{align}
&2\int_{\mathbb{R}^{3N}}\int_{\Delta^{c}}u(t,x)\cdot \nabla_{x}V(x-y)(\mu_{X^{N}}-\rho(t,\cdot))^{\otimes 2}(\dd x\dd y)\rho_{\hbar,N}(t,\dd X^{N}). 
\end{align}
Using that $\nabla_{x}V$ is odd we can symmetrize the last expression and find that it is equal to 
\begin{align}
\int_{\mathbb{R}^{3N}}\int_{\Delta^{c}}(u(t,x)-u(t,y))\cdot \nabla_{x}V(x-y)(\mu_{X^{N}}-\rho(t,\cdot))^{\otimes }(\dd x\dd y)\rho_{\hbar,N}(t,\dd X^{N}).  \label{formula int part Nbody}      
\end{align}
Finally, inserting \eqref{formula int part Nbody} inside \eqref{I1I4 prefinal formula} yields the asserted formula. 
\end{proof}
The proof of Theorem \ref{N body mainthm} now proceeds essentially as in \cite{golse2022mean}, but we include the details below for completeness.
\par\medskip 
\textbf{Proof of Theorem \ref{N body mainthm}.} 
Set 
\begin{align*}
\mathcal{D}_{1}(t)\coloneqq -\frac{1}{2}\sum_{1\leq j,k\leq 3} \mathrm{tr}\left((\Pi_{j}+u_{j})\vee ((\Pi_{k}+u_{k})\vee \partial_{x_{k}}u_{j})R_{\hbar,N:1}(t)\right)   
\end{align*}
and 
\begin{align*}
\mathcal{D}_{2}(t)\coloneqq\int_{\mathbb{R}^{3N}}\int_{\Delta^{c}}(u(t,x)-u(t,y))\cdot \nabla_{x} V(x-y)(\mu_{X^{N}}-\rho(t,\cdot))^{\otimes 2}(\dd x\dd y)\rho_{\hbar,N}(t,\dd X^{N}).     
\end{align*}
According to Theorem \ref{time derivative of N body mod energy} it holds that 
\begin{align*}
\frac{\dd}{\dd t}\mathcal{E}_{\hbar,N}(t)=\mathcal{D}_{1}(t)+\mathcal{D}_{2}(t).     
\end{align*}
The crucial  difference in comparison to the estimates presented in the proof of Theorem \ref{main thm semi classical} is reflected in the estimate for $\mathcal{D}_{2}(t)$, which necessitates the application of Proposition \ref{commutator estimate}.\\ 
\textbf{Step 1. Estimate on $\mathcal{D}_{1}(t)$.} Let $\{\psi_{\ell}\}_{\ell=1}^{\infty}$ be the eigenfunction decomposition of $R_{\hbar,N:1}$, i.e $\{\psi_{\ell}\}_{\ell=1}^{\infty}\subset C([0,T];\mathfrak{H})$ is a complete system such that: 
\begin{align*}
R_{\hbar,N:1}(t)=\sum_{\ell=1}^{\infty}\lambda_{\ell}\vert \psi_{\ell
}(t,\cdot)\rangle\vert \langle \psi_{\ell}(t,\cdot)\vert \qquad \mbox{with}\    \lambda_{\ell}\geq 0,\ \sum_{\ell=1}^{\infty}\lambda_{\ell}=1.  
\end{align*}
Recall the notation $P_{j}=\Pi_{j}+u_{j}$ and the formula \eqref{formula for anticommutator of Pj}:
\begin{align*}
P_{j}\vee (P_{k}\vee \partial_{x_{k}}u_{j})= P_{j}P_{k}\partial_{x_{k}}u_{j}+P_{j}\partial_{x_{k}}u_{j}P_{k}+P_{k}\partial_{x_{k}}u_{j}P_{j}+\partial_{x_{k}}u_{j}P_{k}P_{j}   
\end{align*}
Therefore, it follows that 
\begin{align}
\mathcal{D}_{1}(t)=&-\frac{1}{2}\sum_{\ell=1}^{\infty}\lambda_{\ell}\sum_{1\leq j,k\leq 3}\langle \psi_{\ell}\vert  P_{j}P_{k}\partial_{x_{k}}u_{j}\vert\psi_{\ell}\rangle \notag\\
&-\frac{1}{2}\sum_{\ell=1}^{\infty}\lambda_{\ell}\sum_{1\leq j,k\leq 3} \langle \psi_{\ell} \vert P_{j}(\partial_{x_{k}}u_{j}+\partial_{x_{j}}u_{k})P_{k}\vert \psi_{\ell}\rangle \notag\\
&-\frac{1}{2}\sum_{\ell=1}^{\infty}\lambda_{\ell}\sum_{1\leq j,k\leq 3}\langle \psi_{\ell} \vert \partial_{x_{k}}u_{j}P_{k}P_{j}\vert \psi_{\ell}\rangle.  \label{D1 estimate N body} 
\end{align}
Keep $\ell$ fixed. By exactly the same argument leading to \eqref{Est D1}-\eqref{EST D3} we obtain the following estimates:
\begin{align*}
&\sum_{1\leq j,k\leq 3}\langle \psi_{\ell}\vert  P_{j}P_{k}\partial_{x_{k}}u_{j}\vert \psi_{\ell}\rangle\leq  (3\left\Vert D_{x}u\right\Vert_{\infty} +9)\sum_{j=1}^{3}\langle \psi_{\ell}\vert  \Pi_{j}^{2}\vert \psi_{\ell}\rangle  +9\hbar^{2}\left\Vert \Delta_{x}u \right\Vert_{\infty}^{2},\\
&\sum_{1\leq j,k\leq 3}\langle \psi_{\ell}\vert P_{j}(\partial_{x_{k}}u_{j}+\partial_{x_{j}}u_{k})P_{k}\vert \psi_{\ell}\rangle\leq 3\left\Vert D^{\mathbf{s}}u\right\Vert_{\infty}\sum_{j=1}^{3}\langle \psi_{\ell}\vert  \Pi_{j}^{2}\vert \psi_{\ell}\rangle     
\end{align*}
and 
\begin{align*}
\sum_{1\leq j,k\leq 3}\langle \psi_{\ell}\vert  \partial_{x_{k}}u_{j}P_{k}P_{j}\vert \psi_{\ell}\rangle \leq 3(\left\Vert D_{x}u\right\Vert_{\infty}+\frac{1}{2} )\sum_{j=1}^{3}\langle \psi_{\ell}\vert \Pi_{j}^{2}\vert \psi_{\ell}\rangle    +\frac{9\hbar^{2}}{2}\left\Vert \Delta_{x}u\right\Vert_{\infty}^{2}.   
\end{align*}
Substituting the above estimates inside \eqref{D1 estimate N body} we conclude that 
\begin{align}
\mathcal{D}_{1}(t)&\leq c_{1}\sum_{\ell=1}^{\infty}\lambda_{\ell}\sum_{j=1}^{3}\langle \psi_{\ell}\vert \Pi_{j}^{2}\vert \psi_{\ell}\rangle+c_{2}\hbar^{2} \notag\\
&=c_{1}\mathrm{tr}\Big(\sum_{j=1}^{3}\Pi_{j}^{2}R_{\hbar,N:1}(t)\Big)+c_{2}\hbar^{2}=c_{1}\mathcal{K}_{\hbar,N}(t)+c_{2}\hbar^{2}. \label{final est on cal D1}       
\end{align}
for some constants $c_{1}=c_{1}(\left\Vert D_{x}u\right\Vert_{\infty})$ and $c_{2}=c_{2}(\left\Vert \Delta_{x}u\right\Vert_{\infty} )$.
\\
\textbf{Step 2. Estimate on $\mathcal{D}_{2}(t)$.} Applying  Proposition \ref{commutator estimate} with $\mu=\rho(t,\cdot)$ we obtain  
\begin{align*}
&\left\vert \int_{\Delta^{c}}(u(t,x)-u(t,y))\cdot \nabla_{x}V(x-y)(\mu_{X^{N}}-\rho(t,\cdot))^{\otimes 2}(\dd x\dd y)\right\vert\\
&\leq C\left(\mathbf{V}(X^{N},\rho(t,\cdot))+\frac{1+\left\Vert \rho\right\Vert_{\infty}}{N^{\beta}}\right). \end{align*}
Therefore we get 
\begin{align}
\mathcal{D}_{2}(t) &\leq \int_{\mathbb{R}^{3N}}\left\vert \int_{\Delta^{c}}(u(t,x)-u(t,y))\cdot \nabla_{x}V(x-y)(\mu_{X^{N}}-\rho(t,\cdot))^{\otimes 2}(\dd x\dd y)\right\vert\rho_{\hbar,N}(t,\dd X^{N}) \notag\\
&\leq C\int_{\mathbb{R}^{3N}} \mathbf{V}(X^{N},\rho(t,\cdot))\rho_{\hbar,N}(t,\dd X^{N})+C\int_{\mathbb{R}^{3N}}\left(\frac{1+\left\Vert \rho\right\Vert_{\infty}}{N^{\beta}}\right)\rho_{\hbar,N}(t,\dd X^{N}) \notag\\
&\leq C \mathcal{V}_{\hbar,N}(t)+C\left(\frac{1+\left\Vert \rho\right\Vert_{\infty}}{N^{\beta}}\right),  \label{D2 est N body}
\end{align}
where in the last inequality we used \eqref{VhN formula} and that $\int_{\mathbb{R}^{3N}}\rho_{\hbar,N}(t,\dd X^{N})=1$. Combining \eqref{final est on cal D1} and \eqref{D2 est N body} we deduce that 
\begin{align*}
\frac{\dd}{\dd t}\mathcal{E}_{\hbar,N}(t)\leq  C\left(\mathcal{E}_{\hbar,N}(t)+N^{-\beta}+\hbar^{2}\right)    
\end{align*}
so that by Gr\"onwall's lemma we infer 
\begin{align*}
\underset{t\in [0,T]}{\sup}\mathcal{E}_{\hbar,N}(t)\leq e^{CT}\left(\mathcal{E}_{\hbar,N}(0)+T(N^{-\beta}+\hbar^{2})\right)    
\end{align*}
which concludes the proof of \eqref{N body N body energy convergence}. For the convergence \eqref{N body final convergence} see Remark \eqref{remark how to deduce convergence}. 
\qed 
\section{Admissible initial data}
{\label{initial data sec}}
The purpose of this section is to construct initial data realizing the vanishing of the initial quantum modulated energy and re-normalized quantum modulated energy, i.e. initial data such that $\mathcal{E}_{\hbar}(0)\underset{\hbar\rightarrow 0}{\rightarrow}0$ and $\mathcal{E}_{\hbar,N}(0)\underset{\hbar+\frac{1}{N}\rightarrow 0}{\rightarrow}0$.  
\begin{prop}\label{addmisible data 1body}
Suppose that:
\begin{itemize}
\item $u^{\mathrm{in}}\in W^{1,\infty}(\mathbb{R}^{3})$.
\item $A\in \mathrm{Lip}(\mathbb{R}^{3})$. 
\item   $\rho^{\mathrm{in}}\in \mathcal{P}_{c}(\mathbb{R}^{3})\cap L^{\infty}(\mathbb{R}^{3})$
and 
$\nabla_{x}\sqrt{\rho^{\mathrm{in}}}\in L^{2}(\mathbb{R}^{3})$.  
\end{itemize}  
Then, there is a family of wave functions $\left\{\psi_{\hbar}\right\}_{\hbar>0}$ such that 
\begin{itemize}
    \item $\left\{\psi_{\hbar}\right\}_{\hbar>0}\subset L^{2}(\mathbb{R}^{3};\mathbb{C})$ and $\int_{\mathbb{R}^{3}}\left\vert \psi_{\hbar}\right\vert^{2}(x)\ \dd x=1 $. 
    \item It holds that  $\mathcal{E}_{\hbar}(0)\underset{\hbar \rightarrow 0}{\rightarrow} 0$. 
\end{itemize}
\end{prop}
\begin{proof}
\textbf{1. Construction of the family $\left\{\psi_{\hbar}\right\}_{\hbar>0}$. }
For brevity we set 
\begin{align*}
U\coloneqq   A+u^{\mathrm{in}}\ \mbox{and}\ \mu=\sqrt{\rho^{\mathrm{in}}}.   
\end{align*}
Suppose that the support of $\mu$ is contained in some square, i.e. $\mathrm{supp}(\mu)\subset [-R,R]^{3}$ for some $R>0$. Let $\{Q_{k,\hbar}\}_{k=1}^{N_{\hbar}}$ be a equipartition of $[-R,R]^{3}$ into cubes with center $x_{k,\hbar}$ and width $\varepsilon_{\hbar}$. For any $1\leq k\leq N_{\hbar}$ fix a cut-off function $\chi
_{k,\hbar}$ such that: 
\begin{itemize}
\item $ \chi_{k,\hbar}\in C^{\infty}_{0}(\mathbb{R}^{3})$ and $0\leq \chi_{k,\hbar}\leq 1$.
\item  $\mathrm{supp}(\chi_{k,\hbar})\Subset Q_{k,\hbar}$ is such that $\chi_{k,\hbar}\equiv 1$ on a subcube $Q_{k,\hbar}'\subset Q_{k,\hbar}$ centered at $x_{k,\hbar}$ with $\mathrm{dist}(\partial Q_{k,\hbar}',\partial Q_{k,\hbar})=\delta_{\hbar}$. Moreover $\left\Vert \nabla_{x}\chi_{\hbar,k}\right\Vert_{\infty}\leq 2\delta_{\hbar}^{-1} $. 
\item It holds that: 
$\varepsilon_{\hbar}\underset{\hbar\rightarrow 0}{\rightarrow}0$, $\delta_{\hbar}\underset{\hbar\rightarrow 0}{\rightarrow}0$  and $\varepsilon_{\hbar}^{-1}\delta_{\hbar}+\hbar^{2}\varepsilon_{\hbar}^{-3}\delta_{\hbar}^{-2}\underset{\hbar\rightarrow 0}{\rightarrow} 0$. 
\end{itemize}
Let 
\begin{align}
\Psi_{\hbar}(x)\coloneqq \mu
(x)\sum_{k=1}^{N_{\hbar}}\chi_{k,\hbar}(x)e^{\frac{i(x-x_{k,\hbar})\cdot U(x_{k,\hbar})}{\hbar}} \label{Psihbar sum}    
\end{align}
and set 
\begin{align}
\psi_{\hbar}(x)=\frac{\Psi_{\hbar}(x)}{\lambda_{\hbar}}\ \mbox{where}\ \lambda_{\hbar}\coloneqq \left\Vert \Psi_{\hbar}\right\Vert_{2}. \label{def of psihbar}      
\end{align}
Note the following:  
\begin{enumerate}
\item $\left\Vert \psi_{\hbar}\right\Vert_{2}=1$. 
\item Since the supports of the $\chi_{k,\hbar}$'s are pairwise disjoint we have $\chi_{k,\hbar}\chi_{j,\hbar}\equiv 0$ whenever $k\neq j$ and consequently   
\begin{align}
\left\vert \psi_{\hbar}\right\vert^{2} (x)=\frac{1}{\lambda_{\hbar}^{2}}\mu^{2}(x)\sum_{k=1}^{N_{\hbar}}\chi_{k,\hbar}^{2}(x). \label{psihbar square}    
\end{align}
\end{enumerate}
Recall that $\mathcal{K}_{\hbar},\mathcal{V}_{\hbar}$ are given by \eqref{kinetic and int parts def}. We proceed by showing that $\mathcal{V}_{\hbar}(0)\underset{\hbar \rightarrow 0}{\rightarrow} 0$ and $\mathcal{K}_{\hbar}(0)\underset{\hbar \rightarrow 0}{\rightarrow} 0$.\\
\textbf{2. Vanishing of $\mathcal{V}_{\hbar}(0)$ as $\hbar\rightarrow 0$.}
Set $\rho_{\hbar}\coloneqq \left\vert \psi_{\hbar}\right\vert^{2}$ and  recall that by passing to Fourier we have  
\begin{align*}
\int_{\mathbb{R}^{3}}V\ast (\rho_{\hbar}-\rho)(\rho_{\hbar}-\rho)(x)\ \dd x=\int_{\mathbb{R}^{3}}\frac{\left\vert \widehat{\rho_{\hbar}-\rho}\right\vert^{2}(\xi)}{\left\vert \xi\right\vert^{2} }\ \dd \xi=\left\Vert \rho_{\hbar}-\rho\right\Vert_{\dot H^{-1}}^{2}.      
\end{align*}
We proceed by showing that $\left\Vert \rho_{\hbar}-\rho\right\Vert_{\dot{H}^{-1}}\underset{\hbar \rightarrow 0}{\rightarrow}0$. In fact we will show the stronger convergence $\left\Vert \rho_{\hbar}-\rho\right\Vert_{L^{\frac{6}{5}}}\underset{\hbar\rightarrow 0}{\rightarrow} 0$: due to the embedding  $L^{\frac{6}{5}}(\mathbb{R}^{3})\hookrightarrow \dot H^{-1}(\mathbb{R}^{3})$ this convergence would imply $\dot H^{-1}(\mathbb{R}^{3})$ convergence. First, thanks to \eqref{psihbar square} we have 
\begin{align*}
\lambda_{\hbar}^{2}=\sum_{k=1}^{N_{\hbar}}\int_{\mathbb{R}^{3}}\chi_{k,\hbar}^{2}(x)\mu^{2}(x)\ \dd x.    
\end{align*}
Note that $\sum_{k=1}^{N_{\hbar}}\chi_{k,\hbar}^{2}=\sum_{k=1}^{N_{\hbar}}\chi_{k,\hbar}\equiv 1$ on $\mathcal{Q}_{\hbar}\coloneqq\bigsqcup_{k=1}^{N_{\hbar}}Q_{k,\hbar}'$. Using that $N_{\hbar}\sim \varepsilon_{\hbar}^{-3}$ we can estimate the Lebesgue measure of the complement of $\mathcal{Q}_{\hbar}$ in $[-R,R]^{3}$ as follows:  
\begin{align} \left\vert [-R,R]^{3}\setminus \mathcal{Q}_{\hbar}\right\vert \leq CN_{\hbar}\varepsilon_{\hbar}^{2}\delta_{\hbar}\leq C\varepsilon_{\hbar}^{-1}\delta_{\hbar}\underset{\hbar \rightarrow 0}{\rightarrow}0.     \label{measure of complement} 
\end{align}
Therefore it follows that 
\begin{align*}
\lambda_{\hbar}^{2}&=\int_{\mathcal{Q}_{\hbar}}\sum_{k=1}^{N_\hbar}\chi_{k,\hbar}^{2}(x)\mu^{2}(x)\ \dd x+\int_{[-R,R]^{3}\setminus\mathcal{Q}_{\hbar}}\sum_{k=1}^{N_{\hbar}}\chi_{k,\hbar}^{2}(x)\mu^{2}(x)\ \dd x\\
&=\int_{\mathcal{Q}_{\hbar}}\mu^{2}(x)\ \dd x+\int_{[-R,R]^{3}\setminus \mathcal{Q}_{\hbar}}\sum_{k=1}^{N_{\hbar}}\chi_{k,\hbar}^{2}(x)\mu^{2}(x)\ \dd x.     
\end{align*}
In view of \eqref{measure of complement} we have 
\begin{align*}
\int_{\mathcal{Q}_{\hbar}} \mu^{2}(x)\ \dd x=\int_{[-R,R]^{3}}\mu^{2}(x)\ \dd x-\int_{[-R,R]^{3}\setminus \mathcal{Q}_{\hbar}}\mu^{2}(x)\ \dd x\underset{\hbar\rightarrow 0}{\rightarrow} 1  
\end{align*}
and 
\begin{align*}
\int_{[-R,R]^{3}\setminus \mathcal{Q}_{\hbar}}\sum_{k=1}^{N_{\hbar}}\chi_{k,\hbar}^{2}(x)\mu^{2}(x)\ \dd x\leq C\left\vert [-R,R]^{3}\setminus \mathcal{Q}_{\hbar}
\right\vert\underset{\hbar\rightarrow 0}{\rightarrow}0.
\end{align*}
This proves that 
\begin{align}
\lambda_{\hbar}\underset{\hbar\rightarrow 0}{\rightarrow}1. \label{lambdah goes to 0}    
\end{align}
Furthermore, we have 
\begin{align}
&\int_{\mathbb{R}^{3}}\left\vert \left\vert \psi_{\hbar}\right\vert^{2}-\mu^{2}\right\vert^{\frac{6}{5}}(x)\ \dd x=\int_{\mathbb{R}^{3}}\left\vert \mu^{2}\left(\frac{1}{\lambda_{\hbar}^{2}}\sum_{k=1}^{N_{\hbar}}\chi_{k,\hbar}^{2}-1\right)\right\vert^{\frac{6}{5}}(x)\ \dd x \notag\\
&=\int_{\mathcal{Q}_{\hbar}}\left\vert \mu^{2}\left(\frac{1}{\lambda_{\hbar}^{2}}\sum_{k=1}^{N_{\hbar}}\chi_{k,\hbar}^{2}-1\right)\right\vert^{\frac{6}{5}}(x)\ \dd x+\int_{[-R,R]^{3}\setminus \mathcal{Q}_{\hbar}}\left\vert \mu^{2}\left(\frac{1}{\lambda_{\hbar}^{2}}\sum_{k=1}^{N_{\hbar}}\chi_{k,\hbar}^{2}-1\right)\right\vert^{\frac{6}{5}}(x)\ \dd x \notag\\
&=\int_{\mathcal{Q}_{\hbar}}\left\vert \mu^{2}\left(\frac{1}{\lambda_{\hbar}^{2}}-1\right)\right\vert^{\frac{6}{5}}(x)\ \dd x+\int_{[-R,R]^{3}\setminus \mathcal{Q}_{\hbar}}\left\vert \mu^{2}\left(\frac{1}{\lambda_{\hbar}^{2}}\sum_{k=1}^{N_{\hbar}}\chi_{k,\hbar}^{2}-1\right)\right\vert^{\frac{6}{5}}(x)\ \dd x. \label{rhs of L65 norm}
\end{align}
Since $\lambda_{\hbar}\underset{\hbar\rightarrow 0} {\rightarrow} 1$ (by \eqref{lambdah goes to 0}), we conclude from the dominated convergence theorem that the first integral in \eqref{rhs of L65 norm} converges to $0$ as $\hbar\rightarrow 0$ and the second integral in \eqref{rhs of L65 norm} is bounded by $C\left\vert [-R,R]^{3}\setminus \mathcal{Q}_{\hbar}\right\vert \underset{\hbar \rightarrow 0}{\rightarrow} 0$. 
This concludes the proof that $\left\Vert \left\vert \psi_{\hbar}\right\vert^{2}-\mu^{2} \right\Vert_{\frac{6}{5}}\underset{\hbar\rightarrow 0}{\rightarrow}0$.\\
\textbf{3. Vanishing of  $\mathcal{K}_{\hbar}(0)$ as $\hbar \rightarrow 0$.}
We compute that 
\begin{align*}
&i\hbar \nabla_{x}(\mu(x)\chi_{k,\hbar}(x)e^{\frac{i(x-x_{k,\hbar})\cdot U(x_{k,\hbar})}{\hbar}})\\
&=i\hbar \nabla_{x}(\mu(x)\chi_{k,\hbar}(x))e^{\frac{i(x-x_{k,\hbar})\cdot U(x_{k,\hbar})}{\hbar}}-\mu(x)\chi_{k,\hbar}(x)U(x_{k,\hbar})e^{\frac{i(x-x_{k,\hbar})\cdot U(x_{k,\hbar})}{\hbar}}\\
&=i\hbar \chi_{k,\hbar}(x)\nabla_{x}\mu(x)e^{\frac{i(x-x_{k,\hbar})\cdot U(x_{k,\hbar})}{\hbar}}+i\hbar\mu(x)\nabla_{x}\chi_{k,\hbar}(x)e^{\frac{i(x-x_{k,\hbar})\cdot U(x_{k,\hbar})}{\hbar}}\\
&-\mu(x)\chi_{k,\hbar}(x)U(x_{k,\hbar})e^{\frac{i(x-x_{k,\hbar})\cdot U(x_{k,\hbar})}{\hbar}}. 
\end{align*}
Therefore we find that 
\begin{align}
&(i\hbar\nabla_{x}+U(x))(\mu(x)\chi_{k,\hbar}(x)e^{\frac{i(x-x_{k,\hbar})\cdot U(x_{k,\hbar})}{\hbar}}) \notag\\
&=i\hbar \chi_{k,x}(x)\nabla_{x}\mu(x)e^{\frac{i(x-x_{k,\hbar})\cdot U(x_{k,\hbar})}{\hbar}}+i\hbar\mu(x)\nabla_{x}\chi_{k,\hbar}(x)e^{\frac{i(x-x_{k,\hbar})\cdot U(x_{k,\hbar})}{\hbar}} \notag\\
&+\mu(x)\chi_{k,\hbar}(x)(U(x)-U
(x_{k,\hbar}))e^{\frac{i(x-x_{k,\hbar})\cdot U(x_{k,\hbar})}{\hbar}}. \label{Formula of magnetic graident of psihbar}    
\end{align}
In view of \eqref{Formula of magnetic graident of psihbar} we obtain   
\begin{align*}
&\int_{\mathbb{R}^{3}}\left\vert (i\hbar\nabla_{x}+U(x))\psi_{\hbar}\right\vert^{2}(x) \ \dd x=\frac{1}{\lambda_{\hbar}^{2}}\int_{\mathbb{R}^{3}}\left\vert (i\hbar\nabla_{x}+U(x))\Psi_{\hbar}\right\vert^{2}(x)\ \dd x\\
&\leq \frac{4}{\lambda_{\hbar}^{2}}\int_{\mathbb{R}^{3}}\left\vert \sum_{k=1}^{N_{\hbar}}i\hbar\chi_{k,\hbar}(x)\nabla_{x}\mu(x)e^{\frac{i(x-x_{k,\hbar})\cdot U(x_{k,\hbar})}{\hbar}}\right\vert^{2}\ \dd x
\\
&+\frac{4}{\lambda_{\hbar}^{2}}\int_{\mathbb{R}^
{3}}\left\vert \sum_{k=1}^{N_{\hbar}}i\hbar \mu(x)\nabla_{x}\chi_{k,\hbar}(x)e^{\frac{i(x-x_{k,\hbar})\cdot U(x_{k,\hbar})}{\hbar}}\right\vert^{2} \ \dd x\\
&+\frac{4}{\lambda_{\hbar}^{2}}\int_{\mathbb{R}^{3}}\left\vert \sum_{k=1}^{N_{\hbar}}\mu(x)\chi_{k,\hbar}(x)(U(x)-U(x_{k,\hbar}))e^{i(x-x_{k,\hbar})\cdot U(x_{k,\hbar})}\right\vert^{2}\ \dd x\coloneqq \mathcal{T}_{1,\hbar}+\mathcal{T}_{2,\hbar}+\mathcal{T}_{3,\hbar}.     
\end{align*}
We estimate $\mathcal{T}_{1,\hbar},\mathcal{T}_{2,\hbar},\mathcal{T}_{3,\hbar}$ as follows. First, we have 
\begin{align*}
\left\vert \sum_{k=1}^{N_{\hbar}}i\hbar\chi_{k,\hbar}(x)\nabla_{x}\mu(x)e^{\frac{i(x-x_{k,\hbar})\cdot U(x_{k,\hbar})}{\hbar}}\right\vert^{2}=\hbar^{2}\sum_{k=1}^{N_{\hbar}}\chi_{k,\hbar}^{2}(x)\left\vert \nabla_{x}\mu\right\vert^{2}(x)\leq \hbar^{2}N_{\hbar}\left\vert \nabla_{x}\mu\right\vert^{2}(x).      
\end{align*}
Recalling that  $N_{\hbar}$  satisfies  $N_{\hbar}\leq C \varepsilon_{\hbar}^{-3}$ we infer
\begin{align*}
\mathcal{T}_{1,\hbar}\leq \frac{4N_{\hbar}\hbar^{2}}{\lambda_{\hbar}^{2}}\int_{\mathbb{R}^{3}}\left\vert \nabla_{x}\mu\right\vert^{2}(x)\ \dd x\leq \frac{4C\varepsilon_{\hbar}^{-3}\hbar^{2}}{\lambda_{\hbar}^{2}}\left\Vert \nabla_{x}\mu\right\Vert_{2}^{2}\underset{\hbar \rightarrow 0}{\rightarrow}0     \end{align*}
where we used that $\lambda_{\hbar}\underset{\hbar \rightarrow 0}{\rightarrow} 1$ and the assumption that $\varepsilon_{\hbar}^{-3}\hbar^{2}\underset{\hbar \rightarrow 0}{\rightarrow} 0$.
As for $\mathcal{T}_{2,\hbar}$, we have  
\begin{align*}
\mathcal{T}_{2,\hbar}= \frac{4\hbar^{2}}{\lambda_{\hbar}^{2}}\int_{\mathbb{R}^{3}}\sum_{k=1}^{N_{\hbar}}\left\vert \nabla_{x}\chi_{k,\hbar}\right\vert^{2}(x)\mu^{2}(x)\ \dd x\leq \frac{4\hbar^{2}N_{\hbar}\delta_{\hbar}^{-2}}{\lambda_{\hbar}^{2}}\left\Vert \mu\right\Vert_{2}^{2}\leq \frac{4C\hbar^{2}\varepsilon_{\hbar}^{-3}\delta_{\hbar}^{-2}}{\lambda_{\hbar}^{2}}\underset{\hbar \rightarrow 0}{\rightarrow}0      
\end{align*}
where we used the assumption that $\hbar^{2}\varepsilon_{\hbar}^{-3}\delta_{\hbar}^{-2}\underset{\hbar\rightarrow 0}{\rightarrow} 0$. 
Finally, to estimate $\mathcal{T}_{3,\hbar}$ we observe that 
\begin{align*}
\mathcal{T}_{3,\hbar}=\frac{4}{\lambda_{\hbar}^{2}}\int_{\mathbb{R}^{3}}\sum_{k=1}^{N_{\hbar}}\chi_{k,\hbar}^{2}(x)\mu^{2}(x)\left\vert U(x)-U(x_{k,\hbar})\right\vert^{2}\ \dd x.   \end{align*}
Hence, using that $\left\vert U(x)-U(x_{k,\hbar})\right\vert\leq C\varepsilon_{\hbar}$ for all $x\in \mathcal{Q}_{\hbar}$ we obtain 
\begin{align*}
\mathcal{T}_{3,\hbar}\leq \frac{4C\varepsilon_{\hbar}^{2}}{\lambda_{\hbar}^{2}}\int_{\mathbb{R}^{3}}\sum_{k=1}^{N_{\hbar}}\chi_{k,\hbar}^{2}(x)\mu^{2}(x)\ \dd x= 4C\varepsilon_{\hbar}^{2}\underset{\hbar \rightarrow 0}{\rightarrow}0.    
\end{align*}
\end{proof}
Having constructed initial data such that $\mathcal{E}_{\hbar}(0)\underset{\hbar \rightarrow 0}{\rightarrow} 0$ it is relatively straightforward to extend the construction to the $N$-body problem. 
\begin{prop}
Let the assumption of Proposition \ref{addmisible data 1body} hold. Then, there exist a family of symmetric density operators $\{R_{\hbar,N}\}_{\hbar>0,N\in \mathbb{N}}\subset \mathcal{D}_{s}(\mathfrak{H}^{\otimes N})$ such that $\mathcal{E}_{\hbar,N}(0)\underset{\hbar+\frac{1}{N}\rightarrow 0}{\rightarrow}0$.
\end{prop}
\begin{proof}
Let $\psi_{\hbar}$ be defined as in \eqref{def of psihbar}. Set $R_{\hbar}\coloneqq \vert \psi_{\hbar}\rangle\vert \langle \psi_{\hbar}\vert $ and  $R_{\hbar,N}\coloneqq R_{\hbar}^{\otimes N}$. Clearly $R_{\hbar,N}\in \mathcal{D}_{s}(\mathfrak{H}^{\otimes N})$ and moreover the density $\rho_{\hbar,N}$ of $R_{\hbar,N}$ is given by   
$\rho_{\hbar,N}=(\left\vert \psi_{\hbar}\right\vert^{2})^{\otimes N}=\rho_{\hbar}^{\otimes N}$. 
\\
\textbf{1.} 
We show that $\mathcal{V}_{\hbar,N}(0)\underset{\hbar+\frac{1}{N}\rightarrow 0}{\rightarrow} 0$. By Lemma 3.5 in \cite{golse2022mean} there holds the following formula 
\begin{align*}
&\int_{\mathbb{R}^{3N}} \int_{\Delta^{c}}V(x-y)(\mu_{X^{N}}-\rho^{\mathrm{in}})^{\otimes 2}(\dd x\dd y)\rho_{\hbar,N}(\dd X^{N})\\
&=\int_{\mathbb{R}^{3}\times \mathbb{R}^{3}}V(x-y)\left(\frac{N-1}{N}\rho_{\hbar}(x)\rho_{\hbar}(y)+\rho^{\mathrm{in}}(x)\rho^{\mathrm{in}}(y)-2\rho_{\hbar}(x)\rho^{\mathrm{in}}(y)\right)\ \dd x\dd y\\
&=\int_{\mathbb{R}^{3}}V\ast (\rho_{\hbar}-\rho^{\mathrm{in}})(\rho_{\hbar}-\rho^{\mathrm{in}})(x)\ \dd x-\frac{1}{N}\int_{\mathbb{R}^{3}}V\ast \rho_{\hbar}(x)\rho_{\hbar}(x)\ \dd x. 
\end{align*}
By step 1 in the proof of Proposition \ref{addmisible data 1body} we have 
\begin{align}
\int_{\mathbb{R}^{3}}V\ast(\rho_{\hbar}-\rho^{\mathrm{in}})(\rho_{\hbar}-\rho^{\mathrm{in}})(x)\ \dd x\underset{\hbar\rightarrow 0}{\rightarrow}0.  \label{vanishing of ususal part}  
\end{align}
Moreover, a particular consequence of step 1 in the proof of Proposition \ref{addmisible data 1body} is that $\left\Vert \rho_{\hbar}\right\Vert_{\dot{H}^{-1}}=O_{\hbar}(1)$. Therefore we see that  
\begin{align*}
\int_{\mathbb{R}^{3}}V\ast\rho_{\hbar}(x)\rho_{\hbar}(x)\ \dd x=\left\Vert \rho_{\hbar}\right\Vert_{\dot{H}^{-1}}^{2}=O_{\hbar}(1),    
\end{align*}
and hence 
\begin{align}
\frac{1}{N}\int_{\mathbb{R}^{3}}V\ast \rho_{\hbar}(x)\rho_{\hbar}(x)\ \dd x\underset{N\rightarrow \infty}{\rightarrow} 0. \label{Vanishing of reminder}    
\end{align}
In view of \eqref{vanishing of ususal part} and \eqref{Vanishing of reminder} we conclude that $\mathcal{V}_{\hbar,N}(0)\underset{\hbar+\frac{1}{N}\rightarrow 0}{\rightarrow} 0$. 
\\
\textbf{2.}  We prove that $\mathcal{K}_{\hbar,N}(0)\underset{\hbar+\frac{1}{N}\rightarrow 0}{\rightarrow} 0$. 
\begin{align}
\mathcal{K}_{\hbar,N}(0)&=\frac{1}{N}\sum_{\ell=1}^{N}\mathrm{tr}\left((i\hbar\nabla_{x^{\ell}}+A(x^{\ell})+u^{\mathrm{in}}(x^{\ell}))^{2}R_{\hbar,N}\right) \notag\\
&=\mathrm{tr}\left((i\hbar\nabla_{x}+A+u^{\mathrm{in}})^{2}R_{\hbar}\right)=\int_{\mathbb{R}^{3}}\left\vert (i \hbar\nabla_{x}+A+u^{\mathrm{in}})\psi_{\hbar}\right\vert^{2}(x)\ \dd x. \label{recasting Nkinetic energy as 1kinetic energy}     
\end{align}
By step 2 in the proof of Proposition \ref{addmisible data 1body} the right-hand side of \eqref{recasting Nkinetic energy as 1kinetic energy} goes to $0$ as $\hbar \rightarrow 0$.  
\end{proof}
\begin{rem}
Note that the construction of admissible initial data demonstrated in \cite{ben2026quantum} is not directly applicable in the present settings. Indeed, in \cite{ben2026quantum} it is assumed that the Laplacian is perturbed by a gradient, whereas this assumption is incompatible with the problem at stake, since any gradient is curl free and therefore the magnetic field will be $0$ if we make this extra demand.      
\end{rem}
\noindent{\bf Acknowledgments.}
The author is indebted to the University of Basel for financial support.
\bibliographystyle{abbrv}
\bibliography{references}
\end{document}